\documentclass[a4paper, 11pt, pdftex]{amsart}

\usepackage{amsmath}
\usepackage{amssymb}
\usepackage{amsthm}
\usepackage[abbrev,alphabetic]{amsrefs}
\usepackage{amscd}
\usepackage{graphicx}
\usepackage{graphics}
\usepackage{ulem}
\usepackage{url}
\usepackage{ascmac}
\usepackage{multirow}

\newcommand{\doi}[1]{\url{https://doi.org/#1}}

\usepackage[top=30truemm,bottom=25truemm,left=35truemm,right=35truemm]{geometry}

\usepackage[all,cmtip]{xy}

\usepackage{color}
\usepackage{xypic}

\title[Ehrhart theory on periodic graphs]
{Ehrhart theory on periodic graphs}

\author[T. Inoue]{Takuya Inoue}
\address{Graduate School of Mathematical Sciences, 
the University of Tokyo, 3-8-1 Komaba, Meguro-ku, Tokyo 153-8914, Japan.}
\email{inoue@ms.u-tokyo.ac.jp}

\author[Y. Nakamura]{Yusuke Nakamura}
\address{Graduate School of Mathematical Sciences, 
the University of Tokyo, 3-8-1 Komaba, Meguro-ku, Tokyo 153-8914, Japan.}
\email{nakamura@ms.u-tokyo.ac.jp}

\newtheorem{thm}{Theorem}[section]
\newtheorem{lem}[thm]{Lemma}

\newtheorem{cor}[thm]{Corollary}
\newtheorem{prop}[thm]{Proposition}

\newtheorem{quest}[thm]{Question}
\newtheorem{claim}[thm]{Claim}

\theoremstyle{definition}
\newtheorem{defi}[thm]{Definition}

\theoremstyle{remark}
\newtheorem{rmk}[thm]{Remark}
\newtheorem{ex}[thm]{Example}
\newtheorem*{ackn}{Acknowledgements}

\newlength{\zua}
\newlength{\zub}
\newlength{\zuc}
\begin{document}
\begin{abstract}
The purpose of this paper is to extend the scope of Ehrhart theory to periodic graphs. 
We give sufficient conditions for the growth sequences of periodic graphs to be a quasi-polynomial and to satisfy the reciprocity laws. 
Furthermore, we apply our theory to determine the growth series in several new examples. 
\end{abstract}

\maketitle

\tableofcontents

\section{Introduction}

In this paper, a graph $\Gamma$ means a directed graph that may have loops and multiple edges. 
We define an \textit{$n$-dimensional periodic graph} $(\Gamma, L)$ as a graph $\Gamma$ on which a free abelian group $L$ of rank $n$ acts freely and its quotient graph $\Gamma/L$ is finite (see Definition \ref{defi:pg}). 
For a vertex $x_0$ of $\Gamma$, the \textit{growth sequence} $(s_{\Gamma, x_0, i})_{i \ge 0}$ 
(resp.\ \textit{cumulative growth sequence} $(b_{\Gamma, x_0, i})_{i \ge 0}$) is defined as the number of vertices of $\Gamma$ whose distance from $x_0$ is $i$ (resp.\ at most $i$). 
The purpose of this paper is to discuss phenomena similar to Ehrhart theory that appear in the growth sequences of periodic graphs. 

Periodic graphs naturally appear in crystallography, and their growth sequences have been studied intensively in this field. 
In crystallography, the growth sequence is also referred to as the \textit{coordination sequence}. 
For a periodic graph $\Gamma$ corresponding to a crystal, $s_{\Gamma, x_0, 1}$ is nothing but the usual coordination number of the atom $x_0$. 
The coordination sequence is utilized in several crystal database entries (cf.\ \cite{DZS}), and it can be useful, for instance, in distinguishing between two allotropes that cannot be distinguished by the coordination number. 

In \cite{GKBS96}, Grosse-Kunstleve, Brunner and Sloane conjectured 
that the growth sequences of periodic graphs are of quasi-polynomial type, i.e., 
there exist an integer $M$ and a quasi-polynomial $f_s: \mathbb{Z} \to \mathbb{Z}$ such that 
$s_{\Gamma, x_0, i} = f_s(i)$ holds for all $i \ge M$ (see Definition \ref{defi:qp}). 
In \cite{NSMN21}, the second author, Sakamoto, Mase, and Nakagawa prove this conjecture to be true for any periodic graphs (Theorem \ref{thm:NSMN}). 
Although it was proved to be of quasi-polynomial type, determining the quasi-polynomial in practice is still difficult. Thus, the following natural question arises. 

\begin{quest}\label{quest:algorithm}
Find an effective algorithm to determine the explicit formulae of the growth sequences.
\end{quest}

So far, various computational methods have been established for several special classes of periodic graphs. 
In \cite{CS97}, Conway and Sloane give growth sequences of the contact graphs of 
some lattices from the viewpoint of Ehrhart theory (cf.\ \cites{BHV99, ABHPS}). 
In \cite{GS19}, Goodman-Strauss and Sloane propose ``the coloring book approach" 
to obtain the growth sequence for some periodic tilings. 
In \cites{SM19, SM20a}, Shutov and Maleev obtained the growth sequences for 
tilings satisfying certain conditions that contain the 20 2-uniform tilings. 
However, as far as we know, no algorithm that can be applied to general periodic graphs has been proposed, even in dimension two. 

The difficulty of Question \ref{quest:algorithm} is due to the difficulty of determining $M$ that appears in the definition of ``quasi-polynomial type'' above. 
Indeed, if this $M$ can be determined and a quasi-period of $f_s$ is known, 
then the explicit formula of  $(s_{\Gamma, x_0, i})_i$ can be determined by its first few terms. 
In this paper, in the context of Ehrhart theory, 
we focus on the category of graphs whose growth sequences are honest quasi-polynomials  (i.e.\ quasi-polynomials on $i > 0$). 
We note in advance that for these graphs, the difficulty of Question \ref{quest:algorithm} is avoided.

In Ehrhart theory, for a rational polytope $Q \subset \mathbb{R}^N$, it is proved that 
the function 
\[
h_Q:\mathbb{Z}_{\ge 0} \to \mathbb{Z}_{\ge 0};\quad i \mapsto \# \bigl( iQ \cap \mathbb{Z}^N \bigr)
\]
is a quasi-polynomial on $i \ge 0$. 
As we will discuss in Subsection \ref{subsection:PtoPG}, for a rational polytope $Q \subset \mathbb{R}^N$ with $0 \in Q$, we can construct a periodic graph $(\Gamma _{Q}, \mathbb{Z}^N)$ such that its cumulative growth sequence $b_{\Gamma _Q, 0, i}$ coincides with $h_Q(i)$. 
Therefore, we can say that the study of the growth sequences of periodic graphs essentially contains the Ehrhart theory of rational polytopes $Q$ satisfying $0 \in Q$. 
Since the cumulative growth sequence $(b_{\Gamma _Q, 0, i})_i$ is a quasi-polynomial on $i \ge 0$, the following natural question arises. 
\begin{quest}\label{quest:qp}
Find a reasonable class $\mathcal{P}$ of pairs $(\Gamma, x_0)$ that consist of a periodic graph $\Gamma$ and one of its vertices $x_0$ such that 
\begin{itemize}
\item 
$\mathcal{P}$ contains the class $\{ (\Gamma _Q, 0) \mid \text{$Q$ is a rational polytope with $0 \in Q$} \}$, and 

\item 
for any $(\Gamma , x_0) \in \mathcal{P}$, the sequence $(b_{\Gamma, x_0, i})_i$ is an honest quasi-polynomial (i.e.\ a quasi-polynomial on $i \ge 0$).
\end{itemize}
\end{quest}
\noindent
The word ``reasonable'' here means that $\mathcal{P}$ should be a class that can be described in terms of graph theory. 
Note that, unlike the case of Ehrhart theory, the growth sequences of periodic graphs in general are not necessarily quasi-polynomials, 
and they may have finite exceptional terms (see Example \ref{ex:CS}). 
However, it has been observed that for some highly symmetric periodic graphs, they are often honest quasi-polynomials. 
The intention of this question is to describe the properties of such good periodic graphs.

Another important topic of Ehrhart theory is the reciprocity law. 
When we think of the function $h_Q$ as a quasi-polynomial and substitute a negative value for it, 
we have $h_Q(-i) = (-1)^{\dim Q} \# \bigl( i \cdot \operatorname{relint}(Q) \cap \mathbb{Z}^N \bigr)$ for $i > 0$. 
In the growth sequences of some $n$-dimensional periodic graphs, it has been observed that they sometimes satisfy the equations 
\[
f_b(-i) = (-1)^n f_b(i-1), \qquad f_s(-i) = (-1)^{n+1} f_s(i), \tag{$\diamondsuit$}
\]
where $f_b$ and $f_s$ are the corresponding quasi-polynomials to the sequences $(b_{\Gamma, x_0, i})_i$ and $(s_{\Gamma, x_0, i})_i$ (see \cites{CS97, SM19, Wakatsuki}). 
These equations in ($\diamondsuit$) are consistent with the reciprocity laws of reflexive polytopes. 
Thus, the following natural question arises. 
\begin{quest}\label{quest:rl}
Find a reasonable class $\mathcal{P}'$ of pairs $(\Gamma, x_0)$ such that 
\begin{itemize}
\item 
$\mathcal{P}'$ contains the class $\{ (\Gamma _Q, 0) \mid \text{$Q$ is a reflexive polytope} \}$, and 

\item 
for any $(\Gamma , x_0) \in \mathcal{P}'$, its growth sequence satisfies the reciprocity laws ($\diamondsuit$).
\end{itemize} 
\end{quest}

The purpose of this paper is to give answers to Questions \ref{quest:qp} and \ref{quest:rl}. 
First, we introduce invariants $C_1(\Gamma, \Phi, x_0) \in \mathbb{R}_{\ge 0}$ and $C_2(\Gamma, \Phi, x_0) \in \mathbb{R}_{\ge 0}$ for $(\Gamma, x_0)$ and a periodic realization $\Phi: \Gamma \to L_{\mathbb{R}} := L \otimes _{\mathbb{Z}} \mathbb{R}$ (i.e.\ a map preserving an $L$-action). 
These invariants measure the deviation from the polytope approximation of $d_{\Gamma}$, where $d_{\Gamma}$ is the distance function of the graph $\Gamma$. 
Using these invariants, we can give sufficient conditions for the cumulative growth sequence to be a quasi-polynomial and to satisfy the reciprocity laws ($\diamondsuit$).

\begin{thm}[{$=$\ Theorem \ref{thm:ET}}]\label{thm:introET}
Let $(\Gamma, L)$ be a strongly connected $n$-dimensional periodic graph, and let $x_0$ be a vertex of $\Gamma$. Then, the following assertions hold. 
\begin{enumerate}
\item 
If a periodic realization $\Phi$ of $(\Gamma, L)$ satisfies $C_1(\Gamma, \Phi, x_0) + C_2(\Gamma, \Phi, x_0) < 1$, then the cumulative growth sequence $(b_{\Gamma, x_0, i})_i$ is a quasi-polynomial on $i \ge 0$. 

\item 
If a periodic realization $\Phi$ of $(\Gamma, L)$ satisfies both $C_1(\Gamma, \Phi, x_0) < \frac{1}{2}$ and $C_2(\Gamma, \Phi, x_0) < \frac{1}{2}$, then the reciprocity laws ($\diamondsuit$) are satisfied. 
\end{enumerate}
\end{thm}
\noindent
Note that if $\Gamma = \Gamma _Q$ is a periodic graph obtained by a polytope $Q$ with $0 \in Q$, then it follows that $C_1(\Gamma, \Phi, x_0) = 0$ and $C_2(\Gamma, \Phi, x_0) < 1$. 
Furthermore, if $\Gamma = \Gamma _Q$ is a periodic graph obtained by a reflexive polytope $Q$, then it follows that $C_1(\Gamma, \Phi, x_0) = C_2(\Gamma, \Phi, x_0) = 0$.
Therefore, Theorem \ref{thm:introET} is an answer to Questions \ref{quest:qp} and \ref{quest:rl}. 
We also note that Theorem \ref{thm:introET} is a generalization of a result by Conway and Sloane \cite{CS97}, where they only consider the contact graphs of lattices (see Remark \ref{rmk:CS97} for more detail). 

\vspace{4mm}

\begin{itembox}{Answers to Questions \ref{quest:qp} and \ref{quest:rl}}
\[
\xymatrix{
\text{$\Gamma = \Gamma _{Q}$ for a polytope $Q$ with $0 \in Q$} \ar@{=>}[d]^-{\text{\ \ Lemma \ref{lem:PG}}} & \text{$\Gamma = \Gamma _{Q}$ for a reflexive polytope $Q$} \ar@{=>}[d]^-{\text{\ \ Lemma \ref{lem:PG}}} \\
C_1 = 0,\ C_2 <1 \ar@{=>}[d] & C_1=C_2=0 \ar@{=>}[d]\\
C_1 + C_2<1 \ar@{=>}[d]^-{\text{\ \ Theorem \ref{thm:ET}}} & C_1<\frac{1}{2},\ C_2<\frac{1}{2} \ar@{=>}[d]^-{\text{\ \ Theorem \ref{thm:ET}}} \ar@{=>}[l]\\
\text{$(b_{\Gamma, x_0, i})_i$ is an honest q-polynomial.} & 
\text{The reciprocity laws ($\diamondsuit$).}\\
}
\]
\end{itembox}

\vspace{4mm}

For undirected periodic graphs, we can give a larger class that satisfies the reciprocity laws ($\diamondsuit$). 
In Subsection \ref{subsection:WA}, we define a class of undirected periodic graphs called ``well-arranged", and we prove that their growth sequences satisfy the reciprocity laws ($\diamondsuit$) (Theorem \ref{thm:WA}). 
We also introduce the notion of ``$P$-initial" for a vertex $x_0$, and 
we see that a quasi-period of the growth sequence can be explicitly given in this case (Theorem \ref{thm:Piniqp}). The relationship with the invariants $C_1$ and $C_2$ can be summarized in the following diagram: 

\vspace{4mm}

\begin{itembox}{Well-arranged graphs}
\[
\xymatrix{
\text{$\Gamma$ is undirected},\ C_1<\frac{1}{2},\ C_2<\frac{1}{2}. \ar@{=>}[d]^-{\text{\ \ Proposition \ref{prop:refWA}}} & C_2 < 1 \ar@{=>}[d]^-{\text{\ \ Lemma \ref{lem:C1C2}}}\\
\text{$(\Gamma, \Phi, x_0)$ is well-arranged.} \ar@{=>}[d]^-{\text{\ \ Theorem \ref{thm:WA}}} \ar@{=>}[r]^-{\text{Lemma \ref{lem:WAPI}}} & \text{$x_0$ is $P$-initial.} \ar@{=>}[d]^-{\text{\ \ Theorem \ref{thm:Piniqp}}}\\
{\begin{array}{l}
\text{$\bullet$ $(b_{\Gamma, x_0, i})_i$ is an honest q-polynomial.}\\
\text{$\bullet$ The reciprocity laws ($\diamondsuit$).}
\end{array}} &
{\begin{array}{l}
\text{$\bullet$ a quasi-period of $(b_{\Gamma, x_0, i})_i$}\\
\text{\phantom{$\bullet$} can be explicitly given.}
\end{array}}
}
\]
\end{itembox}

\vspace{4mm}

When $(\Gamma, \Phi, x_0)$ is well-arranged, the growth sequence $(s_{\Gamma, x_0, i})_i$ is a quasi-polynomial on $i \ge 1$, and its quasi-period can be explicitly given. 
Therefore, it is possible to determine the explicit formula of $(s_{\Gamma, x_0, i})_i$ by computing the first few terms of it. 
By this method, we compute the growth series in several new examples in Section \ref{section:eg}. 
As far as we know, this is the first time that the growth sequences have been computed for nontrivial $3$-dimensional periodic graphs. 
However, this method is only applicable to well-arranged graphs, and
answers to Question \ref{quest:algorithm} for general periodic graphs remain as future work.

The paper is organized as follows: 
in Section \ref{section:pre}, we first summarize the notations of graphs,
and we define periodic graphs and their growth sequences. 
We then define an important concept, the \textit{growth polytope}, and use it to define the invariants $C_1$ and $C_2$. 
In Section \ref{section:ETP}, we give sufficient conditions for the cumulative growth sequence to be a quasi-polynomial type and to satisfy the reciprocity laws (Theorem \ref{thm:ET}). Furthermore, in Subsection \ref{subsection:PtoPG}, 
we see that Theorem \ref{thm:ET} can be seen as a generalization of the usual Ehrhart theory for polytopes $Q$ with $0 \in Q$ and for reflexive polytopes $Q$. 
In Section \ref{section:PiniWA}, we treat a periodic graph $(\Gamma, L)$ with a $P$-initial vertex $x_0$. In this case, we can calculate the invariant $C_2$ (Proposition \ref{prop:C2}) and 
a quasi-period of the growth sequence of $\Gamma$ with the start point $x_0$ (Theorem \ref{thm:Piniqp}).
Furthermore, we also introduce a class of periodic graphs called ``well-arranged", 
and we prove that their growth sequences satisfy the reciprocity laws (Theorem \ref{thm:WA}). 
In Section \ref{section:eg}, for some specific periodic graphs, we will see whether they are well-arranged and discuss their growth series. 
Furthermore, as an application of Theorem \ref{thm:WA}, we determine the growth series in several new examples (Subsections \ref{subsection:eg_2dim} and \ref{subsection:eg_3dim}). 
In Appendix \ref{section:GP}, we summarize the properties of the growth polytope necessary for the definition of the invariants $C_1$ and $C_2$.
In Section \ref{section:ET}, we discuss a variant of Ehrhart theory (Theorem \ref{thm:Eh}), which is necessary for the proof of Theorem \ref{thm:ET}. The difference from the usual Ehrhart theory is that the center of the dilation need not be the origin, and the dilation factor may be shifted by a constant. 

\begin{ackn}
We would like to thank Professors Akihiro Higashitani, Atsushi Ito, Takafumi Mase, Junichi Nakagawa, Hiroyuki Ochiai, and Masahiko Yoshinaga for many discussions. 
The first author is partially supported by FoPM, WINGS Program, the University of Tokyo and JSPS KAKENHI No.\ 23KJ0795.
The second author is partially supported by JSPS KAKENHI No.\ 18K13384, 22K13888, and JPJSBP120219935. 
Figures \ref{fig:sacada1} and \ref{fig:sacada60} are drawn by using {\sf VESTA} \cite{MI11}. 
\end{ackn}

\section{Notation and Preliminaries}\label{section:pre}

\subsection{Notation}
For a set $X$, $\# X$ denotes the cardinality of $X$, 
and $2 ^X$ denotes the power set of $X$. 

For a finite subset $S \subset \mathbb{Z}_{>0}$, $\operatorname{LCM}(S)$ denotes the least common multiple of the elements of $S$. 

For a polytope $P \subset \mathbb{R}^N$, 
$\operatorname{Facet}(P)$ denotes the set of facets of $P$, 
$\operatorname{Face}(P)$ denotes the set of faces of $P$, and 
$V(P)$ denotes the set of vertices of $P$. 
Note that both $P$ itself and the empty set $\emptyset$ are considered as faces of $P$. 

For a subset $C \subset \mathbb{R}^N$, 
$\operatorname{int}(C)$ denotes the interior of $C$, and 
$\operatorname{relint}(C)$ denotes the relative interior of $C$. 

For a polytope $P \subset \mathbb{R}^N$ of dimension $d$, 
a \textit{triangulation} $T$ of $P$ means a finite collection of $d$-simplices with the following two conditions: 
\begin{itemize}
\item $P = \bigcup _{\Delta \in T} \Delta$. 
\item For any $\Delta _1, \Delta _2 \in T$, $\Delta _1 \cap \Delta _2$ is a face of $\Delta _1$ and $\Delta _2$. 
\end{itemize}

In this paper, monoids always mean commutative monoids. 
We refer the reader to \cite{BG2009} and \cite{NSMN21} for the terminology of monoid and its module theory.

Let $M$ be a set equipped with a binary operation $*$. 
For $u \in M$ and subsets $X, Y \subset M$, we define subsets $u * X, X * Y \subset M$ by
\[
u * X := \{ u * x \mid x \in X \}, \qquad
X * Y := \{ x * y \mid (x,y) \in X \times Y \}. 
\]

\subsection{Graphs and walks}\label{subsection:GandW}

In this paper, a \textit{graph} means a directed weighted graph which may have loops and multiple edges. 
A graph $\Gamma = (V_{\Gamma}, E_{\Gamma}, s_{\Gamma}, t_{\Gamma}, w_{\Gamma})$ 
consists of the set $V_{\Gamma}$ of vertices, the set $E_{\Gamma}$ of edges, 
the source function $s_{\Gamma}: E_{\Gamma} \to V_{\Gamma}$, 
the target function $t_{\Gamma}: E_{\Gamma} \to V_{\Gamma}$, and 
the (integer) weight function $w_{\Gamma}: E_{\Gamma} \to \mathbb{Z}_{> 0}$. 
We often abbreviate $s_{\Gamma}$, $t_{\Gamma}$ and $w_{\Gamma}$ to $s$, $t$ and $w$ 
when no confusion can arise.

\begin{defi}
Let $\Gamma = (V_{\Gamma}, E_{\Gamma}, s, t, w)$ be a graph. 

\begin{enumerate}
\item 
$\Gamma$ is called to be \textit{unweighted} if $w(e) = 1$ holds for every $e \in E_{\Gamma}$. 
$\Gamma$ is called to be \textit{undirected} when 
there exists an involution $E_{\Gamma} \to E_{\Gamma};\ e \mapsto e'$
such that $s(e)=t(e')$, $t(e)=s(e')$ and $w(e)=w(e')$. 
$\Gamma$ is called to be \textit{locally finite} when for all $x \in V_{\Gamma}$, 
there are only finitely many edges $e$ satisfying $s(e)=x$ and only finitely many edges $e$ satisfying $t(e)=x$. 

\item
A \textit{walk} $p$ in $\Gamma$ is a sequence $e_1 e_2 \cdots e_{\ell}$ of edges $e_i$ of $\Gamma$ satisfying 
$t(e_i) = s(e_{i+1})$ for each $i = 1, \ldots , \ell -1$. 
We define 
\[
s(p) := s(e_1), \quad
t(p) := t(e_{\ell}), \quad
w(p) := \sum _{i = 1} ^{\ell} w(e_i), \quad 
\operatorname{length}(p) := \ell. 
\]
Note that we have $w(p) = \operatorname{length}(p)$ if $\Gamma$ is unweighted.

We say that ``$p$ is a walk from $x$ to $y$" when $x = s(p)$ and $y = t(p)$. 
We also define the \textit{support} $\operatorname{supp}(p) \subset V_{\Gamma}$ of $p$ by
\[
\operatorname{supp}(p) := \{ s(e_1), t(e_1), t(e_2), \ldots , t(e_{\ell}) \} \subset V_{\Gamma}. 
\]

By convention, each vertex $v \in V_{\Gamma}$ is also considered as a walk of length $0$. 
This is called the \textit{trivial walk} at $v$ and denoted by $\emptyset _v$: i.e., 
we define 
\[
s(\emptyset _v) := v, \ \ 
t(\emptyset _v) := v, \ \ 
w(\emptyset _v) := 0, \ \  
\operatorname{length}(\emptyset _v) := 0, \ \ 
\operatorname{supp}(\emptyset _v) := \{ v \}.
\]

\item 
A \textit{path} in $\Gamma$ is a walk $e_1 \cdots e_{\ell}$ such that $s(e_1), t(e_1), t(e_2), \ldots, t(e_{\ell})$ are distinct. 
A walk of length $0$ is considered as a path. 

\item
A \textit{cycle} in $\Gamma$ is a walk $e_1 \cdots e_{\ell}$ with $s(e_1) = t(e_{\ell})$ 
such that $t(e_1), t(e_2), \ldots, t(e_{\ell})$ are distinct. 
A walk of length $0$ is not considered as a cycle. 
$\operatorname{Cyc}_{\Gamma}$ denotes the set of cycles in $\Gamma$. 

\item 
For $x, y \in V_{\Gamma}$, $d_{\Gamma}(x, y) \in \mathbb{Z}_{\ge 0} \cup \{ \infty \}$ denotes the smallest weight $w(p)$ of any walk $p$ from $x$ to $y$. 
By convention, we define $d_{\Gamma}(x, y) = \infty$ when there is no walk from $x$ to $y$. 
A graph $\Gamma$ is said to be \textit{strongly connected} 
when we have $d_{\Gamma}(x, y) < \infty$ for all $x, y \in V_{\Gamma}$. 
When $\Gamma$ is undirected, we have $d_{\Gamma}(x,y) = d_{\Gamma}(y,x)$ for all $x,y \in V_{\Gamma}$. 

\item 
$C_1(\Gamma, \mathbb{Z})$ denotes the group of $1$-chains on $\Gamma$ with coefficients in $\mathbb{Z}$, i.e., 
$C_1(\Gamma, \mathbb{Z})$ is the free abelian group generated by $E_{\Gamma}$. 
For a walk $p = e_1 \cdots e_{\ell}$ in $\Gamma$, let $\langle p \rangle$ denote the $1$-chain 
$\sum _{i = 1} ^{\ell} e_i \in C_1(\Gamma, \mathbb{Z})$. 
$H_1(\Gamma, \mathbb{Z}) \subset C_1(\Gamma, \mathbb{Z})$ denotes the $1$-st homology group, i.e., 
$H_1(\Gamma, \mathbb{Z})$ is a subgroup generated by $\langle p \rangle$ for $p \in \operatorname{Cyc}_{\Gamma}$. 
We refer the reader to \cite{Sunada} for more detail. 
\end{enumerate}
\end{defi}

\subsection{Periodic graphs}

\begin{defi}\label{defi:pg}
Let $n$ be a positive integer. 
An \textit{$n$-dimensional periodic graph} $(\Gamma, L)$ is a graph $\Gamma$ and 
a free abelian group $L \simeq \mathbb{Z}^n$ of rank $n$ with the following two conditions:
\begin{itemize}
\item $L$ freely acts on both $V_\Gamma$ and $E_\Gamma$, 
and their quotients $V_\Gamma / L$ and $E_\Gamma / L$ are finite sets. 
\item This action preserves the edge relations, i.e.,\ for any $u \in L$ and $e \in E_\Gamma$, 
we have $s_\Gamma(u(e))=u(s_\Gamma(e))$, $t_\Gamma(u(e))=u(t_\Gamma(e))$ and 
$w_{\Gamma}(u(e))=w_{\Gamma}(e)$.
\end{itemize}
Then, $L$ is called the \textit{period lattice} of $\Gamma$. 
Note that $\Gamma$ automatically becomes a locally finite graph. 
\end{defi}

If $(\Gamma, L)$ is an $n$-dimensional periodic graph, 
then the \textit{quotient graph} $\Gamma /L = (V_{\Gamma/L}, E_{\Gamma/L}, s_{\Gamma/L},t_{\Gamma/L},w_{\Gamma/L})$ is defined by 
$V_{\Gamma/L} := V_{\Gamma} /L$, $E_{\Gamma/L} := E_{\Gamma} /L$, 
and the functions $s_{\Gamma/L}: E_\Gamma/L \to V_\Gamma/L$, 
$t_{\Gamma/L}: E_\Gamma/L \to V_\Gamma/L$, and 
$w_{\Gamma/L}: E_\Gamma/L \to \mathbb{Z}_{>0}$ induced from $s_{\Gamma}$, $t_{\Gamma}$, and $w_{\Gamma}$. 
Note that the functions $s_{\Gamma/L}$, $t_{\Gamma/L}$ and $w_{\Gamma/L}$ 
are well-defined due to the second condition in Definition \ref{defi:pg}.

\begin{defi}\label{defi:peri}
Let $(\Gamma,L)$ be an $n$-dimensional periodic graph.
\begin{enumerate}
\item Since $L$ is an abelian group, we use the additive notation: 
for $u \in L$, $x \in V_{\Gamma}$, $e \in E_{\Gamma}$ and a walk $p = e_1 \cdots e_{\ell}$,
$u + x$, $u + e$ and $u + p$ denote their translations by $u$. 

\item 
For any $x \in V_{\Gamma}$ and $e \in E_{\Gamma}$, 
let $\overline{x} \in V_{\Gamma /L}$ and $\overline{e} \in E_{\Gamma /L}$ denote their images in 
$V_{\Gamma /L} = V_{\Gamma} /L$ and $E_{\Gamma /L} = E_{\Gamma} /L$. 
For a walk $p = e_1 \cdots e_{\ell}$ in $\Gamma$, 
let $\overline{p} := \overline{e_1} \cdots \overline{e_{\ell}}$ denote its image in $\Gamma /L$. 

\item 
When $x, y \in V_{\Gamma}$ satisfy $\overline{x} = \overline{y}$, there exists an element $u \in L$ such that $u + x = y$. 
Since the action is free, such $u \in L$ uniquely exists and is denoted by $y - x$. 

\item 
For a walk $p$ in $\Gamma$ with $\overline{s(p)} = \overline{t(p)}$, we define 
\[
\operatorname{vec}(p) := t(p) - s(p) \in L. 
\]
\end{enumerate}
\end{defi}

\begin{defi}\label{defi:pr}
Let $(\Gamma,L)$ be an $n$-dimensional periodic graph.
We define $L_{\mathbb{R}} := L \otimes _{\mathbb{Z}} \mathbb{R}$. 
\begin{enumerate}
\item 
A \textit{periodic realization} $\Phi: V_{\Gamma} \to L_{\mathbb{R}}$ is a map satisfying
$\Phi(u+x) = u + \Phi(x)$ for any $u \in L$ and $x \in V _{\Gamma}$. 
When we fix an injective periodic realization of $\Phi: V _{\Gamma} \to L_{\mathbb{R}}$, 
we sometimes identify $V _{\Gamma}$ with the subset of $L_{\mathbb{R}}$. 

\item 
Let $\Phi$ be a periodic realization of $(\Gamma, L)$. 
For an edge $e$ and a walk $p$ in $\Gamma$, we define 
\[
\operatorname{vec}_{\Phi}(e) := \Phi(t(e)) - \Phi(s(e)) \in L_{\mathbb{R}}, \quad
\operatorname{vec}_{\Phi}(p) := \Phi(t(p)) - \Phi(s(p)) \in L_{\mathbb{R}}. 
\]
It is easy to see that the value $\operatorname{vec}_{\Phi}(e) \in L_{\mathbb{R}}$ depends only on the class $\overline{e} \in E_{\Gamma /L}$, 
and therefore, 
the map 
\[
\mu _{\Phi} : E_{\Gamma /L} \to L_{\mathbb{R}};\quad \overline{e} \mapsto \operatorname{vec}_{\Phi}(e)
\]
is well-defined. 
It can be extended to a homomorphism 
\[
\mu _{\Phi}: C_1(\Gamma/L, \mathbb{Z}) \to L_{\mathbb{R}}; \quad 
\sum a_i \overline{e_i} \mapsto \sum a_i \mu _{\Phi} (\overline{e_i}). 
\]
By construction, it satisfies $\mu _{\Phi} (\langle \overline{p} \rangle) = \operatorname{vec}_{\Phi}(p)$ for any walk $p$ in $\Gamma$. 
\end{enumerate}
\end{defi}

\begin{rmk}\label{rmk:finmap}\hfill
\begin{enumerate} 
\item 
An injective periodic realization of $(\Gamma, L)$ always exists. 
To see this fact, we take any injective map $V_{\Gamma} / L \to L_{\mathbb{R}}/L$. 
It is possible because $\# (V_{\Gamma} / L) < \infty$. 
Then, any injective map $V_{\Gamma} / L \to L_{\mathbb{R}}/L$ lifts to 
an injective periodic realization $V_{\Gamma} \to L_{\mathbb{R}}$. 

\item 
Even if a periodic realization $\Phi$ is not injective, 
the map $V_{\Gamma} \to L_{\mathbb{R}} \times V_{\Gamma}/L;\ x \mapsto \bigl( \Phi(x), \overline{x} \bigr)$ is always injective. 

\item
In Definition \ref{defi:pr}(2), we have $\operatorname{vec}_{\Phi}(p) = \operatorname{vec}(p)$
for any $p$ with $\overline{s(p)} = \overline{t(p)}$.
\end{enumerate}
\end{rmk}

\begin{ex}[{cf.\ \cite{Wakatsuki}, \cite{NSMN21}*{Figure 3}}]\label{ex:WG}
The \textit{Wakatsuki graph} is an undirected unweighted graph 
$\Gamma = (V_\Gamma, E_\Gamma, s_{\Gamma}, t_{\Gamma})$ defined by
\[
V_\Gamma = \{ v_0,v_1,v_2 \} \times \mathbb{Z}^2, \quad
E_\Gamma = \{ e_0,e_1,\ldots,e_9 \} \times \mathbb{Z}^2, 
\]
and 
\begin{gather*}
s_\Gamma(e_0,(x,y))=s_\Gamma(e_1,(x,y))=s_\Gamma(e_2,(x,y))=s_\Gamma(e_3,(x,y))=(v_0,(x,y)), \\
s_\Gamma(e_4,(x,y))=s_\Gamma(e_5,(x,y))=s_\Gamma(e_6,(x,y))=s_\Gamma(e_7,(x,y))=(v_1,(x,y)), \\
s_\Gamma(e_8,(x,y))=s_\Gamma(e_9,(x,y))=(v_2,(x,y)), \\
t_\Gamma(e_0,(x,y))=(v_1,(x,y)),\quad t_\Gamma(e_1,(x,y))=(v_1,(x-1,y)),\\ 
t_\Gamma(e_2,(x,y))=(v_1,(x-1,y-1)),\quad t_\Gamma(e_3,(x,y))=(v_2,(x,y)),\\
t_\Gamma(e_4,(x,y))=(v_0,(x,y)),\quad t_\Gamma(e_5,(x,y))=(v_0,(x+1,y)),\\
t_\Gamma(e_6,(x,y))=(v_0,(x+1,y+1)),\quad t_\Gamma(e_7,(x,y))=(v_2,(x,y)),\\
t_\Gamma(e_8,(x,y))=(v_0,(x,y)),\quad t_\Gamma(e_9,(x,y))=(v_1,(x,y))
\end{gather*}
for any $(x,y) \in \mathbb{Z}^2$. 
Then, $\Gamma$ admits an action of $L=\mathbb{Z}^2$ by 
\begin{align*}
(a,b) + (v_i, (x,y)) &:= (v_i, (x+a,y+b)), \\ 
(a,b) + (e_j, (x,y)) &:= (e_j, (x+a,y+b))
\end{align*}
for each $(a,b), (x,y) \in \mathbb{Z}^2$, $0 \le i \le 2$ and $0 \le j \le 9$. 
By this action, $(\Gamma, L)$ becomes a $2$-dimensional periodic graph. 
A periodic realization $\Phi$ of $(\Gamma, L)$ is defined by
\[
	\Phi(v_0,(x,y))=(x,y),\,\, 
	\Phi(v_1,(x,y))=\left(x+\frac{1}{2},y+\frac{1}{2}\right),\,\, 
	\Phi(v_2,(x,y))=\left(x+\frac{1}{2},y\right)
\]
Let $v'_i := (v_i, (0,0))$ for each $0 \le i \le 2$, and 
let $e'_j := (e_j, (0,0))$ for each $0 \le j \le 9$. 
Then, the realization and the quotient graph $\Gamma / L$ can be illustrated 
as in Figures \ref{fig:WG1}, \ref{fig:WG2} and \ref{fig:WG3}. 
\begin{figure}[htbp]
\centering
\includegraphics[height=6cm]{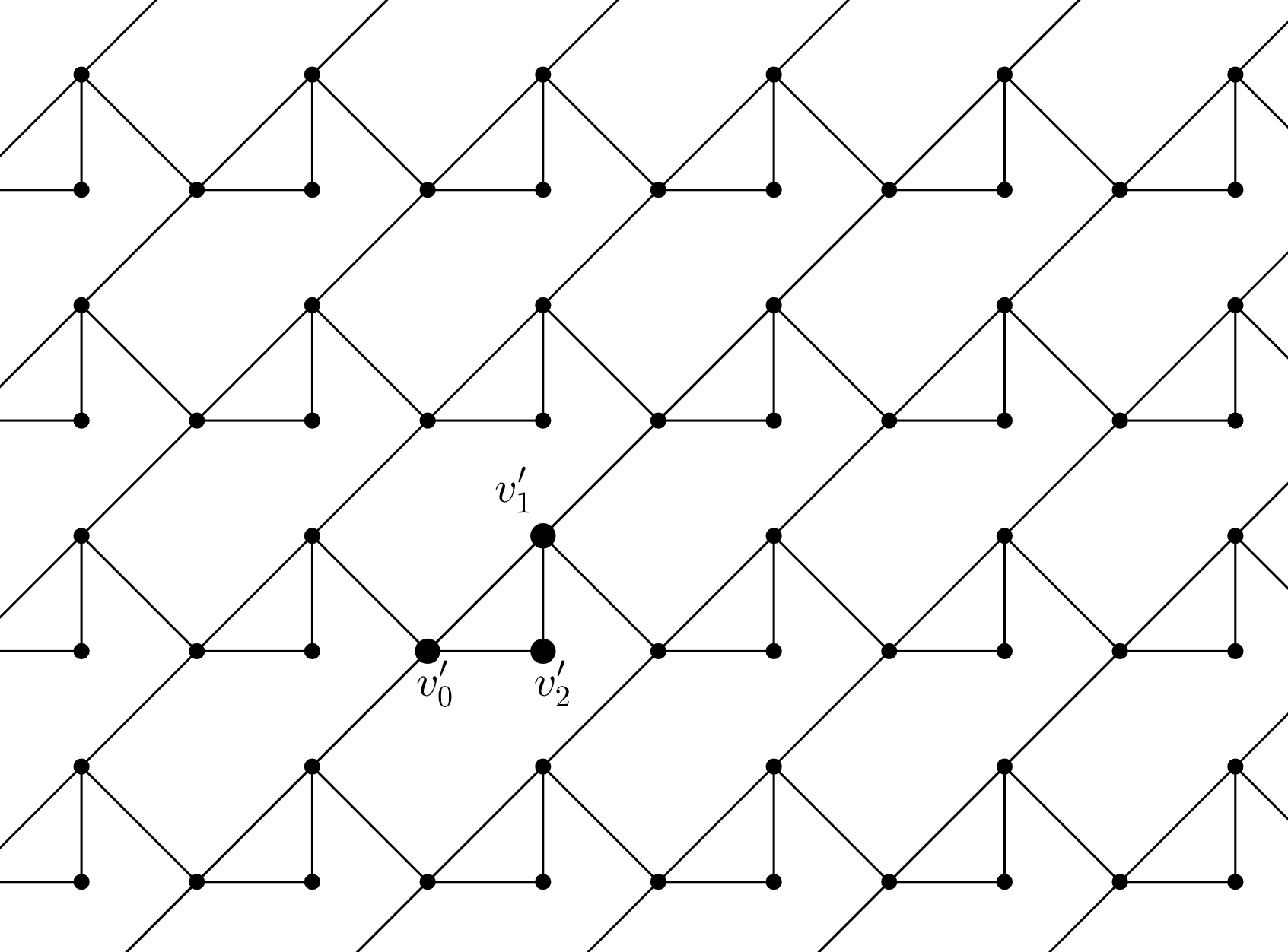}

\vspace{\zub}

\caption{The Wakatsuki graph $\Gamma$.}
\label{fig:WG1}
\end{figure}

\vspace{\zuc}

\begin{figure}[htbp]
\begin{tabular}{cc}
\begin{minipage}{0.49\hsize}
\centering
\includegraphics[height=6cm]{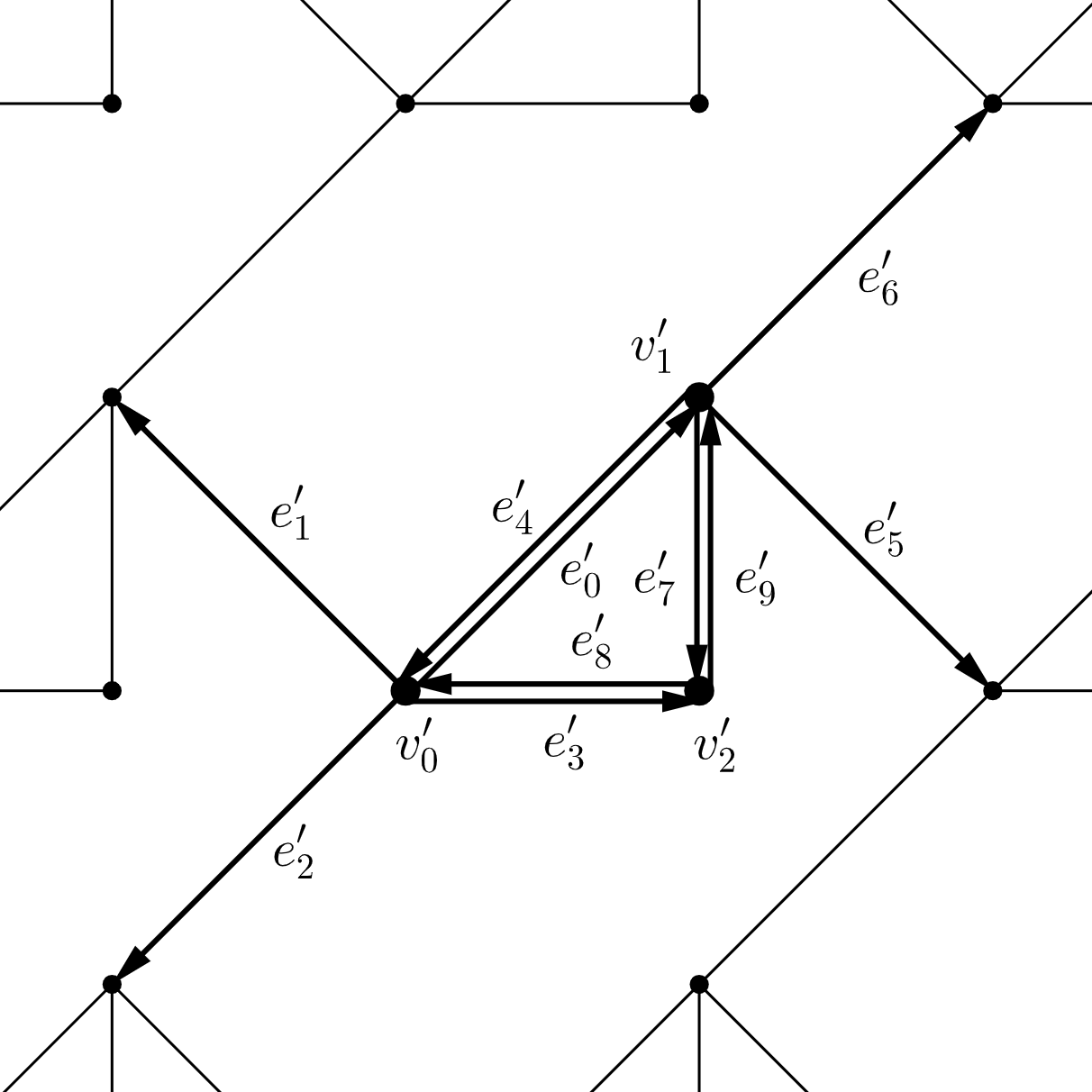}

\vspace{\zub}

\caption{$e'_0, \ldots, e'_9$.}
\label{fig:WG2}
\end{minipage}&
\begin{minipage}{0.49\hsize}
\centering
\includegraphics[width=6cm]{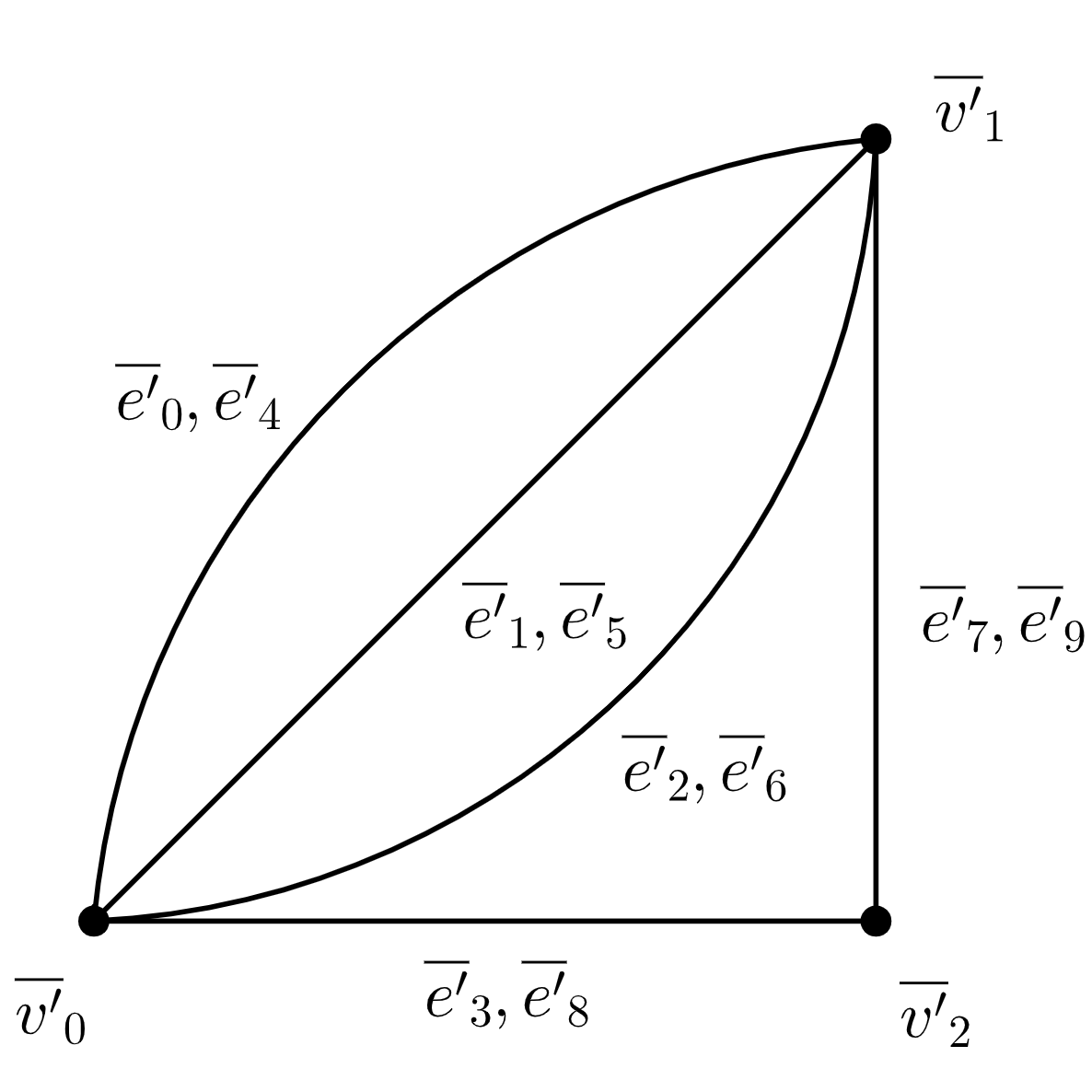}

\vspace{\zub}

\caption{$\Gamma / L$.}
\label{fig:WG3}
\end{minipage}
\end{tabular}
\end{figure}
\end{ex}

\begin{lem}\label{lem:mu}
Let $(\Gamma,L)$ be an $n$-dimensional periodic graph. 
Let $\mu _{\Phi}$ be the homomorphism defined in Definition \ref{defi:pr}(2). 
The restriction map $\mu_{\Phi} |_{H_1(\Gamma/L, \mathbb{Z})}: H_1(\Gamma/L, \mathbb{Z}) \to L_{\mathbb{R}}$ 
is independent of the choice of the periodic realization $\Phi$. 
Furthermore, its image is contained in $L$. 
\end{lem}
\begin{proof}
Since $H_1(\Gamma/L, \mathbb{Z})$ is generated by $\langle q \rangle$ for $q \in \operatorname{Cyc}_{\Gamma /L}$, 
it is sufficient to show that $\mu_{\Phi} (\langle q \rangle) \in L$ 
and that the value $\mu_{\Phi} (\langle q \rangle)$ is independent of the choice of $\Phi$. 
Take any walk $p$ in $\Gamma$ such that $\overline{p} = q$. 
Since $\overline{t(p)} = \overline{s(p)}$, we have 
\[
\mu_{\Phi} (\langle q \rangle) = \operatorname{vec}_{\Phi}(p) = \Phi(t(p)) - \Phi(s(p)) = t(p) - s(p) \in L, 
\]
and this is independent of the choice of $\Phi$. 
\end{proof}

\begin{defi}\label{defi:mu}
Let $(\Gamma, L)$ be an $n$-dimensional periodic graph. 
We define $\mu : H_1(\Gamma/L, \mathbb{Z}) \to L$ as the restriction of 
$\mu _{\Phi}$ in Definition \ref{defi:pr}(2) to $H_1(\Gamma/L, \mathbb{Z})$. 
Note that this restriction map is well-defined by Lemma \ref{lem:mu}. 
\end{defi}

\begin{rmk}\label{rmk:mu}
The homomorphism $\mu$ in Definition \ref{defi:mu} coincides with $\mu$ defined in \cite{Sunada}*{Section 6.1}. 
\end{rmk}

\begin{ex}
In the Wakatsuki graph (see Example \ref{ex:WG}), 
the walks $\overline{e'_5} \ \overline{e'_0}$ and $\overline{e'_6} \ \overline{e'_3} \ \overline{e'_9}$ in $\Gamma / L$ are examples of cycles, and we have 
\[
\mu \left( \left \langle \overline{e'_5} \ \overline{e'_0} \right \rangle \right) = (1,0), \quad 
\mu \left( \left \langle \overline{e'_6} \ \overline{e'_3} \ \overline{e'_9} \right \rangle \right) = (1,1). 
\]
\end{ex}

We finish this subsection with some observations on the decomposition and the composition of walks.

\begin{defi}
Let $(\Gamma, L)$ be an $n$-dimensional periodic graph. 
Let $q_0$ be a path in $\Gamma /L$, and let $q_1, \ldots, q_{\ell} \in \operatorname{Cyc}_{\Gamma /L}$ be cycles. 
The sequence $(q_0, q_1, \ldots, q_{\ell})$ is called \textit{walkable} 
if there exists a walk $q'$ in $\Gamma / L$ such that 
$\langle q' \rangle = \sum _{i = 0} ^{\ell} \langle q_i \rangle$. 
\end{defi}

\begin{lem}\label{lem:bunkai}
Let $(\Gamma, L)$ be an $n$-dimensional periodic graph. 
\begin{enumerate}
\item 
For a walk $q'$ in $\Gamma / L$, there exists a walkable sequence $(q_0, q_1, \ldots , q_{\ell})$ such that 
$\langle q' \rangle = \sum _{i = 0} ^{\ell} \langle q_i \rangle$. 

\item 
Let $q_0$ be a path in $\Gamma /L$, and let $q_1, \ldots, q_{\ell} \in \operatorname{Cyc}_{\Gamma /L}$ be cycles. 
Then, $(q_0, q_1, \ldots, q_{\ell})$ is walkable if and only if 
there exists a permutation $\sigma:  \{ 1,2, \ldots, \ell \} \to \{ 1,2, \ldots, \ell \}$ such that
\[
\biggl( \operatorname{supp}(q_0) \cup \bigcup _{1 \le i \le k} \operatorname{supp}(q_{\sigma(i)}) \biggr) 
\cap \operatorname{supp}(q_{\sigma(k+1)}) \not = \emptyset
\] 
holds for any $0 \le k \le \ell -1$.

\end{enumerate}
\end{lem}
\begin{proof}
For any walk $q'$ in $\Gamma /L$, 
if $q'$ is not a path, then 
there exist a walk $q''$ and a cycle $q_1$ in $\Gamma /L$ 
such that $\langle q' \rangle = \langle q'' \rangle + \langle q_1 \rangle$. 
Therefore, the assertion (1) follows by the induction on the length of $q'$. 
The assertion (2) also follows from the induction. 
\end{proof}

\begin{rmk}\label{rmk:bunkai} \hfill
\begin{enumerate}
\item
If $q'$ in Lemma \ref{lem:bunkai}(1) satisfies $s(q') = t(q')$, then $q_0$ must be a trivial path 
(i.e., $\operatorname{length}(q_0)=0$). 

\item
If a walk $q'$ in $\Gamma /L$ and a vertex $x_0 \in V_{\Gamma}$ satisfy $s(q') = \overline{x_0}$, 
then there exists the unique walk $p$ in $\Gamma$ satisfying $\overline{p} = q'$ and $s(p) = x_0$ 
(we call such $p$ the \textit{lift} of $q'$ with initial point $x_0$). 
Therefore, for a walkable sequence $(q_0, q_1, \ldots, q_{\ell})$ and a vertex $x_0 \in V_{\Gamma}$ satisfying $s(q_0) = \overline{x_0}$, 
there exists a walk $p$ in $\Gamma$ such that 
$s(p) = x_0$ and $\langle \overline{p} \rangle = \sum _{i = 0} ^{\ell} \langle q_i \rangle$. 
Conversely, for a walk $p$ in $\Gamma$, applying Lemma \ref{lem:bunkai}(1) to $\overline{p}$, 
there exists a walkable sequence $(q_0, q_1, \ldots, q_{\ell})$ satisfying 
$\langle \overline{p} \rangle = \sum _{i = 0} ^{\ell} \langle q_i \rangle$. 

\item
For a walk $p$ in $\Gamma$ and a walkable sequence $(q_0, q_1, \ldots, q_{\ell})$ satisfying 
$\langle \overline{p} \rangle = \sum _{i = 0} ^{\ell} \langle q_i \rangle$,  
it follows that 
\[
w(p) = 
\sum _{i=0}^{\ell} w(q_i), \qquad
\operatorname{length}(p) = 
\sum _{i=0}^{\ell} \operatorname{length}(q_i). 
\]
Furthermore, if we fix a periodic realization $\Phi$, we also have 
\[
\operatorname{vec}_{\Phi}(p) = \sum _{i=0}^{\ell} \mu_{\Phi} ( \langle q_i \rangle). 
\]
\end{enumerate}
\end{rmk}

\begin{ex}\label{ex:path}
In the Wakatsuki graph (see Example \ref{ex:WG}), we consider a walk $p$ as in Figure \ref{fig:p}. 
Then, the image of $p$ in the quotient graph $\Gamma /L$ is given by
\[
\overline{p} = \overline{e'_3} \ \overline{e'_9} \ \overline{e'_5} \ \overline{e'_0} \ \overline{e'_6} \ \overline{e'_3} \ 
\overline{e'_9} \ \overline{e'_6} \ \overline{e'_1} \ \overline{e'_6} \ \overline{e'_3} \ \overline{e'_9}, 
\]
Then, we have two decompositions 
\begin{align*}
\langle \overline{p} \rangle 
&= \left \langle \overline{e'_1} \right \rangle + \left \langle \overline{e'_5} \ \overline{e'_0} \right \rangle
+ 3 \left \langle \overline{e'_6} \ \overline{e'_3} \ \overline{e'_9} \right \rangle, \\
\langle \overline{p} \rangle 
&= \left \langle \overline{e'_3} \ \overline{e'_9} \right \rangle + \left \langle \overline{e'_5} \ \overline{e'_0} \right \rangle
+ \left \langle \overline{e'_6} \ \overline{e'_1} \right \rangle + 2 \left \langle \overline{e'_6} \ \overline{e'_3} \ \overline{e'_9} \right \rangle. 
\end{align*}
Here, $\overline{e'_1}$ and $\overline{e'_3} \ \overline{e'_9}$ are paths in $\Gamma / L$, and 
$\overline{e'_5} \ \overline{e'_0}$, $\overline{e'_6} \ \overline{e'_3} \ \overline{e'_9}$, and $\overline{e'_6} \ \overline{e'_1}$
are cycles in $\Gamma / L$. 
\begin{figure}[htbp]
\centering
\includegraphics[height=6cm]{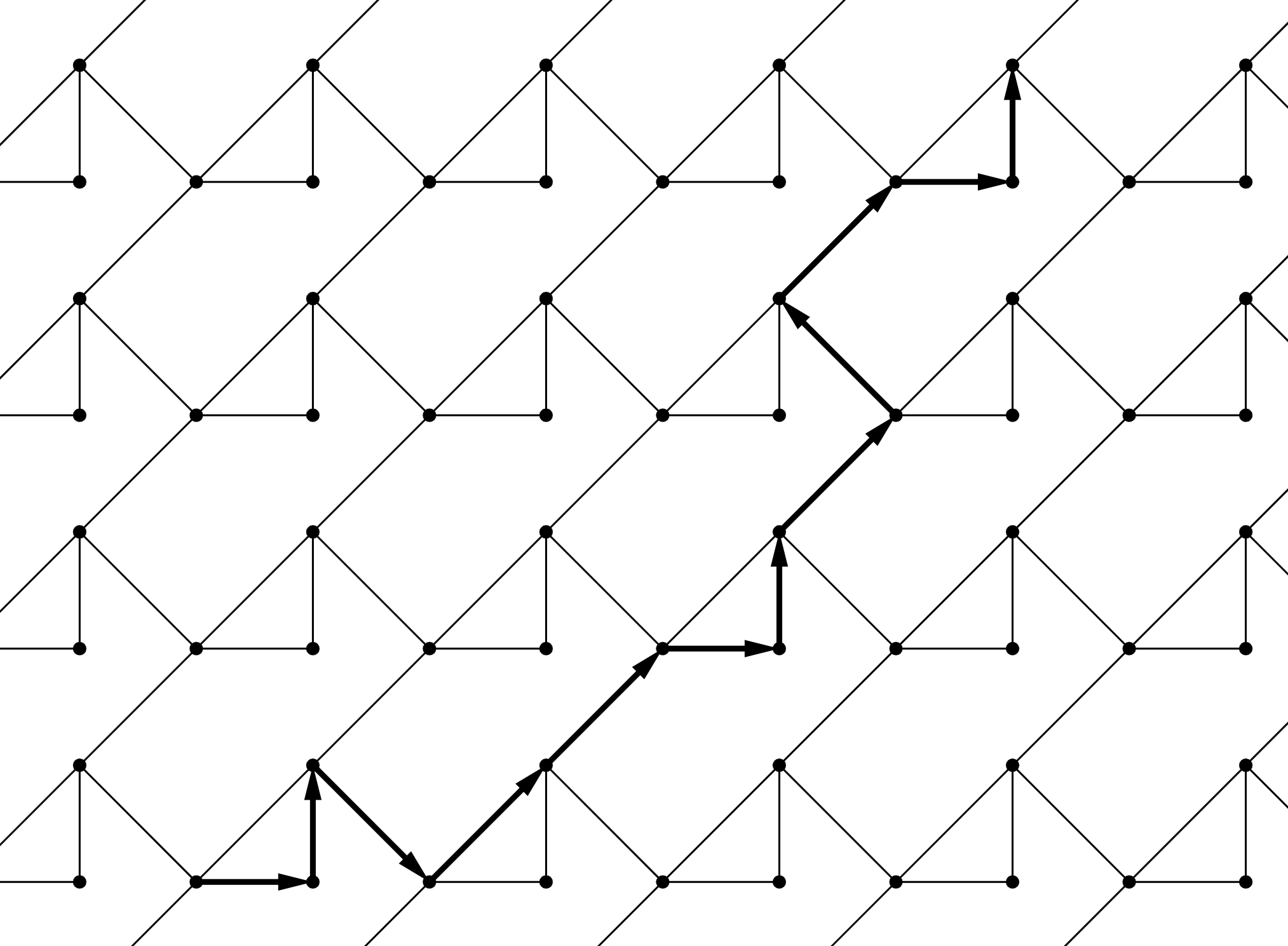}

\vspace{\zub}

\caption{The walk $p$.}
\label{fig:p}
\end{figure}
\end{ex}

\subsection{Growth sequences of periodic graphs}\label{subsection:CS}
Let $\Gamma$ be a locally finite graph, and let $x_0 \in V_{\Gamma}$. 
For $i \in \mathbb{Z}_{\ge 0}$, we define subsets $B_{\Gamma, x_0, i}, S_{\Gamma, x_0,i} \subset V _{\Gamma}$ by
\[
B_{\Gamma, x_0,i} := \{ y \in V _{\Gamma} \mid d_{\Gamma}(x_0, y) \le i \}, \quad
S_{\Gamma, x_0,i} := \{ y \in V _{\Gamma} \mid d_{\Gamma}(x_0, y) = i \}. 
\]
Let $b_{\Gamma, x_0,i} := \# B_{\Gamma, x_0,i}$ and $s_{\Gamma, x_0,i} := \# S_{\Gamma, x_0,i}$ denote their cardinalities. 
The sequence $(s_{\Gamma, x_0,i})_i$ is called the \textit{growth sequence} of $\Gamma$ with the start point $x_0$. The sequence $(b_{\Gamma, x_0,i})_i$ is called the \textit{cumulative growth sequence}. 

The \textit{growth series} $G_{\Gamma, x_0}(t)$ is the generating function 
\[
G_{\Gamma, x_0}(t) := \sum _{i \ge 0} s_{\Gamma, x_0,i} t^i
\]
of the growth sequence $(s_{\Gamma, x_0,i})_{i}$. 

\begin{rmk}
In crystallography, the growth sequence is called a \textit{coordination sequence} (see \cite{NSMN21}). 
\end{rmk}

\begin{defi}[{cf.\ \cite{Sta96}*{Chapter 0}}]\label{defi:qp}  \hfill
\begin{enumerate}
\item \label{item:f}
A function $f: \mathbb{Z} \to \mathbb{C}$ is called a \textit{quasi-polynomial} 
if there exist a positive integer $N$ and polynomials $Q_0, \ldots, Q_{N-1} \in \mathbb{C}[x]$ such that 
$f(n) = Q_i(n)$ holds for all $n \in \mathbb{Z}$ and $i \in \{ 0, \ldots , N-1 \}$ with $n \equiv i \pmod N$. 
The polynomials $Q_0, \ldots, Q_{N-1}$ are called the \textit{constituents} of $f$. 

\item 
A function $g: \mathbb{Z} \to \mathbb{C}$ is called to be of \textit{quasi-polynomial type} if 
there exists a non-negative integer $M \in \mathbb{Z}_{\ge 0}$ and a quasi-polynomial $f$ such that
$g(n) = f(n)$ holds for all $n > M$. 
The positive integer $N$ is called a \textit{quasi-period} of $g$ when $f$ is of the form in (\ref{item:f}). 
Note that the notion of quasi-period is not unique. 
The minimum quasi-period is called the \textit{period} of $g$. 
We say that the function $g$ is a \textit{quasi-polynomial on} $n \ge m$ if 
$g(n) = f(n)$ holds for $n \ge m$. 
\end{enumerate}
\end{defi}

The growth sequences of periodic graphs are known to be of quasi-polynomial type (Theorem \ref{thm:NSMN}). 
\begin{thm}[{\cite{NSMN21}*{Theorem 2.2}}]\label{thm:NSMN}
Let $(\Gamma, L)$ be a periodic graph, and let $x_0 \in V_{\Gamma}$. 
Then, the functions $b: i \mapsto b_{\Gamma, x_0, i}$ and $s: i \mapsto s_{\Gamma, x_0, i}$ are of quasi-polynomial type. 
In particular, its growth series is a rational function. 
\end{thm}

\noindent
In \cite{NSMN21}, Theorem \ref{thm:NSMN} is proved for unweighted periodic graphs, 
and the same proof also works for weighted periodic graphs.

\begin{ex}\label{ex:CS}
One can show that the growth sequence of the Wakatsuki graph (see Example \ref{ex:WG}) with the start point $v'_0$ is given by 
$s_{\Gamma, v'_0, 0}=1$ and
\[
s_{\Gamma, v'_0, n} = \begin{cases}
\frac{9}{2}n - 1 & (n \equiv 0 \mod 2) \\
\frac{9}{2}n - \frac{1}{2} & (n \equiv 1 \mod 2) 
\end{cases} 
\]
for $n \ge 1$. 
The growth sequence is exactly the same when the start point is $v'_1$. 

When the start point is $v' _2$, the growth sequence is given by 
$s_{\Gamma, v'_2, 0}=1$, $s_{\Gamma, v'_2, 1}=2$, $s_{\Gamma, v'_2, 2}=4$ and
\[
s_{\Gamma, v'_2, n} = \begin{cases}
3n & (n \equiv 0 \mod 2) \\
6n - 6 & (n \equiv 1 \mod 2) 
\end{cases} 
\]
for $n \ge 3$. 
\end{ex}

\subsection{Growth polytope (periodic graphs $\to$ polytopes)} \label{subsection:GP}
In this subsection, we define the \textit{growth polytope} $P_{\Gamma} \subset L_{\mathbb{R}}$ for a periodic graph $(\Gamma, L)$. 
The concept of a growth polytope has been defined and studied in various contexts \cites{KS02, KS06, Zhu02, Eon04, MS11, Fri13, ACIK}. 
This is helpful in understanding the asymptotic behavior of the growth sequence via convex geometry (see Theorem \ref{thm:asymp}). 

\begin{defi}\label{defi:P}
Let $(\Gamma , L)$ be an $n$-dimensional periodic graph. 
\begin{enumerate}
\item 
We define the \textit{normalization map} 
$\nu : \operatorname{Cyc}_{\Gamma /L} \to L_{\mathbb{R}} := L \otimes _{\mathbb{Z}} \mathbb{R}$ by 
\[
\nu: \operatorname{Cyc}_{\Gamma /L} \to L_{\mathbb{R}}; \quad p \mapsto \frac{\mu(\langle p \rangle )}{w(p)}. 
\]
We define the \textit{growth polytope}
\[
P_{\Gamma} := \operatorname{conv} \bigl( \operatorname{Im}(\nu) \cup \{ 0 \} \bigr) \subset L_{\mathbb{R}}
\]
as the convex hull of the set 
$\operatorname{Im}(\nu) \cup \{ 0 \} \subset L_{\mathbb{R}}$. 
Note that $\operatorname{Cyc}_{\Gamma /L}$ is a finite set. 
Furthermore, we have $\operatorname{Im}(\nu) \subset L_{\mathbb{Q}}$ since $w(p) \in \mathbb{Z}_{>0}$ and $\mu(\langle p \rangle ) \in L$ (cf.\ Definition \ref{defi:mu}). 
Therefore, $P_{\Gamma}$ is a rational polytope 
(i.e., $P_{\Gamma}$ is a polytope whose vertices are on $L_{\mathbb{Q}} := L \otimes _{\mathbb{Z}} \mathbb{Q}$). 
When $\Gamma$ is strongly connected, we have $0 \in \operatorname{int}(P_{\Gamma})$ by Lemma \ref{lem:int}. 
 
\item 
For a polytope $Q \subset L_{\mathbb{R}}$ and $y \in L_{\mathbb{R}}$, we define 
\[
d_{Q}(y) := \inf \{ t \in \mathbb{R}_{\ge 0} \mid y \in t Q \} \in \mathbb{R}_{\ge 0} \cup \{ \infty \}. 
\]
When $0 \in \operatorname{int}(Q)$, 
we have $d_{Q}(y) < \infty$ for any $y \in L_{\mathbb{R}}$. 

\item
For a periodic realization $\Phi:V_{\Gamma} \to L_{\mathbb{R}}$, we define 
\[
d_{P_{\Gamma}, \Phi}(x, y) := d_{P_{\Gamma}}\bigl( \Phi(y) - \Phi(x) \bigr)
\]
for $x, y \in V_{\Gamma}$. 
\end{enumerate}
\end{defi}

\begin{rmk}
In this paper, we assume that the weight function $w$ takes integer values (see Subsection \ref{subsection:GandW}). 
This assumption is used for $P_{\Gamma}$ to be a rational polytope. 
\end{rmk}

Next, we define the notation ``a vertex is $P$-initial" as follows. 
As far as we know, this concept was first considered by Shutov and Maleev in \cite{SM19}.

\begin{defi}
Let $(\Gamma , L)$ be a periodic graph. 
A vertex $y \in V_{\Gamma}$ is called \textit{$P$-initial} if the following condition holds: 
\begin{itemize}
\item For any vertex $u \in V(P_{\Gamma})\setminus \{ 0 \}$, 
there exists a cycle $p_u \in \operatorname{Cyc}_{\Gamma /L}$ such that 
$\nu (p_u) = u$ and $\overline{y} \in \operatorname{supp}(p_u)$. 
\end{itemize}
\end{defi}

\begin{ex}
For the Wakatsuki graph (see Example \ref{ex:WG}), $\operatorname{Im}(\nu)$ can be illustrated as in Figure \ref{fig:Im}. Here, the numbers written beside each point are the possible lengths of the cycles that give that point. 
For example, the cycle $q_1 := \overline{e'_5}\ \overline{e'_0}$ gives a 
point $\frac{\mu (\langle q_1 \rangle)}{w(q_1)} = \frac{1}{2}(1,0)$. 
The cycle $q_2 := \overline{e'_6}\ \overline{e'_3}\ \overline{e'_9}$ gives a 
point $\frac{\mu (\langle q_2 \rangle)}{w(q_2)} = \frac{1}{3}(1,1)$. 
In this case, the growth polytope $P_{\Gamma}$ is a hexagon. 

Furthermore, we can see that $v' _0$ and $v' _1$ are $P$-initial, but $v' _2$ is not. 
The six paths from $v'_0$ in Figure \ref{fig:6paths} give the six vertices of $P_{\Gamma}$, which shows that $v'_0$ is $P$-initial. 
On the other hand, there are no paths from $v'_2$ to $v'_2 + (0,1)$ of length two. 
This causes $v' _2$ not to be $P$-initial. 

\begin{figure}[h]
\begin{tabular}{cc}
\begin{minipage}{0.49\hsize}
\centering
\includegraphics[height=\zua]{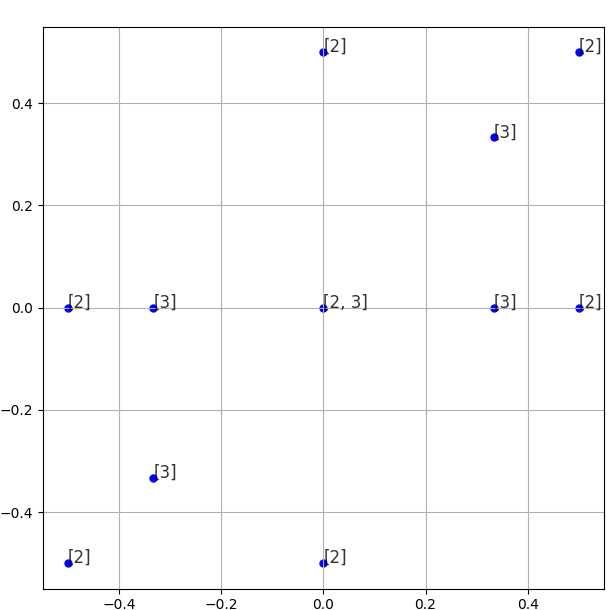}

\vspace{\zub}

\caption{$\operatorname{Im}(\nu)$.}
\label{fig:Im}
\end{minipage}&
\begin{minipage}{0.49\hsize}
\centering
\includegraphics[height=\zua]{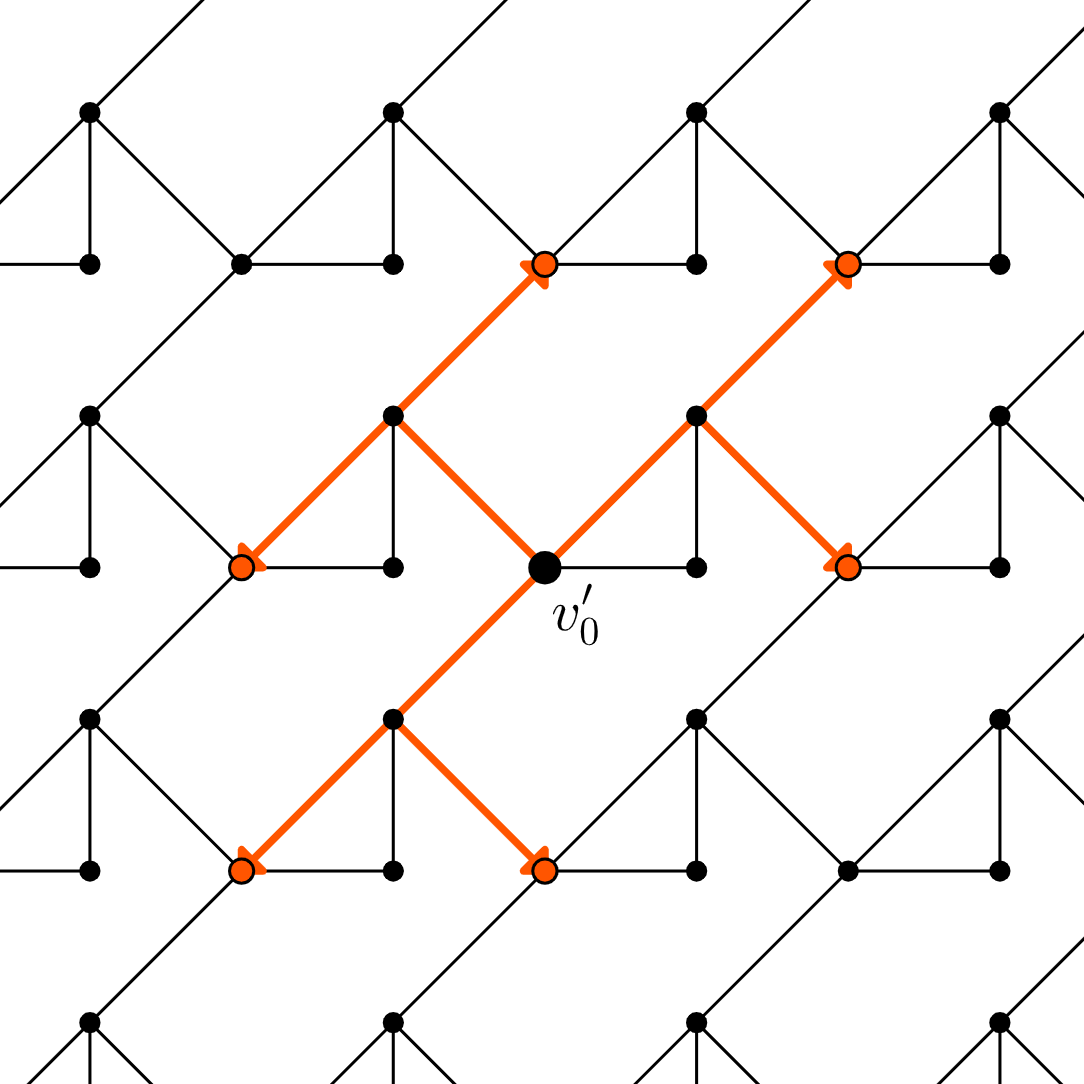}

\vspace{\zub}

\caption{Paths giving $V(P_{\Gamma})$.}
\label{fig:6paths}
\end{minipage}
\end{tabular}
\end{figure}
\end{ex}

\begin{lem}\label{lem:dPd}
Let $(\Gamma , L)$ be a strongly connected periodic graph, and let $x_0 \in V_{\Gamma}$. 
Then we have 
\[
d_{P_{\Gamma}}(y - x_0) \le d_{\Gamma}(x_0, y)
\]
for any $y \in V_{\Gamma}$ satisfying $\overline{y} = \overline{x_0}$. 
\end{lem}
\begin{proof}
Let $p$ be a walk in $\Gamma$ from $x_0$ to $y$ satisfying $w(p) = d_{\Gamma}(x_0, y)$. 
By Lemma \ref{lem:bunkai}, $\overline{p}$ decomposes to a walkable sequence $(q_0, q_1, \ldots , q_{\ell})$ with $q_0 = \emptyset _{\overline{x_0}}$ such that 
$\langle \overline{p} \rangle = \sum _{i = 1} ^{\ell} \langle q_i \rangle$. 
Then, we have 
\[
y-x_0 
= \sum _{i=1} ^{\ell} \mu(\langle q_i \rangle) 
\in \sum _{i = 1} ^{\ell} w(q_i) \cdot P_{\Gamma} 
= w(p) \cdot P_{\Gamma}, 
\]
which proves the desired inequality. 
\end{proof}

We define $C_{1}(\Gamma, \Phi, x_0)$ and $C_{2}(\Gamma, \Phi, x_0)$ as invariants that measure the difference between $d_{\Gamma}$ and $d_{P_{\Gamma}, \Phi}$. 

\begin{defi}
Let $(\Gamma , L)$ be a strongly connected periodic graph. 
Let $\Phi : V_{\Gamma} \to L_{\mathbb{R}}$ be a periodic realization, and let $x_0 \in V_{\Gamma}$. 
Then, we define 
\begin{align*}
C_{1}(\Gamma, \Phi, x_0) &:= \sup _{y \in V_{\Gamma}} \bigl( d_{P_{\Gamma}, \Phi} (x_0, y) - d_{\Gamma} (x_0, y) \bigr), \\
C_{2}(\Gamma, \Phi, x_0) &:= \sup _{y \in V_{\Gamma}} \bigl( d_{\Gamma}(x_0, y) - d_{P_{\Gamma}, \Phi} ( x_0, y ) \bigr). 
\end{align*}
By Theorem \ref{thm:asymp}, we have $C_{1}(\Gamma, \Phi, x_0) < \infty$ and $C_{2}(\Gamma, \Phi, x_0) < \infty$. 
\end{defi}

\begin{rmk}\label{rmk:C1}
\begin{enumerate}
\item
By the proof of Theorem \ref{thm:asymp}, we have 
\[
C_{1}(\Gamma, \Phi, x_0) = 
\max _{y \in B'_{c-1}} \bigl( d_{P_{\Gamma}, \Phi} (x_0, y) - d_{\Gamma} (x_0, y) \bigr), 
\]
where $c := \# (V_{\Gamma} /L)$ and 
\[
B'_{c-1} := \{ y \in V_{\Gamma} \mid \text{there exists a walk $p$ from $x_0$ to $y$ with $\operatorname{length}(p) \le c-1$}\}. 
\]

\item
It is not so easy to calculate $C_2(\Gamma, \Phi, x_0)$ in general. 
In Proposition \ref{prop:C2}, we will discuss a way of the calculation of $C_2(\Gamma, \Phi, x_0)$ when $x_0$ is $P$-initial. 
\end{enumerate}
\end{rmk}

\begin{lem}\label{lem:C1C2}
Let $(\Gamma , L)$ be a strongly connected periodic graph. 
Let $\Phi : V_{\Gamma} \to L_{\mathbb{R}}$ be a periodic realization, and let $x_0 \in V_{\Gamma}$. 
\begin{enumerate}
\item We have $C_{1}(\Gamma, \Phi, x_0) \ge 0$ and $C_{2}(\Gamma, \Phi, x_0) \ge 0$. 
\item Suppose $\# (V_{\Gamma} /L) = 1$. 
Then, any $y \in V_{\Gamma}$ is $P$-initial. Furthermore, we have $C_{1}(\Gamma, \Phi, x_0) = 0$. 
\item If $C_{2}(\Gamma, \Phi, x_0) < 1$, then $x_0$ is $P$-initial. 
\end{enumerate}
\end{lem}
\begin{proof}
(1) can be easily shown by setting $y = x_0$ in the definition of $C_{1}(\Gamma, \Phi, x_0)$ and $C_{2}(\Gamma, \Phi, x_0)$. 
The first assertion of (2) immediately follows from the definition of being $P$-initial. 
The second assertion of (2) follows from Lemma \ref{lem:dPd}. 

We prove (3). 
Suppose $C_{2}(\Gamma, \Phi, x_0) < 1$. 
Let $u \in V(P_{\Gamma})$. 
We take $d \in \mathbb{Z}_{>0}$ such that $du \in L$. 
Then, by Lemma \ref{lem:dPd}, we have 
\[
d_{\Gamma}(x_0, du + x_0) \ge d_{P_{\Gamma}}(du) = d \in \mathbb{Z}_{>0}. 
\]
Since $C_{2}(\Gamma, \Phi, x_0) < 1$, we have $d_{\Gamma}(x_0, du + x_0) = d$. 
Here, we have used the fact that the weight function $w _{\Gamma}$ is defined to be integral. 
Therefore, there exists a walk $p$ in $\Gamma$ from $x_0$ to $du + x_0$ such that $w(p) = d$. 
By Lemma \ref{lem:bunkai}(1), $\overline{p}$ decomposes to a walkable sequence $(q_0, q_1, \ldots , q_{\ell})$ with $q_0 = \emptyset _{\overline{x_0}}$ such that 
$\langle \overline{p} \rangle = \sum _{i = 1} ^{\ell} \langle q_i \rangle$. 
Then, we have 
\[
\frac{\sum _{i=1}^{\ell} \mu(\langle q_i \rangle)}{\sum _{i=1}^{\ell} w(q_i)} 
= \frac{\mu (\langle \overline{p} \rangle)}{d} = u \in V(P_{\Gamma}). 
\]
Here, $\frac{\sum _{i=1}^{\ell} \mu(\langle q_i \rangle)}{\sum _{i=1}^{\ell} w(q_i)}$ is a convex combination of 
$\frac{\mu(\langle q_1 \rangle)}{w(q_1)}, \ldots, \frac{\mu(\langle q_{\ell} \rangle)}{w(q_{\ell})} \in P_{\Gamma}$. 
Since $\frac{\sum _{i=1}^{\ell} \mu(\langle q_i \rangle)}{\sum _{i=1}^{\ell} w(q_i)} = u$ is a vertex of $P_{\Gamma}$, 
we conclude that $\frac{\mu (\langle q_i \rangle)}{w(q_i)} = u$ for any $1 \le i \le \ell$. 
By Lemma \ref{lem:bunkai}(2), $\overline{x_0} \in \operatorname{supp}(q_i)$ holds for some $1 \le i \le \ell$. 
Therefore, we can conclude that $x_0$ is $P$-initial. 
\end{proof}

\section{Ehrhart theory on periodic graphs}\label{section:ETP}
In Subsection \ref{subsection:EG}, we treat a class of periodic graphs for which Ehrhart theory can be applied. 
More precisely, in Theorem \ref{thm:ET}(1)(2), we see that the cumulative growth sequence $(b_{\Gamma, x_0, i})_i$ is a quasi-polynomial on $i \ge 0$ if a periodic realization $\Phi$ satisfies 
$C_1(\Gamma, \Phi, x_0) + C_2(\Gamma, \Phi, x_0) < 1$. 
Furthermore, in Theorem \ref{thm:ET}(3), we see that the growth series has the same reciprocity law as the Ehrhart series of reflexive polytopes if a periodic realization $\Phi$ satisfies both $C_1(\Gamma, \Phi, x_0)< \frac{1}{2}$ and $C_2(\Gamma, \Phi, x_0) < \frac{1}{2}$. 
Theorem \ref{thm:ET} can be seen as a generalization of a result of Conway and Sloane \cite{CS97}, where they treat the contact graphs of lattices (see Remark \ref{rmk:CS97}). 
In the proof of Theorem \ref{thm:ET}, we essentially use a variant of Ehrhart theory that is proved in Appendix \ref{section:ET}. 

In Subsection \ref{subsection:PtoPG}, we construct periodic graphs $(\Gamma _Q, L)$ from rational polytopes $Q$. By this construction, Theorem \ref{thm:ET} can be seen as a generalization of the Ehrhart theory for polytopes $Q$ with $0 \in Q$. 

\subsection{Ehrhart graphs}\label{subsection:EG}
\begin{defi}
Let $(\Gamma , L)$ be a strongly connected periodic graph, 
and let $\Phi : V_{\Gamma} \to L_{\mathbb{R}}$ be a periodic realization. 
Let $x_0 \in V_{\Gamma}$ and let $\alpha \in \mathbb{R}$. 
The triple $(\Gamma, \Phi, x_0)$ is called to be \textit{$\alpha$-Ehrhart} if we have
\[
B_{\Gamma, x_0, i} = 
\{ y \in V_{\Gamma} \mid d_{P_{\Gamma}, \Phi} (x_0, y) \le i + \alpha \}
\]
for all $i \in \mathbb{Z}_{\ge 0}$. 

This condition is equivalent to the condition that 
\[
d_{P_{\Gamma}, \Phi} (x_0, y) - \alpha \le d_{\Gamma} (x_0, y)
< d_{P_{\Gamma}, \Phi} (x_0, y) + 1 - \alpha
\]
holds for all $y \in V_{\Gamma}$. 
\end{defi}

\begin{defi}
Let $(\Gamma , L)$ be a strongly connected $n$-dimensional periodic graph. 
Let $\Phi : V_{\Gamma} \to L_{\mathbb{R}}$ be a periodic realization, and let $x_0 \in V_{\Gamma}$. 
We say that \textit{$\Phi$ is symmetric with respect to $x_0$} if 
$\#(\Phi ^{-1}(y)) = \#(\Phi ^{-1}(y'))$ holds for all $y, y' \in L_{\mathbb{R}}$ satisfying $y' + y = 2 \Phi(x_0)$. 
\end{defi}

\begin{rmk}
When $\# (V_{\Gamma} /L) = 1$, there exists an essentially unique periodic realization $\Phi$ (unique up to translation), 
and this $\Phi$ is symmetric with respect to any vertex $x_0 \in V_{\Gamma}$. 
\end{rmk}

\begin{thm}\label{thm:ET}
Let $(\Gamma , L)$ be a strongly connected $n$-dimensional periodic graph. 
Let $\Phi : V_{\Gamma} \to L_{\mathbb{R}}$ be a periodic realization, and let $x_0 \in V_{\Gamma}$. 
Let $s_i := s_{\Gamma, x_0, i}$ and $b_i := b_{\Gamma, x_0, i}$ be 
the growth sequence and the cumulative growth sequence with the start point $x_0$. 
Let $G_s(t) := \sum _{i \ge 0} s_i t^i$ and $G_b(t) := \sum _{i \ge 0} b_i t^i$ be their generating functions. 
Set $C_1 := C_{1}(\Gamma, \Phi, x_0)$ and $C_2 := C_{2}(\Gamma, \Phi, x_0)$. 
\begin{enumerate}
\item 
Suppose that the triple $(\Gamma, \Phi, x_0)$ is $\alpha$-Ehrhart for some $\alpha \in [0,1)$. 
Then, the function $i \mapsto b_i$ is a quasi-polynomial on $i \ge 0$. 

\item 
Suppose $C_1 + C_2 < 1$. 
Then, $(\Gamma, \Phi, x_0)$ is $\alpha$-Ehrhart for any $\alpha \in [C_{1}, 1-C_{2})$. 

\item 
Suppose both $C_1 < \frac{1}{2}$ and $C_2 < \frac{1}{2}$. 
Suppose one of the following conditions holds: 
\begin{itemize}
\item[(i)] $\Gamma$ is undirected, or

\item[(ii)] $\Phi$ is symmetric with respect to $x_0$. 
\end{itemize}
Then, we have 
\[
G_b(1/t) = (-1)^{n+1} t G_b(t), \qquad 
G_s(1/t) = (-1)^n G_s(t). 
\]
In particular, we have 
\[
f_b(-i) = (-1)^n f_b(i-1)
\]
for any $i \in \mathbb{Z}$, and 
\[
f_s(-i) = (-1)^{n+1} f_s(i)
\]
for any $i \in \mathbb{Z} \setminus \{ 0 \}$, 
where $f_b$ and $f_s$ are the quasi-polynomials corresponding to the sequences $(b_i)_i$ and $(s_i)_i$. 
\end{enumerate}
\end{thm}
\begin{proof}
As in Appendix \ref{section:ET}, 
for a rational polytope $P \subset L_{\mathbb{R}}$, $v \in L_{\mathbb{R}}$, and $\beta \in \mathbb{R}$, 
we define a function $h_{P, v, \beta}: \mathbb{Z} \to \mathbb{Z}$ by
\[
h_{P,v,\beta}(i) := \# \bigl( (v+(i+\beta)P) \cap L \bigr), 
\]
and its generating function $H_{P,v,\beta}(t) = \sum _{i \in \mathbb{Z}} h_{P,v,\beta}(i)t^i$. 
We also define $\overset{\circ}{h}_{P,v,\beta}$ and $\overset{\circ}{H}_{P,v,\beta}$ similarly. 
Note that we have $\operatorname{relint}(P_{\Gamma}) = \operatorname{int}(P_{\Gamma})$ by the assumption that $\Gamma$ is strongly connected (see Lemma \ref{lem:int}).

Let $c := \#(V_{\Gamma}/L)$. Take $y_1, \ldots, y_c \in V_{\Gamma}$ such that 
$\{ \overline{y}_1, \ldots , \overline{y}_c \} = V_{\Gamma} /L$. 
Then, we have $V_{\Gamma} = \bigsqcup _{j = 1} ^c (y_j + L)$, and hence, 
\begin{align*}
B_{\Gamma, x_0,i}
&= \left \{ y \in V_{\Gamma} \ \middle | \  d_{P_{\Gamma}, \Phi} (x_0, y) \le i + \alpha \right \} \\
&= \Phi^{-1}\bigl( \Phi (x_0) + (i + \alpha)P_{\Gamma} \bigr) \\
&= \bigsqcup _{j = 1} ^c \left( \Phi^{-1}\bigl( \Phi (x_0) + (i + \alpha)P_{\Gamma} \bigr) \cap (y_j + L) \right) \\
&= \bigsqcup _{j = 1} ^c \bigl( y_j + \bigl \{ m \in L \ \big | \ \Phi(y_j) + m \in \Phi(x_0) +  (i + \alpha)P_{\Gamma} \bigr \} \bigr) \\
&= \bigsqcup _{j = 1} ^c \bigl( y_j + \bigl( \Phi (x_0) - \Phi(y_j) + (i + \alpha) P_{\Gamma} \bigr) \cap L \bigr). 
\end{align*}
Hence, we have 
\[
b_i = \sum _{j=1}^c \# \bigl( \bigl( \Phi (x_0) - \Phi(y_j) + (i + \alpha) P_{\Gamma} \bigr) \cap L \bigr) 
= \sum _{j = 1} ^c h_{P_{\Gamma}, \Phi (x_0) - \Phi(y_j), \alpha}(i). 
\]
Therefore, (1) follows from Theorem \ref{thm:Eh}(1). 

(2) follows from the definitions of $C_1$ and $C_2$. 

We prove (3). 
When the condition (ii) is satisfied, 
we have
\[
\# \bigl( \Phi^{-1}(\Phi(x_0) + a \cdot \operatorname{int}(P_{\Gamma})) \bigr) 
= \# \bigl( \Phi^{-1}(\Phi(x_0) + a \cdot \operatorname{int}(- P_{\Gamma})) \bigr)
\tag{$\heartsuit$}
\]
for any $a \in \mathbb{R}_{\ge 0}$. 
When the condition (i) is satisfied, we have $- P_{\Gamma} = P_{\Gamma}$, and the same assertion $(\heartsuit)$ holds. 

By (2), $(\Gamma, \Phi, x_0)$ is $\frac{1}{2}$-Ehrhart. 
Furthermore, $(\Gamma, \Phi, x_0)$ is $\alpha$-Ehrhart for $\alpha = \frac{1}{2} - \epsilon$ for sufficiently small $\epsilon > 0$. 
Therefore, we have both
\begin{align*}
B_{\Gamma, x_0, i} 
&= \Phi^{-1} \left( \Phi(x_0) + \left( i + \frac{1}{2} \right) P_{\Gamma} \right), \\
B_{\Gamma, x_0, i}
&= \Phi^{-1} \left( \Phi(x_0) + \left( i + \frac{1}{2} \right) \operatorname{int}(P_{\Gamma}) \right) 
\end{align*}
for any $i \in \mathbb{Z}$. 
Therefore, we have 
\begin{align*}
G_b(t^{-1}) 
&= \sum _{j = 1} ^c H_{P_{\Gamma}, \Phi(x_0) - \Phi(y_j), \frac{1}{2}}(t^{-1}) \\
&= (-1)^{n+1} \sum _{j = 1} ^c \overset{\circ}{H}_{P_{\Gamma}, - (\Phi(x_0) - \Phi(y_j)), - \frac{1}{2}}(t)\\
&= (-1)^{n+1} \sum _{j = 1} ^c \overset{\circ}{H}_{-P_{\Gamma}, \Phi(x_0) - \Phi(y_j), - \frac{1}{2}}(t)\\
&= (-1)^{n+1} t \sum _{j = 1} ^c \overset{\circ}{H}_{-P_{\Gamma}, \Phi(x_0) - \Phi(y_j), \frac{1}{2}}(t)\\
&= (-1)^{n+1} t \sum _{j = 1} ^c \overset{\circ}{H}_{P_{\Gamma}, \Phi(x_0) - \Phi(y_j), \frac{1}{2}}(t)\\
&= (-1)^{n+1} t G_b(t). 
\end{align*}
Here, the second equality follows from Theorem \ref{thm:Eh}(3), the third follows from Lemma \ref{lem:sym}, 
and the fifth follows from $(\heartsuit)$. 
Since $G_s(t) = (1-t)G_b(t)$, we have 
\begin{align*}
G_s(t^{-1}) 
&= (1-t^{-1})G_b(t^{-1}) = (-1)^{n+1}(1-t^{-1})tG_b(t) \\
&= (-1)^{n+1}(1-t^{-1})t(1-t)^{-1}G_s(t) = (-1)^nG_s(t). 
\end{align*}
Since $f_b$ is a quasi-polynomial, we have 
\[
\sum_{i \in \mathbb{Z}_{< 0}} f_b(i) t^i 
= - \sum_{i \in \mathbb{Z}_{\ge 0}} f_b(i) t^i
\]
as rational functions (cf.\ \cite{BR}*{Exercise 4.7}). 
Therefore, we have 
\begin{align*}
\sum_{i \in \mathbb{Z}_{> 0}} f_b(-i) t^{-i}
&= - \sum_{i \in \mathbb{Z}_{\ge 0}} f_b(i) t^i \\
&= - G_b(t) \\
&= (-1)^{n}t^{-1}G_b(t^{-1})\\
&= (-1)^{n}t^{-1}\sum_{i \in \mathbb{Z}_{\ge 0}} f_b(i) t^{-i}\\
&= (-1)^{n}\sum_{i \in \mathbb{Z}_{> 0}} f_b(i-1) t^{-i}
\end{align*}
Here, the second equality follows because the function $i \mapsto b_i$ is a quasi-polynomial on $i \ge 0$. 
By comparing the coefficients, we can conclude that $f_b(-i) = (-1)^n f_b(i-1)$ for any $i \in \mathbb{Z}_{> 0}$. 
The statement for $f_s(-i)$ follows from the same argument. 
\end{proof}

\begin{rmk}\label{rmk:ref1}
The reciprocity laws appearing in Theorem \ref{thm:ET}(3) are the same as the reciprocity laws of the Ehrhart series of reflexive polytopes 
(cf.\ \cite{BR}*{Section 4.4}). 
We will see in Remark \ref{rmk:ref2} that Theorem \ref{thm:ET}(3) 
can be seen as a generalization of the reciprocity laws of the Ehrhart series of reflexive polytopes. 
\end{rmk}

\begin{rmk}\label{rmk:CS97}
In this remark, we explain that
Theorem \ref{thm:ET} is a generalization of results of Conway and Sloane in \cite{CS97}, 
where only the case $\# (V_{\Gamma} /L) = 1$ is treated. 

In \cite{CS97}, Conway and Sloane study the growth sequence of the contact graph $\Gamma$ of 
an $n$-dimensional lattice $L$ in $\mathbb{R}^n$ that is spanned by its minimal vectors. 
More precisely, they considered the graphs obtained in the following way: 
\begin{itemize}
\item 
$L \subset \mathbb{R}^n$ is a lattice of rank $n$. 
Let $F$ be the set of all $v \in L \setminus \{ 0 \}$ such that its Euclidean norm $||v||$ is the smallest among $L \setminus \{ 0 \}$. 
Suppose that $L$ is spanned by $F$. 

\item 
Define a periodic graph $(\Gamma, L)$ by 
\begin{itemize}
\item $V_{\Gamma} := L$, $E_{\Gamma} := L \times F$, and 
\item for $e = (x,v) \in E_{\Gamma}$, we set 
\[
s(e) := x, \qquad t(e) := x+v, \qquad w(e) = 1. 
\]
\end{itemize}
\end{itemize}
Note that $\# (V_{\Gamma} /L) = 1$ in this case. 

Conway and Sloane define the ``contact polytope" $\mathcal{P}$ of $L$ as the convex hull of $F$, 
and they study the growth sequence of $\Gamma$ using Ehrhart theory on $\mathcal{P}$. 
Note that $\mathcal{P}$ coincides with the growth polytope $P_{\Gamma}$ in our notation. 
Since $\# (V_{\Gamma} /L) = 1$, we have $C_1 = 0$. 

Conway and Sloane also introduce the notations 
``well-placed", ``well-rounded" and ``well-coordinated" according to the property of $L$. 
The condition ``well-placed" coincides with the condition 
that the $\mathcal{P}$ is a reflexive polytope. 
The condition ``well-rounded" coincides with the condition ``$C_2 < 1$". 
$L$ is called ``well-coordinated" if $L$ is well-placed and well-rounded. 
They prove the following assertions (\cite{CS97}*{Theorems 2.5, 2.9}):  
\begin{itemize}
\item If $L$ is well-rounded, 
the growth sequence $(b_i)_i$ of $\Gamma$ is a polynomial on $i \ge 0$. 

\item If $L$ is well-coordinated, 
the growth sequence satisfies the reciprocity laws in Theorem \ref{thm:ET}(3). 
\end{itemize}
Therefore, Theorem \ref{thm:ET} can be seen as the generalization of these results to the case where $\# (V_{\Gamma} /L) > 1$. 
\end{rmk}

\subsection{Polytopes $\to$ periodic graphs}\label{subsection:PtoPG}
In this subsection, we define a periodic graph $\Gamma _Q$ from a rational polytope $Q$, and 
we see that the study of the growth sequences of periodic graphs can be essentially seen as 
a generalization of the Ehrhart theory of rational polytopes $Q$ satisfying $0 \in Q$. 

First, we define a periodic graph $\Gamma _Q$ for a rational polytope $Q$ (possibly $0 \not \in Q$). 
\begin{defi}
Let $Q \subset \mathbb{R}^N$ be a $d$-dimensional rational polytope. 
Let $a$ be the minimum positive integer such that $aQ$ is a lattice polytope. 
We define a graph $\Gamma _Q$ as follows: 
\begin{itemize}
\item 
$V_{\Gamma _Q} := \mathbb{Z}^N$. 

\item 
$E_{\Gamma _Q} := \mathbb{Z}^N \times F_Q$, 
where $F_Q := \left \{ (i, m) \in \mathbb{Z}_{>0} \times \mathbb{Z}^N \ \middle | \  i < a(d+1),\ m \in iQ  \right\}$. 

\item 
For $e = (x, (i,m)) \in E_{\Gamma _Q}$, we define 
\[
s(e) := x, \qquad t(e) := x+m, \qquad w(e) := i. 
\]
\end{itemize}
Then, $\left( \Gamma _Q, L \right)$ for $L := \mathbb{Z}^N$ becomes an $N$-dimensional periodic graph. 
Since $\# \left( V_{\Gamma _Q}/L \right) = 1$, there exists the unique realization $\Phi : V_{\Gamma _Q} \to L_{\mathbb{R}}$ such that $\Phi(0)=0$. We set $C_i := C_i(\Gamma _Q, \Phi, 0)$ for $i \in \{1, 2 \}$. 
\end{defi}

\begin{rmk}\label{rmk:GammaQ}
Let $i$ be a positive integer satisfying $i<a(d+1)$, and let $x, y \in V_{\Gamma _Q} = \mathbb{Z}^N$ be any two vertices. 
Then, the graph $\Gamma _Q$ is defined so that the following two conditions are equivalent: 
\begin{itemize}
\item
There exists an edge from $x$ to $y$ of weight $i$. 
 \item
$y - x \in i Q$.
\end{itemize}

Note that without the boundedness condition ``$i < a(d+1)$'' in the definition of $F_{Q}$, 
we could have $\# \left( E_{\Gamma _Q}/L \right) = \# F_Q = \infty$, and therefore, $\Gamma _Q$ could not be a periodic graph. 
This specific value ``$a(d+1)$'' will be used in the proof of Lemma \ref{lem:PG}(1) when applying Lemma \ref{lem:normalP}. 
\end{rmk}
\begin{ex}
Let $N = d = 2$, and let $Q = \operatorname{conv} \bigl( \bigl\{ (0, 0), (0, 1/2), (1/2, 0) \bigr\} \bigr)$. 
In this case, we have $a = 2$ and 
\begin{align*}
F_{Q} ={}
&\{ (i, (0,0)) \mid 1 \le i \le 5 \} 
\cup \{ (i,(1,0)) \mid 2 \le i \le 5 \} 
\cup \{ (i,(0,1)) \mid 2 \le i \le 5 \}\\
&\cup \{ (i,(1,1)) \mid i = 4,5 \} 
\cup \{ (i,(2,0)) \mid i = 4,5 \} 
\cup \{ (i,(0,2)) \mid i = 4,5 \}. 
\end{align*}
For example, for all $x \in V_{\Gamma _Q} = \mathbb{Z}^2$, there are four distinct edges from $x$ to $x + (1,0)$ of weights $2, 3, 4$ and $5$. 
\end{ex}

\begin{lem}\label{lem:PG}
Let $x \in \mathbb{Z}^N$. 
\begin{enumerate}
\item 
For any $i \in \mathbb{Z}_{\ge 0}$, 
the condition $x \in \bigcup _{0 \le j \le i} jQ$ is equivalent to the condition $d_{\Gamma _Q}(0,x) \le i$. 
In particular, the cumulative growth sequence $b_{\Gamma _Q, 0, i}$ coincides with 
\[
\# \left( \left( \bigcup _{0 \le j \le i} jQ \right) \cap \mathbb{Z}^N \right).
\]

\item
The growth polytope $P_{\Gamma _{Q}}$ coincides with $\operatorname{conv} (Q \cup \{ 0 \})$.

\item 
When $0 \in Q$, we have $b_{\Gamma _Q, 0, i} = \# \bigl( iQ \cap \mathbb{Z}^N \bigr)$.

\item 
If $0 \in Q$ and $d_Q(x) < \infty$, we have $d_{\Gamma _Q}(0, x) = \lceil d_Q(x) \rceil$. 

\item
The strong connectedness of $\Gamma _Q$ is equivalent to the condition that $0 \in \operatorname{int}(Q)$. 

\item 
If $0 \in \operatorname{int}(Q)$, we have 
\[
C_1 = 0, \qquad C_2 \in \{0\} \cup \left[ \frac{1}{2}, 1 \right). 
\]

\item
Suppose $0 \in \operatorname{int}(Q)$. 
Then, $C_2 = 0$ holds if and only if
\[
(i+1) \operatorname{int}(Q) \cap \mathbb{Z}^N = iQ \cap \mathbb{Z}^N
\]
holds for all $i \in \mathbb{Z}_{\ge 0}$. 
\end{enumerate}
\end{lem}
\begin{proof}
We prove (1). 
Let $i \in \mathbb{Z}_{\ge 0}$. 
First, we suppose that $d_{\Gamma _Q} (0,x) \le i$. 
Then, there exists a path $p = e_1 \cdots e_{\ell}$ of $\Gamma _Q$ from $0$ to $x$ with $w(p) \le i$. 
By the definition of $\Gamma _Q$ (cf.\ Remark \ref{rmk:GammaQ}), we have $t(e_i) - s(e_i) \in w(e_i) \cdot Q$ for all $1 \le i \le \ell$. 
Hence, we have 
\[
x = \sum _{i=1} ^{\ell} \bigl( t(e_i) - s(e_i) \bigr) \in \left( \sum _{i=1} ^{\ell} w(e_i) \right) \cdot Q = w(p) \cdot Q \subset \bigcup _{0 \le j \le i} jQ. 
\]
Next, we suppose that $x \in jQ$ for some $0 \le j \le i$. 
We set $b := \max \left \{ 0,  \left \lfloor \frac{j - a(d+1)}{a} \right \rfloor + 1 \right \}$. 
Then, by Lemma \ref{lem:normalP} below, we have 
\[
x \in jQ \cap \mathbb{Z}^N \subset \left( (j - ba)Q \cap \mathbb{Z}^N \right) + b \left ( aQ \cap \mathbb{Z}^N \right ). 
\]
Therefore, there exist $m_0 \in (j - ba)Q \cap \mathbb{Z}^N$ and $m_1, \ldots, m_b \in aQ \cap \mathbb{Z}^N$ such that 
$x = \sum_{k=0} ^b m_k$. 
For each $0 \le \ell \le b$, we set $x_{\ell} := \sum _{k=0} ^{\ell} m_k$. 
If $j - ba \not = 0$, there exists an edge from $0$ to $x_0$ of weight $j - ba$ by the definition of $\Gamma _Q$ (cf.\ Remark \ref{rmk:GammaQ}). 
Here, we have used the fact $j - ba < a(d+1)$ which follows from the choice of $b$. 
Therefore, we have $d_{\Gamma _Q} (0, x_0) \le j - ba$ (this is correct even if $j - ba = 0$). 
Similarly, we have $d_{\Gamma _Q}(x_{\ell}, x_{\ell +1}) \le a$ for all $0 \le \ell \le b-1$. 
Hence, we have 
\[
d_{\Gamma _Q}(0, x) \le  d_{\Gamma _Q} (0, x_0) + \sum _{\ell = 0} ^{b-1} d_{\Gamma _Q}(x_{\ell}, x_{\ell+1}) \le j - ba + ba = j \le i, 
\] 
which completes the proof of (1). 

We prove (2). By the definition of the growth polytope $P_{\Gamma _Q}$, we have 
\begin{align*}
P_{\Gamma _Q} 
&= \operatorname{conv} \left( \left \{ \frac{t(e) - s(e)}{w(e)} \ \middle | \  e \in E_{\Gamma _Q} \right \} \cup \{ 0 \} \right) \\
&= \operatorname{conv} \left( \left \{ \frac{m}{i} \ \middle | \  (i,m) \in F_Q \right \} \cup \{ 0 \} \right). 
\end{align*}
By the definition of $F_Q$, we have 
\[
\operatorname{conv} \left( \left \{ \frac{m}{i} \ \middle | \  (i,m) \in F_Q \right \} \right) = Q, 
\]
which completes the proof of (2).

(3) and (4) follow from (1) since we have $\bigcup _{0 \le j \le i} jQ = iQ$ when $0 \in Q$. 
(5) also follows from (1). 
(7) follows from (4). 

We shall see (6) below. 
The assertion $C_1 = 0$ follows from Lemma \ref{lem:C1C2}(2) (or directly from (4)). 
By (4), we have 
\[
C_2 = \sup _{x \in \mathbb{Z}^N} \bigl( \lceil d_Q(x) \rceil - d_Q(x) \bigr)
= \max _{x \in \mathbb{Z}^N} \bigl( \lceil d_Q(x) \rceil - d_Q(x) \bigr)
< 1. 
\]
If $\lceil d_Q(x) \rceil - d_Q(x) \in \left[ 0, \frac{1}{2} \right)$, we have 
$\lceil d_Q(2x) \rceil - d_Q(2x) = 2 (\lceil d_Q(x) \rceil - d_Q(x))$. 
Therefore, we can conclude $C_2 \not \in \left( 0, \frac{1}{2} \right)$. 
\end{proof}

\begin{lem}[{cf.\ \cite{CLS}*{Theorem 2.2.12}}]\label{lem:normalP}
For $k \in \mathbb{R}_{\ge a(d+1)}$, we have 
\[
kQ \cap \mathbb{Z}^N \subset \bigl( (k-a)Q \cap \mathbb{Z}^N \bigr) + \bigl( aQ \cap \mathbb{Z}^N \bigr). 
\]
\end{lem}
\begin{proof}
By taking a triangulation of $Q$, we may assume that $Q$ is a simplex. 
Let $v_0, \ldots , v_d \in \mathbb{Q}^N$ be its vertices. 
By the choice of $a$, we may write $v_i = m_i / a$ for some $m_i \in \mathbb{Z}^N$. 
Let $m \in kQ \cap \mathbb{Z}^N$. 
Then, we may uniquely write 
\[
m = \sum _{i=0} ^d \alpha_i v_i = \sum _{i=0} ^d \frac{\alpha _i}{a} \cdot m_i
\]
for some $\alpha _i \in \mathbb{R}_{\ge 0}$ satisfying $\sum _{i=0} ^d \alpha _i = k$. 
Since $\sum _{i=0} ^d \alpha _i = k \ge a(d+1)$, there exists $i \in \{ 0, 1, \ldots , d \}$ 
such that $\frac{\alpha _i}{a} \ge 1$. 
Then, for such $i$, we have 
\[
m = (m-m_i)+m_i \in \bigl( (k-a)Q \cap \mathbb{Z}^N \bigr) + \bigl( aQ \cap \mathbb{Z}^N \bigr). 
\]
We complete the proof. 
\end{proof}

\begin{rmk}\label{rmk:ref2}
For a lattice polytope $Q$, it is known that $Q$ is a reflexive polytope if and only if the condition 
\[
(i+1) \operatorname{int}(Q) \cap \mathbb{Z}^N = iQ \cap \mathbb{Z}^N
\]
holds for all $i \in \mathbb{Z}_{\ge 0}$ (cf.\ \cite{BR}*{Section 4.4}). 
Therefore, Theorem \ref{thm:ET}(3) can be seen as a generalization of the reciprocity laws of the Ehrhart series of reflexive polytopes. 
\end{rmk}

\section{$P$-initial vertex and well-arranged graphs}\label{section:PiniWA}
\subsection{The $P$-initial case}
In this subsection, we treat a periodic graph $(\Gamma, L)$ and a $P$-initial vertex $x_0 \in V_{\Gamma}$. 
In this case, we can calculate the invariant $C_2$ (Proposition \ref{prop:C2}) and 
a quasi-period of the growth sequence of $\Gamma$ with the start point $x_0$ (Theorem \ref{thm:Piniqp}).

The following lemma will be used in Theorem \ref{thm:Piniqp}. 
\begin{lem}\label{lem:fg}
Let $(\Gamma, L)$ be a periodic graph, and let $x_0 \in V_{\Gamma}$. Suppose that $x_0$ is $P$-initial. 
For each $v \in V(P_{\Gamma}) \setminus \{ 0 \}$, we pick a cycle $q_v \in \nu^{-1}(v)$ such that 
$\overline{x_0} \in \operatorname{supp}(q_v)$. 
We define
\[
B := \{ (i, y) \in \mathbb{Z}_{\ge 0} \times V_{\Gamma} \mid d_{\Gamma}(x_0, y) \le i \} \subset \mathbb{Z}_{\ge 0} \times V_{\Gamma}. 
\]
We define a subset $M' \subset \mathbb{Z}_{\ge 0} \times L$ by 
\[
M' := \bigl\{ \bigl( w(q_v), \mu(\langle q_v \rangle) \bigr) \ \big | \ v \in V(P_{\Gamma}) \setminus \{ 0 \} \bigr \}. 
\]
Let $M \subset \mathbb{Z}_{\ge 0} \times L$ be the submonoid generated by $M'$ and $(1,0)$. 

Then, $B$ is a finitely generated $M$-module. 
\end{lem}

\begin{proof}
First, we prove that $B$ is an $M$-module (i.e.\ $M + B \subset B$). 
Take $(i,y) \in B$ and $v \in V(P_{\Gamma}) \setminus \{0\}$. 
Then, by the definition of $B$, there exists a walk $p$ in $\Gamma$ from $x_0$ to $y$ satisfying $w(p) \le i$. 
By the choice of $q_v$, we have $\overline{x_0} \in \operatorname{supp}(q_v)$. 
In particular, we have 
$\operatorname{supp} (\overline{p}) \cap \operatorname{supp}(q_v) \not = \emptyset$.
Therefore, there exists a path $p'$ in $\Gamma$ from $x_0$ such that 
\begin{align*}
\langle \overline{p'} \rangle 
= \langle q_v \rangle + \langle \overline{p} \rangle. 
\end{align*}
Then, we have 
\[
t(p') 
= \mu(\langle q_v \rangle) + t(p)
= \mu(\langle q_v \rangle) + y. 
\]
Furthermore, we have 
\[
w(p') = w(q_v) + w(p) \le w(q_v) + i. 
\]
They show that $\bigl( w(q_v) + i, \mu(\langle q_v \rangle) + y \bigr) \in B$. 
Hence, we can conclude that $M' + B \subset B$. 
Since it is clear that $(1,0) + B \subset B$, we conclude that $M + B \subset B$. 

Next, we prove that the $M$-module $B$ is generated by some finite subset $B' \subset B$. 
For each $q \in \operatorname{Cyc}_{\Gamma /L}$, we take a positive integer $d_q$ with the following condition: 
\begin{itemize}
\item 
Let $\operatorname{Facet}' (P_{\Gamma})$ be the set of 
$\sigma \in \operatorname{Facet}(P_{\Gamma})$ satisfying $0 \not \in \sigma$. 
First, for each $\sigma \in \operatorname{Facet}'(P_{\Gamma})$, 
we fix a triangulation $T_{\sigma}$ of $\sigma$ such that $V(\Delta) \subset V(\sigma)$ holds for all $\Delta \in T_{\sigma}$. 

\item
For each $q \in \operatorname{Cyc}_{\Gamma /L}$, we take $\sigma \in \operatorname{Facet}'(P_{\Gamma})$ 
and $\Delta \in T_{\sigma}$ such that $\nu(q) \in \mathbb{R}_{\ge 0} \Delta$. 
Then, we take a positive integer $d_q$ such that 
\[
d_q \cdot \mu (\langle q \rangle) 
= \sum _{v \in V(\Delta)} b_v \cdot \mu (\langle q_v \rangle)
\]
holds for some $b_v \in \mathbb{Z}_{\ge 0}$. 
\end{itemize}
We note that for $q \in \operatorname{Cyc}_{\Gamma /L}$, $\Delta$ and $b_v$'s above, we have 
\[
d_q \cdot w(q) 
\ge \sum _{v \in V(\Delta)} b_v \cdot w(q_v)
\]
since $\frac{\mu(\langle q \rangle)}{w(q)} = \nu(q) \in [0,1] \cdot \Delta$. 
Therefore, we have 
\[
d_q \cdot \bigl( w(q), \mu (\langle q \rangle) \bigr) \in M. 
\]

We shall show that the $M$-module $B$ is generated by 
\[
B' := \left \{ (i,y) \in B \ \middle | \ i \le W \cdot \bigl( \# (V_{\Gamma}/L) \bigr)^2 
    + \sum _{q \in \operatorname{Cyc}_{\Gamma /L}} (d_q - 1) \cdot w(q) \right\}, 
\]
where $W := \max _{e \in E_{\Gamma}} w(e)$. 

Take $(i,y) \in B$. 
Then, by the definition of $B$, there exists a walk $p$ in $\Gamma$ from $x_0$ to $y$ satisfying $w(p) \le i$. 
By decomposing $\overline{p}$ (Lemma \ref{lem:bunkai}(1)), 
there exists a walkable sequence $(q_0, q_1, \ldots, q_{\ell})$ such that 
$\langle \overline{p} \rangle = \sum _{i=0} ^{\ell} \langle q_i \rangle$. 
By Lemma \ref{lem:bunkai}(2), by rearranging the indices of $q_1, \ldots, q_{\ell}$, 
we may assume the following condition for each $0 \le j \le \ell -1$: 
\begin{itemize}
\item 
$\left( \bigcup _{0 \le i \le j} \operatorname{supp}(q_i) \right) \cap \operatorname{supp}(q_{j+1}) \not = \emptyset$. 
\end{itemize}
Furthermore, we may also assume the following condition for each $0 \le j \le \ell -1$:
\begin{itemize}
\item
If $\bigcup _{0 \le i \le j} \operatorname{supp}(q_i) \not = \operatorname{supp}(\overline{p})$, then 
$\operatorname{supp}(q_{j+1}) \not \subset \bigcup _{0 \le i \le j} \operatorname{supp}(q_i)$. 
\end{itemize}
In particular, for $\ell' := \#(\operatorname{supp}(\overline{p})) - \#(\operatorname{supp}(q_0)) 
\le \# (V_{\Gamma}/L) - \operatorname{length}(q_0) - 1$, 
it follows that 
\begin{itemize}
\item $(q_0, q_1, \ldots, q_{\ell'})$ is a walkable sequence, and 
\item $\bigcup _{0 \le i \le \ell'} \operatorname{supp}(q_i) = \operatorname{supp}(\overline{p})$. 
\end{itemize}

For each $q \in \operatorname{Cyc}_{\Gamma /L}$, we define $\alpha _q \in \mathbb{Z}_{\ge 0}$ by 
\[
\alpha _q := \# \{ \ell' + 1 \le i \le \ell \mid q_i = q \}. 
\]
Let $\beta _q \in \mathbb{Z}_{\ge 0}$ be the integer satisfying 
$0 \le \beta _q < d_q$ and $\beta _q \equiv \alpha _q \pmod{d_q}$. 
We set $\ell '' := \ell ' + \sum _{q \in \operatorname{Cyc}_{\Gamma /L}} \beta _q$. 
Then, by rearranging the indices of $q_{\ell ' +1}, \ldots, q_{\ell}$, 
we may assume the following condition
\[
\# \{ \ell' + 1 \le i \le \ell '' \mid q_i = q \} = \beta _q. 
\]
Since $\operatorname{supp}(q_i) \subset \operatorname{supp}(\overline{p})$ for any
$1 \le i \le \ell$, the sequence $(q_0, q_1, \ldots , q_{\ell ''})$ is also a walkable sequence. 
Furthermore, since
\begin{align*}
\sum _{i=0} ^{\ell'} \operatorname{length}(q_i)
&= \operatorname{length}(q_0) + \sum _{i=1} ^{\ell'} \operatorname{length}(q_i) \\
&\le \operatorname{length}(q_0) + \ell ' \cdot \# (V_{\Gamma}/L) \\
&\le \operatorname{length}(q_0) + (\# (V_{\Gamma}/L) - \operatorname{length}(q_0) - 1) \cdot \# (V_{\Gamma}/L) \\
&\le (\# (V_{\Gamma}/L))^2, 
\end{align*}
we have 
\[
\sum _{i=0} ^{\ell'} w(q_i) \le W \cdot (\# (V_{\Gamma}/L))^2. 
\]
We also have 
\[
\sum _{i=\ell ' +1} ^{\ell ''} w(q_i) 
= \sum _{q \in \operatorname{Cyc}_{\Gamma /L}} \beta _q \cdot w(q)
\le \sum _{q \in \operatorname{Cyc}_{\Gamma /L}} (d_q - 1) \cdot w(q)
\]

Since $(q_0, q_1, \ldots , q_{\ell ''})$ is a walkable sequence, there exists 
a path $p'$ in $\Gamma$ from $x_0$ such that 
$\langle \overline{p'} \rangle = \sum _{i=0} ^{\ell''} \langle q_i \rangle$. 
Then, we have 
\begin{align*}
w(p') 
&= \sum _{i=0} ^{\ell'} w(q_i) + \sum _{i=\ell ' +1} ^{\ell ''} w(q_i) \\
& \le W \cdot \bigl( \# (V_{\Gamma}/L) \bigr)^2 
    + \sum _{q \in \operatorname{Cyc}_{\Gamma /L}} (d_q - 1) \cdot w(q), 
\end{align*}
and hence, $(w(p'), t(p')) \in B'$. 
Since 
\begin{align*}
&i - w(p') 
\ge w(p) - w(p') 
= \sum _{i=\ell'' +1} ^{\ell} w(q_i)
= \sum _{q \in \operatorname{Cyc}_{\Gamma /L}} (\alpha _q - \beta _q) \cdot w(q), \\
&y-t(p')=t(p)-t(p') 
= \sum _{i=\ell'' +1} ^{\ell} \mu( \langle q_i \rangle)
= \sum _{q \in \operatorname{Cyc}_{\Gamma /L}} (\alpha _q - \beta _q) \cdot \mu( \langle q \rangle), 
\end{align*}
and $d_q \mid (\alpha _q - \beta _q)$ for each $q \in \operatorname{Cyc}_{\Gamma /L}$, 
we have 
\[
(i,y) - \bigl( w(p'), t(p') \bigr) 
= \bigl( i-w(p) ,0 \bigr ) 
+ \sum _{q \in \operatorname{Cyc}_{\Gamma /L}} \frac{\alpha _q - \beta _q}{d_q} 
\cdot d_q \cdot \bigl( w(q), \mu (\langle q \rangle) \bigr)
\in M. 
\]
Therefore, we have $B = M + B'$. 
Since $B'$ is a finite set, we can conclude that 
$B$ is a finitely generated $M$-module. 
\end{proof}

\begin{thm}\label{thm:Piniqp}
Let $(\Gamma, L)$ be a periodic graph, and let $x_0 \in V_{\Gamma}$. Suppose that $x_0$ is $P$-initial. 
For each $v \in V(P_{\Gamma}) \setminus \{ 0 \}$, we pick a cycle $q_v \in \nu^{-1}(v)$ such that 
$\overline{x_0} \in \operatorname{supp}(q_v)$. 
Then, 
\[
\operatorname{LCM} \{ w(q_v) \mid v \in V(P_{\Gamma}) \setminus \{0\} \}
\]
is a quasi-period of the growth sequence $(s_{\Gamma, x_0, d})_d$. 
More precisely, the growth series $G_{\Gamma, x_0}(t)$ is of the form
\[
G_{\Gamma, x_0}(t) = \frac{Q(t)}{\prod _{v \in V(P_{\Gamma}) \setminus \{ 0 \}} (1-t^{w(q_v)})}
\]
with some polynomial $Q(t)$. 

In particular, if the graph $\Gamma$ is unweighted, 
then $\operatorname{LCM} \left \{ 1, 2, \ldots, \#(V_{\Gamma}/L) \right \}$ is 
a quasi-period of the growth sequence. 
\end{thm}

\begin{proof}
We keep the notation $B$, $M'$ and $M$ in Lemma \ref{lem:fg}. 
For $i \in \mathbb{Z}_{\ge 0}$, 
we define $B_i := \{ y \in V_{\Gamma} \mid (i,y) \in B \}$. 
In this notation, we have $b_{\Gamma, x_0, d} = \# B_d$. 

By Lemma \ref{lem:fg}, $B$ is a finitely generated $M$-module. 
Furthermore, the monoid $M$ is generated by the finite set $M' \cup \{ (1,0) \}$, 
and the degree of each element of $M' \cup \{ (1,0) \}$ divides 
$\operatorname{LCM} \{ w(q_v) \mid v \in V(P_{\Gamma}) \setminus \{0\} \}$. 
Therefore, the assertion follows from Theorem \ref{thm:Serre}.   
\end{proof}

\noindent
The following theorem is well-known. 
\begin{thm}[{cf.\ \cite{BG2009}*{Theorem 6.38}}]\label{thm:Serre}
Let $N$ be a monoid. 
Let $M' \subset \mathbb{Z}_{> 0} \times N$ be a finite subset, and 
let $M \subset \mathbb{Z}_{\ge 0} \times N$ be the submonoid generated by $M'$. 
Let $X \subset \mathbb{Z}_{\ge 0} \times N$ be a finitely generated $M$-submodule. 
Then, the function 
\[
h: \mathbb{Z}_{\ge 0} \to \mathbb{Z}_{\ge 0}; \quad  i \mapsto \# \{ x \in N \mid (i,x) \in X \}
\]
is of quasi-polynomial type. 

More precisely, 
its generating function $\sum _{i \ge 0} h(i) t^i$ is of the form 
\[
\frac{Q(t)}{\prod _{a \in M'} (1 - t^{\deg a})}
\]
with some polynomial $Q(t)$. 
Here, $\deg : \mathbb{Z}_{\ge 0} \times N \to \mathbb{Z}_{\ge 0}$ denotes the first projection.
\end{thm}

\begin{prop}\label{prop:C2}
Let $(\Gamma, L)$ be a strongly connected periodic graph. 
Let $\Phi:V_{\Gamma} \to L_{\mathbb{R}}$ be a periodic realization, and let $x_0 \in V_{\Gamma}$. 
Suppose that $x_0$ is $P$-initial. 
We define a bounded set $Q \subset L_{\mathbb{R}}$ as follows:
\begin{itemize}
\item 
For each $v \in V(P_{\Gamma})$, we pick a cycle $q_v \in \nu^{-1}(v)$ such that 
$\overline{x_0} \in \operatorname{supp}(q_v)$, and we define $d_v := w(q_v)$. 

\item 
For each $\sigma \in \operatorname{Facet}(P_{\Gamma})$, 
we fix a triangulation $T_{\sigma}$ of $\sigma$ such that $V(\Delta) \subset V(\sigma)$ holds for any $\Delta \in T_{\sigma}$. 

\item 
We define a bounded set $Q \subset L_{\mathbb{R}}$ as follows:
\[
Q := 
\bigcup _{\substack{\sigma \in \operatorname{Facet}(P_{\Gamma}),\\ 
\Delta \in T_{\sigma}}}
\left( \sum _{v \in V(\Delta)}  [0,1) d_v v \right)
\subset L_{\mathbb{R}}. 
\]
\end{itemize}
Then, we have
\[
C_2(\Gamma, \Phi, x_0) = 
\max \bigl \{ d_{\Gamma}(x_0, y) - d_{P_{\Gamma}, \Phi} (x_0, y) \ \big | \ 
y \in V_{\Gamma},\ \Phi(y) - \Phi(x_0) \in Q \bigr \}.
\]
\end{prop}

\begin{proof}
By the definition of $C_2(\Gamma, \Phi, x_0)$, we have 
\[
C_2(\Gamma, \Phi, x_0) \ge 
\max \bigl \{ d_{\Gamma}(x_0, y) - d_{P_{\Gamma}, \Phi} (x_0, y) \ \big | \ 
y \in V_{\Gamma},\ \Phi(y) - \Phi(x_0) \in Q \bigr \}.
\]
The opposite inequality follows from the same argument as the proof of Theorem \ref{thm:asymp} 
by making the following modifications: 
\begin{itemize}
\item
Replacing $d_v$ and $q_v$ in the proof of Theorem \ref{thm:asymp} with $d_v$ and $q_v$ in the statement of Proposition \ref{prop:C2}.

\item
Replacing $C' _2$ in the proof of Theorem \ref{thm:asymp} with
\[
\max \bigl \{ d_{\Gamma}(x_0, y) - d_{P_{\Gamma}, \Phi} (x_0, y) \ \big | \ 
y \in V_{\Gamma},\ \Phi(y) - \Phi(x_0) \in Q \bigr \}. 
\]
\end{itemize}
Then, by the choice of $q_v$, we have $\overline{x_0} \in \operatorname{supp}(q_v)$. 
In particular, for any walk $p$ in $\Gamma$ from $x_0$, we have 
$\operatorname{supp} (\overline{p}) \cap \operatorname{supp}(q_v) \not = \emptyset$. 
Therefore, we can see that the same argument as the proof of Theorem \ref{thm:asymp} works. 
\end{proof}

\subsection{Well-arranged periodic graphs}\label{subsection:WA}
In this subsection, we introduce a class of periodic graphs called ``well-arranged", 
and we prove that their growth sequences satisfy the same reciprocity laws as in Theorem \ref{thm:ET}(3). 

\begin{defi}\label{defi:WA}
Let $(\Gamma , L)$ be a strongly connected $n$-dimensional periodic undirected graph. 
Let $\Phi : V_{\Gamma} \to L_{\mathbb{R}}$ be a periodic realization, 
and let $x_0 \in V_{\Gamma}$. We say that the triple $(\Gamma, \Phi, x_0)$ is \textit{well-arranged} if 
the following condition holds: 
there exist 
\begin{itemize}
\item 
an integer $d_v \in \mathbb{Z}_{>0}$ for each $v \in V(P_{\Gamma})$, and 
\item 
a triangulation $T_{\sigma}$ of $\sigma$ for each $\sigma \in \operatorname{Facet}(P_{\Gamma})$
\end{itemize}
with the following conditions: 
\begin{enumerate}
\item 
$V(\Delta) \subset V(\sigma)$ holds for any $\Delta \in T_\sigma$. 

\item 
$d_v v \in L$. 

\item
For any $\sigma \in \operatorname{Facet}(P_{\Gamma})$, $\Delta \in T_{\sigma}$ and any subset $V' \subset V(\Delta)$, we have  
\[
d_{\Gamma}(x_0, y) + d_{\Gamma}(y, z) = \sum _{v \in V'} d_v
\]
for $z := \left( \sum _{v \in V'} d_v v \right) + x_0$ and for any $y \in V_{\Gamma}$ 
such that $\Phi(y) - \Phi(x_0) \in \sum _{v \in V'} [0,d_v) v$. 
\end{enumerate}
\end{defi}

\begin{rmk}
Applying Definition \ref{defi:WA}(3) to the case $y=x_0$, we have $d_{\Gamma}(x_0, z) = \sum _{v \in V'} d_v$ in Definition \ref{defi:WA}(3). 
Therefore, the equation in (3) says 
``for any $y$, there exists a shortest walk from $x_0$ to $z$ factors through $y$". 
\end{rmk}

\begin{lem}\label{lem:WAPI}
If $(\Gamma, \Phi, x_0)$ is well-arranged, then $x_0$ is $P$-initial. 
\end{lem}
\begin{proof}
Let $v \in V(P_{\Gamma})$. 
Take $\sigma \in \operatorname{Facet}(P_{\Gamma})$ and $\Delta \in T_{\sigma}$ satisfying $v \in V(\Delta)$. 
Then, by Definition \ref{defi:WA}(3) for $V' = \{ v \}$, we have $d_{\Gamma}(x_0, x_0 + d_v v) = d_v$. 
Therefore, there exists a walk $p$ in $\Gamma$ from $x_0$ such that $w(p) = d_v$ and 
$t(p) = d_v v + s(p)$. 
By Lemma \ref{lem:bunkai}, $\overline{p}$ decomposes to a walkable sequence $(q_0, q_1, \ldots , q_{\ell})$ with $q_0 = \emptyset _{\overline{x_0}}$ such that 
$\langle \overline{p} \rangle = \sum _{i = 1} ^{\ell} \langle q_i \rangle$. 
Note that we have 
$\frac{\mu(\langle \overline{p} \rangle)}{w(p)} = v \in V(P_{\Gamma})$ and 
$\frac{\mu(\langle q_i \rangle)}{w(q_i)} = \nu(q_i) \in P_{\Gamma}$. 
Therefore, we have $\nu(q_i) = v$ for each $1 \le i \le \ell$. 
Since $\overline{x_0} \in \operatorname{supp}(q_i)$ for some $i$, we conclude that $x_0$ is $P$-initial. 
\end{proof}

Proposition \ref{prop:refWA} below shows that the graphs treated in Theorem \ref{thm:ET}(3)(i) are well-arranged. 
Therefore, Theorem \ref{thm:WA} below can be seen as the generalization of 
Theorem \ref{thm:ET}(3) for undirected periodic graphs. 

\begin{prop}\label{prop:refWA}
Let $(\Gamma , L)$ be a strongly connected $n$-dimensional periodic undirected graph. 
Let $\Phi : V_{\Gamma} \to L_{\mathbb{R}}$ be a periodic realization, and let $x_0 \in V_{\Gamma}$. 
Suppose that $C_1(\Gamma, \Phi, x_0) < \frac{1}{2}$ and $C_2(\Gamma, \Phi, x_0) < \frac{1}{2}$. 
Then $(\Gamma, \Phi, x_0)$ is well-arranged. 
\end{prop}
\begin{proof}
Since $\Gamma$ is undirected, we have $d_{\Gamma}(y_1,y_2) = d_{\Gamma}(y_2,y_1)$ for any $y_1, y_2 \in V_{\Gamma}$. 
By Theorem \ref{thm:ET}(2), we have both
\[
d_{\Gamma}(y_1, y_2) 
= \left \lceil d_{P_{\Gamma}, \Phi}(y_1, y_2) - \frac{1}{2} \right \rceil, \ \ 
d_{\Gamma}(y_1, y_2) 
= \left \lfloor d_{P_{\Gamma}, \Phi}(y_1, y_2) + \frac{1}{2} \right \rfloor
\tag{$\diamondsuit$}
\]
for any $y_1, y_2 \in V_{\Gamma}$ satisfying $\overline{y_1} = \overline{x_0}$. 

For each $v \in V(P_{\Gamma})$, take $d_v \in \mathbb{Z}_{>0}$ satisfying Definition \ref{defi:WA}(2). 
For each $\sigma \in \operatorname{Facet}(P_{\Gamma})$, take a triangulation $T_{\sigma}$ of $\sigma$ satisfying Definition \ref{defi:WA}(1). 
We prove that Definition \ref{defi:WA}(3) is satisfied for any choice of such $d_v$'s and $T_{\sigma}$'s. 

Let $\sigma \in \operatorname{Facet}(P_{\Gamma})$, $\Delta \in T_{\sigma}$, and $V' \subset V(\Delta)$. Take $y \in V_{\Gamma}$ such that $\Phi(y) - \Phi(x_0) \in \sum _{v \in V'} [0,d_v) v$. 
Then, by the condition on $y$, we have 
\[
d_{P_{\Gamma}, \Phi}(x_0, y) + d_{P_{\Gamma}, \Phi}(y, z) 
= \sum _{v \in V'} d_v \in \mathbb{Z}_{\ge 0}. 
\]
Then, the desired equality $d_{\Gamma}(x_0, y) + d_{\Gamma}(y,z) = \sum _{v \in V'} d_v$ 
follows from ($\diamondsuit$). 
\end{proof}

\begin{thm}\label{thm:WA}
Let $(\Gamma , L)$ be a strongly connected $n$-dimensional periodic undirected graph. 
Let $\Phi : V_{\Gamma} \to L_{\mathbb{R}}$ be a periodic realization, and let $x_0 \in V_{\Gamma}$. 
Suppose that $(\Gamma, \Phi, x_0)$ is well-arranged. 
Then, the following assertions hold. 
\begin{enumerate}
\item 
The function $i \mapsto s_{\Gamma, x_0, i}$ obtained 
by the growth sequence is a quasi-polynomial on $i \ge 1$. 
The function $i \mapsto b_{\Gamma, x_0, i}$  is a quasi-polynomial on $i \ge 0$. 

\item
More precisely, the generating function $G_s(t) := \sum _{i \ge 0} s_{\Gamma, x_0, i} t^i$ of 
the growth sequence $(s_{\Gamma, x_0, i})_i$ can be expressed as
\[
G_s(t) = \frac{Q_1(t)}{Q_2(t)}
\]
with polynomials $Q_1$ and $Q_2$ satisfying the two conditions:
\begin{itemize}
\item
$Q_2(t) = \operatorname{LCM} \left \{ \prod _{v \in V(\Delta)} (1 - t^{d_v}) 
\ \middle | \ \sigma \in \operatorname{Facet}(P),\ \Delta \in T_{\sigma} \right \}$. 

\item 
$\deg Q_1 \le \deg Q_2$. 
\end{itemize}

\item
The same reciprocity as in Theorem \ref{thm:ET}(3) is satisfied. 
\end{enumerate}
\end{thm}

\begin{proof}
The assertion (1) follows from (2) since the generating function corresponding to the function $i \mapsto b_{\Gamma, x_0, i}$ is equal to $\frac{Q_1(t)}{(1-t)Q_2(t)}$. 
For the reciprocity in Theorem \ref{thm:ET}(3), it is sufficient to show the formula 
$G_s(1/t) = (-1)^n G_s(t)$ because this formula implies the other three formulas. 
For a subset $E \subset L_{\mathbb{R}}$, we define 
\[
G_E(t) := \sum _{y \in V_{\Gamma} \cap \Phi ^{-1}(E)} t^{d_{\Gamma}(x_0, y)}.  
\]
Using this notation, $G_s(t)$ in Theorem \ref{thm:ET}(3) can be expressed as 
$G_s(t) = G_{L_{\mathbb{R}}} (t)$. 

Take $d_v$'s for $v \in V(P_{\Gamma})$ and a triangulation $T_{\sigma}$ of $\sigma$ for each 
$\sigma \in \operatorname{Facet}(P_{\Gamma})$ as in Definition \ref{defi:WA}. 
We set 
\[
T := \{ \Delta' \mid \sigma \in \operatorname{Facet}(P_{\Gamma}), 
\Delta \in T_{\sigma}, \Delta ' \in \operatorname{Face}(\Delta) \}
\]

For $\Delta ' \in T$, 
we define subsets $D^+ _{\Delta '}, D^- _{\Delta '}\subset L_{\mathbb{R}}$ by
\[
D^+ _{\Delta '} := \sum _{v \in V(\Delta ')} [0, \infty) v, \qquad 
D^- _{\Delta '} := \sum _{v \in V(\Delta ')} (- \infty, 0) v. 
\]
Then, we have both
\begin{align*}
G_{L_{\mathbb{R}}}(t) \tag{$\spadesuit$}
&= G_{\{ 0 \}}(t) + \sum_{\Delta ' \in T} G_{D^{-}_{\Delta '}}(t) 
= 1 + \sum_{\Delta ' \in T} G_{D^{-}_{\Delta '}}(t). \\ 
G_{L_{\mathbb{R}}}(t) 
&= (-1)^n \left( G_{\{ 0 \}}(t) 
+ \sum_{\Delta ' \in T} (-1)^{\dim \Delta '} G_{D^{+}_{\Delta '}}(t) \right)\\
&= (-1)^n \left( 1 + \sum_{\Delta ' \in T} (-1)^{\dim \Delta '} G_{D^{+}_{\Delta '}}(t) \right). 
\end{align*}
For $\Delta ' \in T$, we define 
\[
\diamondsuit^+ _{\Delta '} := \sum _{v \in V(\Delta ')} [0, d_v) v, \qquad 
\diamondsuit^- _{\Delta '} := \sum _{v \in V(\Delta ')} [-d_v, 0) v. 
\]
Then the following assertions hold. 
\begin{claim}\label{claim:diamond}
\begin{enumerate}
\item 
For any $y \in V_{\Gamma} \cap \diamondsuit^+ _{\Delta '}$
and for $a_v \in \mathbb{Z}_{\ge 0}$, 
we have 
\[
d_{\Gamma} \left( x_0, y + \sum _{v \in V(\Delta ')} a_v d_v v \right) 
= d_{\Gamma} (x_0, y) + \sum _{v \in V(\Delta ')} a_v d_v. 
\]

\item 
For any $y \in V_{\Gamma} \cap \diamondsuit^- _{\Delta '}$
and for $a_v \in \mathbb{Z}_{\ge 0}$, 
we have 
\[
d_{\Gamma} \left( x_0, y - \sum _{v \in V(\Delta ')} a_v d_v v \right) 
= d_{\Gamma} (x_0, y) + \sum _{v \in V(\Delta ')} a_v d_v. 
\]

\item
For any $y \in V_{\Gamma} \cap \diamondsuit^+ _{\Delta '}$, 
we have
\[
d_{\Gamma} (x_0, y) + d_{\Gamma} \left( x_0, y - \sum _{v \in V(\Delta ')} d_v v \right)
= \sum _{v \in V(\Delta ')} d_v.
\]

\item We have
\[
G_{D^+ _{\Delta '}}(t)
= \frac{G_{\diamondsuit ^+ _{\Delta '}}(t)}{\prod _{v \in V(\Delta ')} (1 - t^{d_v})}, \quad
G_{D^- _{\Delta '}}(t)
= \frac{G_{\diamondsuit ^- _{\Delta '}}(t)}{\prod _{v \in V(\Delta ')} (1 - t^{d_v})}. 
\]

\item
We have $t^{d_{\Delta '}} G_{\diamondsuit ^- _{\Delta '}}(1/t) 
= G_{\diamondsuit ^+ _{\Delta '}}(1/t)$, where 
$d_{\Delta '} := \sum _{v \in V(\Delta ')} d_v$.

\item
We have 
$G_{D^- _{\Delta '}}(1/t) = (-1)^{\dim \Delta ' +1}G_{D^+ _{\Delta '}}(t)$. 

\item
We have 
$\deg G_{\diamondsuit ^+ _{\Delta '}}(t) \le \sum _{v \in V(\Delta ')} d_v$. 
\end{enumerate}
\end{claim}
\begin{proof}
We prove (1). 
We set 
\begin{align*}
z &:= x_0 + \sum _{v \in V(\Delta ')} d_v v, \\
y'&:= y + \sum _{v \in V(\Delta ')} a_v d_v v, \\
z'&:= z + \sum _{v \in V(\Delta ')} a_v d_v v. 
\end{align*}
For each $v \in V(P_{\Gamma})$, by the proof of Lemma \ref{lem:WAPI},  
we can take a closed walk $q_v$ in $\Gamma /L$ such that 
\[
w(q_v) = d_v, \quad 
\mu(\langle q_v \rangle) = d_v v, \quad
x_0 \in \operatorname{supp}(q_v). 
\]
Let $p$ be a walk in $\Gamma$ from $x_0$ to $y$ with $w(p) = d_{\Gamma}(x_0, y)$. 
Since $\operatorname{supp}{\overline{p}} \cap \operatorname{supp}(q_v) \not = \emptyset$, 
there exists a walk $p'$ in $\Gamma$ from $x_0$ to $y'$ such that 
$w(p') = w(p) + \sum _{v \in V(\Delta ')} a_v d_v$. 
Therefore, we have
\[
d_{\Gamma} (x_0, y') 
\le d_{\Gamma} (x_0, y) + \sum _{v \in V(\Delta ')} a_v d_v. \tag{i}
\]
On the other hand, we have 
\[
d_{\Gamma} (x_0, z') \le d_{\Gamma} (x_0, y') + d_{\Gamma} (y', z'). 
\]
Here, we have 
\[
d_{\Gamma} (x_0, z') = \sum _{v \in V(\Delta ')} (a_v +1) d_v
\]
by Lemma \ref{lem:dPd} and the inequality (i) for $y = x_0$. 
Furthermore, we have 
\[
d_{\Gamma} (y', z') = d_{\Gamma} (y, z) = -d_{\Gamma}(x_0, y) + \sum _{v \in V(\Delta ')} d_v
\]
by Definition \ref{defi:WA}(3). 
Therefore, we get the opposite inequality of (i). 
We complete the proof of the assertion (1). 
The assertion (2) is proved by the similar way. 

By translation, we have 
\[
d_{\Gamma} \left( x_0, y - \sum _{v \in V(\Delta ')} d_v v \right)
= d_{\Gamma} (z, y). 
\]
Therefore, the assertion (3) follows from Definition \ref{defi:WA}(3). 

The assertion (4) follows from (1) and (2). 
The assertion (5) follows from (3). 
The assertion (6) follows from (4) and (5). 

By the inequality (i), we have $d_{\Gamma}(x_0, y) \le \sum _{v \in V(\Delta ')} d_v$ for 
any $y \in V_{\Gamma} \cap \diamondsuit^+ _{\Delta '}$. 
Therefore, we conclude the assertion (7). 
\end{proof}
By ($\spadesuit$) and Claim \ref{claim:diamond}(4)(7),
we conclude that $G_{L_{\mathbb{R}}}(t)$ is a rational function of the form 
$\frac{Q_1(t)}{Q_2(t)}$ satisfying the two conditions in (2). 
By ($\spadesuit$) and Claim \ref{claim:diamond}(6), 
we have the reciprocity $G_{L_{\mathbb{R}}}(1/t) = (-1)^n G_{L_{\mathbb{R}}}(t)$. 
We complete the proof. 
\end{proof}

\begin{rmk}\label{rmk:compute}
When a periodic graph is well-arranged for some realization, 
we can conclude that the growth sequence $(s_{\Gamma, x_0, i})_i$ is a quasi-polynomial on $d \ge 1$ (Theorem \ref{thm:WA}(1)), and we can find a quasi-period (Theorem \ref{thm:WA}(2)). 
Therefore, we can determine the growth sequence from its first few terms. 
In Section \ref{section:eg}, by using this method, we determine the growth series for several new examples. 
\end{rmk}

\section{Examples}\label{section:eg}
In this section, for some specific periodic graphs, we will see whether they are well-arranged or not. 
Furthermore, we determine the growth series in several new examples by the method explained in Remark \ref{rmk:compute}. 

\begin{itemize}
\item
In Subsection \ref{subsection:GSS}, we will examine seven tilings in \cite{GS19}. 
For these examples, the growth sequences have already been determined by Goodman-Strauss and Sloane in \cite{GS19}. 
However, we expect that this subsection will help the reader to understand the concepts ``well-arranged'' and ``$P$-initial''.

\item
In Subsection \ref{subsection:eg_2dim}, we treat another tiling called the $(3^6;3^2.4.3.4;3^2.4.3.4)$ 3-uniform tiling. 
We check that the corresponding periodic graph is well-arranged under a suitable realization. 
Furthermore, we determine its growth series. 
As far as we know, this is the first time that its growth series have been determined with a proof.

\item
In Subsection \ref{subsection:eg_3dim}, we treat $3$-dimensional periodic graphs obtained by carbon allotropes. 
We confirm that $22$ of them are well-arranged, and we determine their growth series. 
As far as we know, this is the first time that the growth sequences have been computed for nontrivial $3$-dimensional periodic graphs. 
\end{itemize}

\begin{rmk}\label{rmk:impl}
It is not difficult to implement the following calculations and verifications in a computer program. 
\begin{enumerate}
\item 
Calculate the first few terms of the growth sequence $(s_{\Gamma, x_0, i})_i$ 
 (breadth-first search algorithm).
 
\item 
Check that $x_0$ is $P$-initial or not. 

\item
Calculate $C_1(\Gamma, \Phi, x_0)$ (Remark \ref{rmk:C1}(1)). 

\item
Calculate $C_2(\Gamma, \Phi, x_0)$ when $x_0$ is $P$-initial (Proposition \ref{prop:C2}).
 
\item
Check that $(\Gamma, \Phi, x_0)$ is well-arranged or not for given $d_v$'s and triangulations $T_{\sigma}$'s. 

\item
Determine the growth series for well-arranged periodic graphs (Remark \ref{rmk:compute}). 
\end{enumerate}

We prepare implementations of the algorithms in {\sf Python} to compute or check (1)--(6) above for unweighted periodic graphs. 
For details, see the source code:
\begin{center}
\url{https://github.com/yokozuna57/Ehrhart_on_PG}
\end{center}
\end{rmk}

\subsection{2-dimensional periodic graphs from \cite{GS19}}\label{subsection:GSS}
In this subsection, we examine seven specific periodic tilings from \cite{GS19} illustrated in Figures \ref{fig:cairo}-\ref{fig:snub632}. 
Let $(\Gamma, L)$ be the corresponding unweighted undirected periodic graphs, 
and let $\Phi$ be the periodic realizations shown in the figures. 
The parallelogram drawn with red lines represents a fundamental region of the periodic graph $(\Gamma, L)$. 
In \cite{GS19}, Goodman-Strauss and Sloane determine their growth sequences. 
We list their generating functions in the item ``Growth series'' of Table \ref{table:tilings}. 
With the help of a computer program (cf.\ Remark \ref{rmk:impl}), we can check whether these tilings (and their starting points) are $P$-initial and whether they are well-arranged as in Table \ref{table:tilings}. 
In the table, ``PI'' stands for $P$-initial, ``WA'' for well-arranged, and ``RL'' for reciprocity law.

\begin{figure}[h]
\begin{minipage}[t]{0.49\hsize}
\centering
\includegraphics[height=\zua]{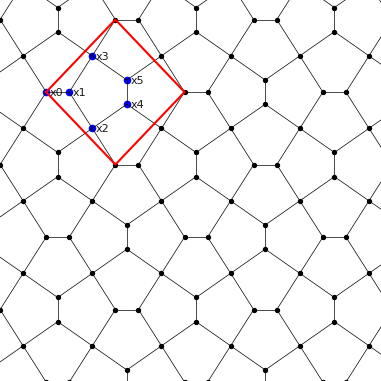}
\end{minipage}
\begin{minipage}[t]{0.49\hsize}
\centering
\includegraphics[height=\zua]{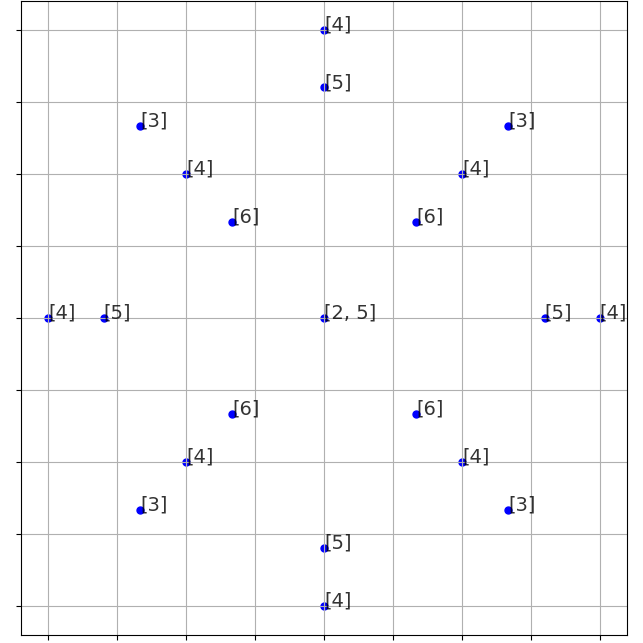}
\end{minipage}

\vspace{\zub}

\caption{The Cairo tiling and its $\operatorname{Im}(\nu)$. }
\label{fig:cairo}
\end{figure}

\vspace{\zuc}

\begin{figure}[h]
\begin{minipage}[t]{0.49\hsize}
\centering
\includegraphics[height=\zua]{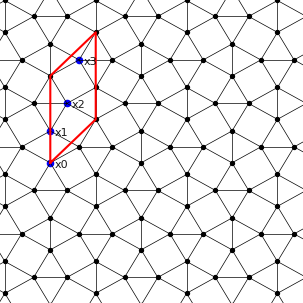}
\end{minipage}
\begin{minipage}[t]{0.49\hsize}
\centering
\includegraphics[height=\zua]{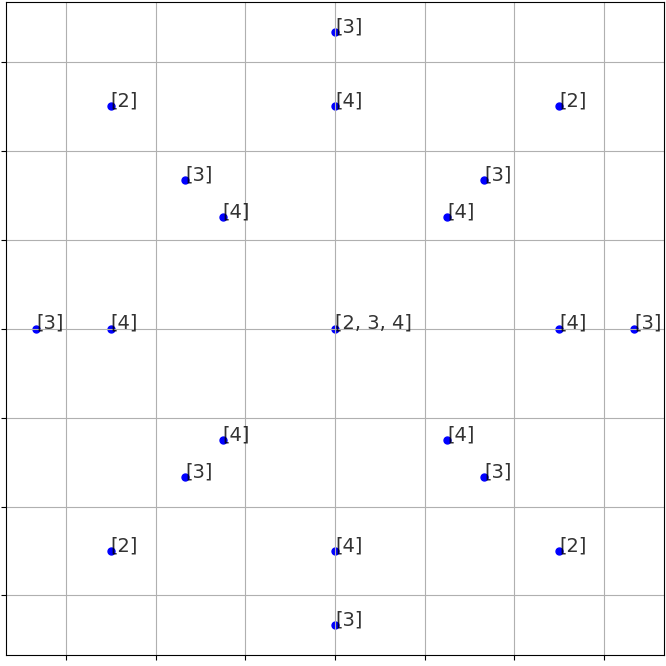}
\end{minipage}

\vspace{\zub}

\caption{The $3^2.4.3.4$ uniform tiling and its $\operatorname{Im}(\nu)$.}
\label{fig:32434}
\end{figure}

\vspace{\zuc}

\begin{figure}[h]
\begin{minipage}[t]{0.49\hsize}
\centering
\includegraphics[height=\zua]{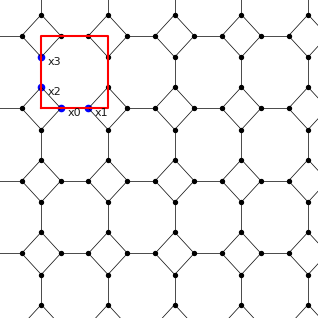}
\end{minipage}
\begin{minipage}[t]{0.49\hsize}
\centering
\includegraphics[height=\zua]{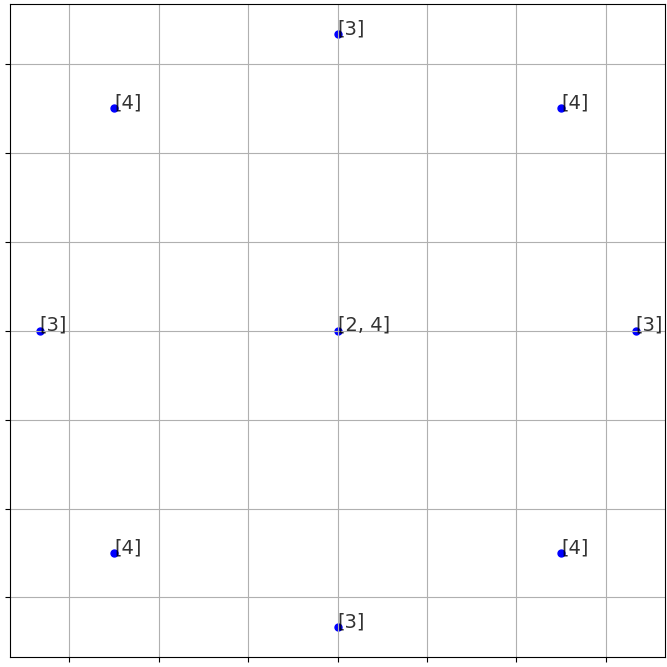}
\end{minipage}

\vspace{\zub}

\caption{The $4.8^2$ uniform tiling and its $\operatorname{Im}(\nu)$.}
\label{fig:482}
\end{figure}

\vspace{\zuc}

\begin{figure}[h]
\begin{minipage}[t]{0.49\hsize}
\centering
\includegraphics[height=\zua]{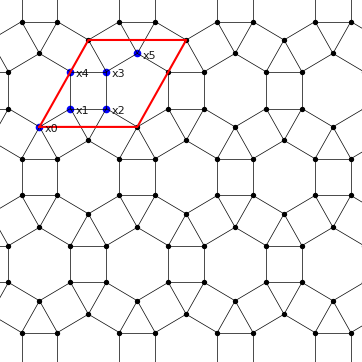}
\end{minipage}
\begin{minipage}[t]{0.49\hsize}
\centering
\includegraphics[height=\zua]{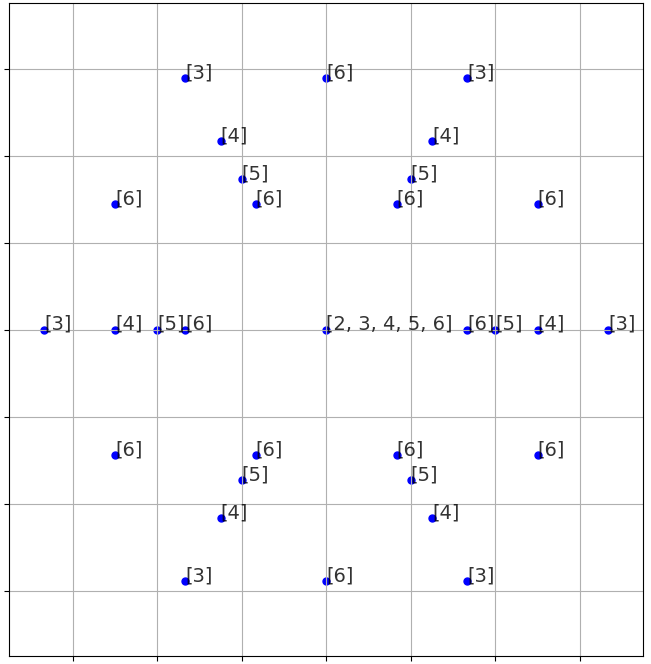}
\end{minipage}

\vspace{\zub}

\caption{The $3.4.6.4$ uniform tiling and its $\operatorname{Im}(\nu)$.}
\label{fig:3464}
\end{figure}

\vspace{\zuc}

\begin{figure}[h]
\begin{minipage}[t]{0.49\hsize}
\centering
\includegraphics[height=\zua]{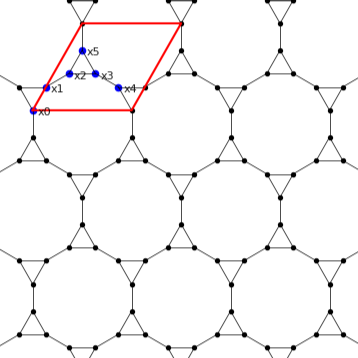}
\end{minipage}
\begin{minipage}[t]{0.49\hsize}
\centering
\includegraphics[height=\zua]{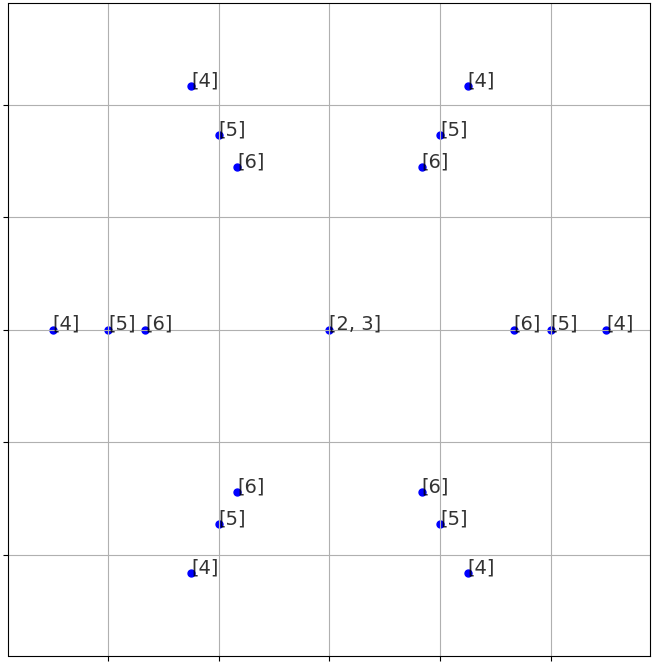}
\end{minipage}

\vspace{\zub}

\caption{The $3.12^2$ uniform tiling and its $\operatorname{Im}(\nu)$.}
\label{fig:3122}
\end{figure}

\vspace{\zuc}

\begin{figure}[h]
\begin{minipage}[t]{0.49\hsize}
\centering
\includegraphics[height=\zua]{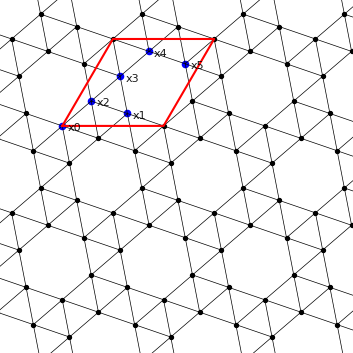}
\end{minipage}
\begin{minipage}[t]{0.49\hsize}
\centering
\includegraphics[height=\zua]{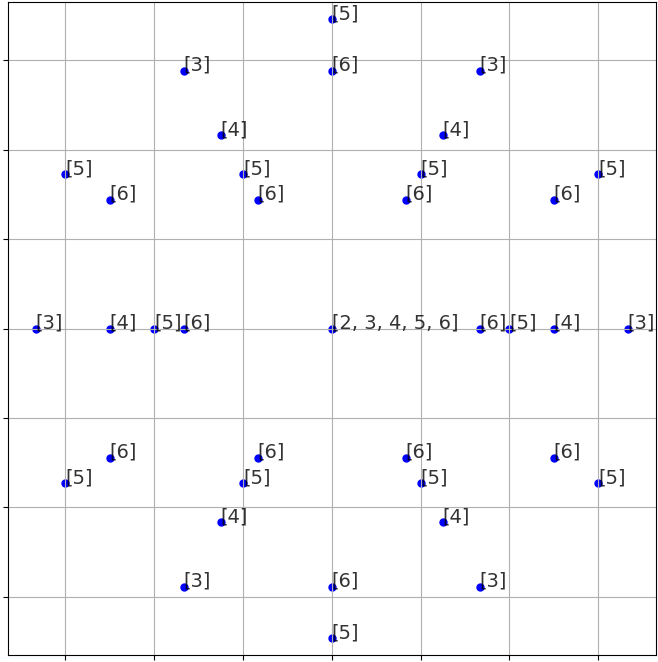}
\end{minipage}

\vspace{\zub}

\caption{The $3^4.6$ uniform tiling and its $\operatorname{Im}(\nu)$.}
\label{fig:346}
\end{figure}

\vspace{\zuc}

\begin{figure}[h]
\begin{minipage}[t]{0.49\hsize}
\centering
\includegraphics[height=\zua]{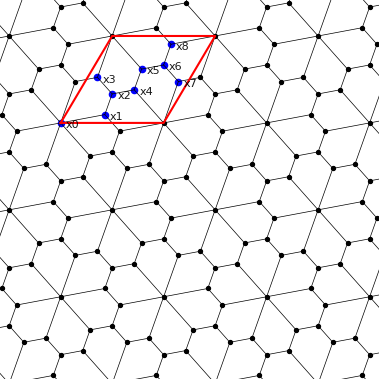}
\end{minipage}
\begin{minipage}[t]{0.49\hsize}
\centering
\includegraphics[height=\zua]{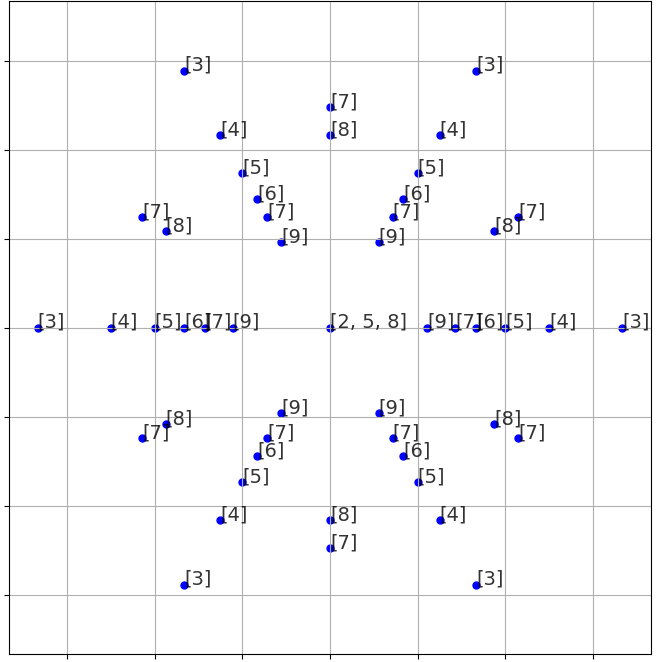}
\end{minipage}

\vspace{\zub}

\caption{The snub-$632$ 3-uniform tiling and its $\operatorname{Im}(\nu)$.}
\label{fig:snub632}
\end{figure}

\renewcommand{\arraystretch}{1.5}
\begin{table}[h]

\vspace{5mm}

\centering
\begin{tabular}{l||cccccc} \hline
\multirow{2}{*}{Tiling / start pt.} & $\#(V_{\Gamma /L})$ & PI  & WA & RL &  $C_1$ & $C_2$   \\
& \multicolumn{6}{c}{Growth series} \\
\hline \hline
\multirow{2}{*}{Fig.\ \ref{fig:cairo} / $x_2, x_3$} & 6 & YES & YES & YES & 1/3 & 1/3 \\
& \multicolumn{6}{c}{\Large $\frac{1+ 2t + t^2}{(1 - t)^2}$} \\
\hline
\multirow{2}{*}{Fig.\ \ref{fig:cairo} / otherwise} & 6 & NO & NO & NO & 2/3 & $\ge 1$ \\
& \multicolumn{6}{c}{\Large $\frac{1+2 t+5 t^{2}+4 t^{3}+2 t^{4}+3 t^{5}-t^{7}}{(1 - t)(1 - t^4)}$} \\
\hline
\multirow{2}{*}{Fig.\ \ref{fig:32434} / all} & 4 & YES & YES & YES & 0.36... & 0.36...  \\
& \multicolumn{6}{c}{\Large $\frac{1+4 t+6 t^{2}+4 t^{3}+t^{4}}{(1-t)(1-t^3)}$} \\
\hline
\multirow{2}{*}{Fig.\ \ref{fig:482} / all} & 4 & YES & YES & YES & 0.58... & 0.58...  \\
& \multicolumn{6}{c}{\Large $\frac{1+2t+2t^{2}+2t^{3}+t^{4}}{(1-t)(1-t^3)}$} \\
\hline
\multirow{2}{*}{Fig.\ \ref{fig:3464} / all} & 6 & YES & YES & YES & 0.46... & 0.46...  \\
& \multicolumn{6}{c}{\Large $\frac{1 + 2t + t^2}{(1-t)^2}$} \\
\hline
\multirow{2}{*}{Fig.\ \ref{fig:3122} / all} & 6 & NO & NO & NO & 0.38... & $\ge 1$  \\
& \multicolumn{6}{c}{\Large $\frac{1+3 t+4 t^{2}+6 t^{3}+6 t^{4}+6 t^{5}+6 t^{6}+3 t^{7}+3 t^{8}-2 t^{10}}{(1-t^4)^2}$} \\
\hline
\multirow{2}{*}{Fig.\ \ref{fig:346} / all} & 6 & YES & YES & YES & 5/7 & 5/7  \\
& \multicolumn{6}{c}{\Large $\frac{1+4 t +4 t^{2}+6 t^{3}+4 t^{4}+4 t^{5}+t^{6}}{(1-t)(1-t^5)}$} \\
\hline
\multirow{2}{*}{Fig.\ \ref{fig:snub632} / $x_0$} & 9 & YES & NO & NO & 3/7 & 1  \\
& \multicolumn{6}{c}{\Large $\frac{1+6 t+12 t^{2}+10 t^{3}+12 t^{4}+12 t^{5}+t^{6}}{(1 - t^3)^2}$} \\
\hline
\multirow{2}{*}{Fig.\ \ref{fig:snub632} / $x_2, x_6$} & 9 & NO & NO & NO & 3/7 & $\ge 1$  \\
& \multicolumn{6}{c}{\Large $\frac{1+3 t+6 t^{2}+13 t^{3}+15 t^{4}+6 t^{5}+4 t^{6}+9 t^{7}-3 t^{10}}{(1- t^3)^2}$} \\
\hline
\multirow{2}{*}{Fig.\ \ref{fig:snub632} / otherwise} & 9 & NO & NO & NO & 6/7 & $\ge 1$  \\
& \multicolumn{6}{c}{\Large $\frac{1+3 t+9 t^{2}+13 t^{3}+12 t^{4}+9 t^{5}+8 t^{6}+4 t^{7}-t^{8}-2 t^{9}-2 t^{10}}{(1 -t^3)^2}$} \\ 
\hline
\end{tabular}

\vspace{3mm}

\caption{Growth series of the seven tilings in \cite{GS19}.}
\label{table:tilings}
\end{table}
\renewcommand{\arraystretch}{1}

\newpage

\subsection{The $(3^6;3^2.4.3.4;3^2.4.3.4)$ 3-uniform tiling}\label{subsection:eg_2dim}
The $(3^6;3^2.4.3.4;3^2.4.3.4)$ 3-uniform tiling is a two-dimensional tiling illustrated in Figure \ref{fig:t363243432434}. 
Let $(\Gamma, L)$ be the corresponding unweighted undirected periodic graph, 
and let $\Phi$ be the periodic realization shown in the figure. 
The parallelogram drawn with red lines represents a fundamental region of the periodic graph $(\Gamma, L)$. 
The vertices $x_1, x_2, x_3, x_5, x_{10}$ and $x_{11}$ are symmetric with respect to $\operatorname{Aut} \Gamma$. 
The vertices $x_4, x_6, x_7, x_8, x_9$ and $x_{12}$ are also symmetric. 

\begin{figure}[h]
\begin{minipage}[t]{0.49\hsize}
\centering
\includegraphics[height=\zua]{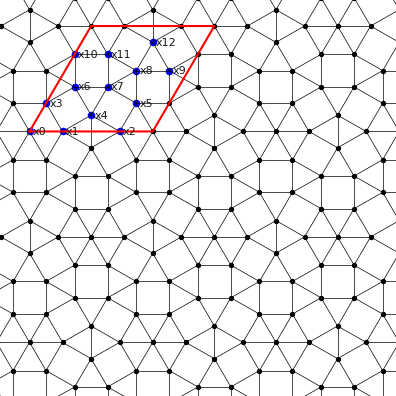}
\end{minipage}
\begin{minipage}[t]{0.49\hsize}
\centering
\includegraphics[height=\zua]{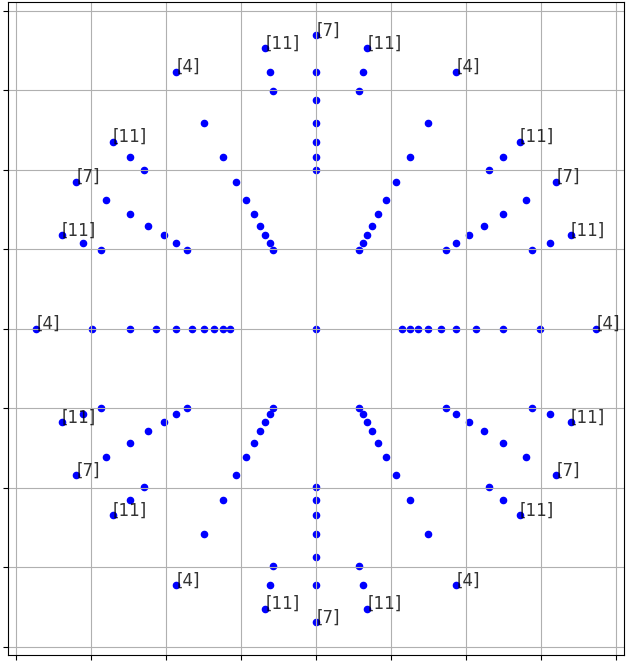}
\end{minipage}

\vspace{\zub}

\caption{The $(3^6;3^2.4.3.4;3^2.4.3.4)$ 3-uniform tiling and its $\operatorname{Im}(\nu)$.}
\label{fig:t363243432434}
%
\vspace{6mm}
%
\centering
\includegraphics[width=\zua]{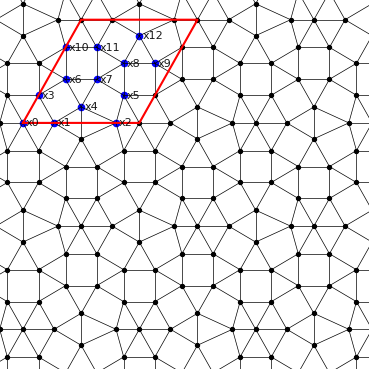}

\vspace{\zub}

\caption{New realization $\Phi '$ obtained by moving the vertex $x_2$ to the right.}
\label{fig:t363243432434_2}
\end{figure}

With the help of a computer program, we can see that all vertices are $P$-initial, and we have
\begin{align*}
C_1(\Gamma, \Phi, x_0) &= C_2(\Gamma, \Phi, x_0) = 0.267..., \\
C_1(\Gamma, \Phi, x_1) &= C_2(\Gamma, \Phi, x_1) = 0.535..., \\
C_1(\Gamma, \Phi, x_4) &= C_2(\Gamma, \Phi, x_4) = 0.422.... 
\end{align*}
Therefore, $(\Gamma, \Phi, x_0)$ and $(\Gamma, \Phi, x_4)$ are well-arranged by Proposition \ref{prop:refWA}. 
Note that we cannot apply Proposition \ref{prop:refWA} to $x_1$ because
$C_1(\Gamma, \Phi, x_1) + C_2(\Gamma, \Phi, x_1) \ge 1$. 
Furthermore, we could not check whether $(\Gamma, \Phi, x_1)$ is well-arranged or not by our computer program 
(indeed, we checked the desired condition (3) in Definition \ref{defi:WA} for one triangulation $T_{\sigma}$ and $d_v$ satisfying (1) and (2), but the result was negative). 
Fortunately, by changing the periodic realization as Figure \ref{fig:t363243432434_2}, 
we can confirm that $(\Gamma, \Phi ', x_1)$ is well-arranged.

Since we have confirmed that $(\Gamma, \Phi, x_0)$, $(\Gamma, \Phi ', x_1)$, and $(\Gamma, \Phi, x_4)$ are well-arranged, 
according to the method explained in Remark \ref{rmk:compute}, we can determine their growth series as in Table \ref{table:3uni}. 

\renewcommand{\arraystretch}{1.8}
\begin{table}[h]
  \centering
  \begin{tabular}{cc}
    \hline
    Start pt. & Growth series  \\
    \hline \hline
    $x_0$ & $\frac{1+6 t+12 t^{2}+12 t^{3}+23 t^{4}+24 t^{5}+24 t^{6}+23 t^{7}+12 t^{8}+12 t^{9}+6 t^{10}+t^{11}}{(1-t^4)(1-t^7)}$\\ 
    $x_1, x_2, x_3, x_5, x_{10}, x_{11}$ & $\frac{1 + 4t + 5t^2 + 7t^3 + 5t^4 + 7t^5 + 5t^6 + 4t^7 + t^8}{(1-t)(1-t^7)}$ \\
    $x_4, x_6, x_7, x_8, x_9, x_{12}$ & $\frac{1+5 t+12 t^{2}+17 t^{3}+22 t^{4}+21 t^{5}+21 t^{6}+22 t^{7}+17 t^{8}+12 t^{9}+5 t^{10}+t^{11}}{(1-t^4)(1-t^7)}$ \\
    \hline
  \end{tabular}

\vspace{3mm}

\caption{Growth series of the $(3^6;3^2.4.3.4;3^2.4.3.4)$ 3-uniform tiling.}
\label{table:3uni}
\end{table}
\renewcommand{\arraystretch}{1}

In what follows, we shall give the calculation of $G_{\Gamma, x_0}$ in detail. 
By Theorem \ref{thm:WA}, it follows that $G_{\Gamma, x_0}(t)$ is of the form 
\[
G_{\Gamma, x_0}(t) = \frac{Q(t)}{(1-t^4)(1-t^7)}
\]
with some polynomial $Q(t)$ with $\deg Q \le 11$. 
In particular, we can conclude that the growth sequence $(s_{\Gamma, x_0, i})_{i \ge 0}$ satisfies the linear recurrence relation corresponding to $(1-t^4)(1-t^7)$ for $i \ge 1$. 

With the help of a computer program (breadth-first search algorithm), the first 12 terms $(s_{\Gamma, x_0, i})_{0 \le i \le 11}$ can be computed: 
1, 6, 12, 12, 24, 30, 36, 36, 42, 54, 54, 60. 
Then the sequence $(s_{\Gamma, x_0, i})_{i \ge 0}$ is completely determined by the linear recurrence relation, and its generating function can be calculated as
\begin{align*}
&G_{\Gamma, x_0}(t) 
= 
\frac{\bigl(\text{The terms of $(1-t^4)(1-t^7)\sum_{i=0}^{11}s_{\Gamma, x_0, i}t^i$ of degree $11$ or less}\bigr)}{(1-t^4)(1-t^7)}\\
&=
\frac{1+6 t+12 t^{2}+12 t^{3}+23 t^{4}+24 t^{5}+24 t^{6}+23 t^{7}+12 t^{8}+12 t^{9}+6 t^{10}+t^{11}}{(1-t^4)(1-t^7)}. 
\end{align*}
Here, we can see that $G_{\Gamma, x_0}(t)$ actually satisfies the reciprocity law $G_{\Gamma, x_0}(1/t) = G_{\Gamma, x_0}(t)$.

\subsection{3-dimensional periodic graphs}\label{subsection:eg_3dim}
In this subsection, we treat $3$-dimensional periodic graphs obtained by some carbon allotropes. 
\textit{Samara Carbon Allotrope Database} \cite{SACADA} currently lists $525$ carbon allotropes. 
In this paper, we examine only the carbon allotropes that satisfy the following conditions:
\begin{itemize}
\item The corresponding graph $\Gamma$ is a uniform graph (i.e.\ all vertices of $\Gamma$ are symmetric with respect to $\operatorname{Aut}(\Gamma)$). 

\item Each vertex of $\Gamma$ has order $4$.  
\end{itemize}
There are 49 such carbon allotropes in {\sf SACADA} database: 
$\#$1, 8, 10, 11, 12, 13, 20, 21, 29, 30, 31, 33, 35, 37, 39, 51, 52, 56, 57, 58, 59, 60, 65, 66, 67, 68, 69, 70, 71, 73, 74, 75, 76, 77, 78, 79, 80, 81, 82, 83, 84, 85, 86, 87, 88, 89, 613, 628, 629 in {\sf SACADA} database.
Of these graphs, $22$ graphs can be verified to be well-arranged by a computer program: 
$\#$1, 10, 21, 37, 39, 52, 56, 60, 65, 67, 74, 75, 76, 77, 80, 81, 82, 86, 87, 88, 89, 613 in {\sf SACADA} database. 
Using the method explained in Remark \ref{rmk:compute}, we can determine the growth series of these $22$ periodic graphs as in Table \ref{table:22series}. 
We also display the graphs $\Gamma$ and the growth polytopes $P_{\Gamma}$ for $\#1$ and $\#60$ in Figures \ref{fig:sacada1} and \ref{fig:sacada60}.

\renewcommand{\arraystretch}{1.8}
\begin{table}[htp]
  \centering
  \begin{tabular}{cc}
    \hline
    {\sf SACADA}\# & Growth series  \\
    \hline \hline
    $\#1$ & $\frac{1+2 t+4 t^{2}+2 t^{3}+t^{4}}{\left(1-t\right)^{2} \left(1-t^{2}\right)}$\\ 
    $\#10$ & $\frac{1+t+t^{2}+t^{3}}{(1-t)^3}$\\
    $\#21$ & $\frac{1+2 t+4 t^{2}+2 t^{3}+t^{4}}{(1-t)^2(1-t^2)}$\\
    $\#37$ & $\frac{1+2 t+5 t^{2}+5 t^{3}+5 t^{4}+2 t^{5}+t^{6}}{(1-t)^2(1-t^4)}$\\
    $\#39$ & $\frac{1+3 t+5 t^{2}+9 t^{3}+12 t^{4}+9 t^{5}+5 t^{6}+3 t^{7}+t^{8}}{(1-t)(1-t^3)(1-t^4)}$ \\
    $\#52$ & $\frac{1+3 t+5 t^{2}+8 t^{3}+10 t^{4}+8 t^{5}+5 t^{6}+3 t^{7}+t^{8}}{(1-t)(1-t^3)(1-t^4)}$\\
    $\#56$ & $\frac{1+2 t+5 t^{2}+6 t^{3}+5 t^{4}+2 t^{5}+t^{6}}{(1-t)^2(1-t^4)}$\\
    $\#60$ & $\frac{1+3 t+7 t^{2}+11 t^{3}+11 t^{4}+7 t^{5}+3 t^{6}+t^{7}}{(1-t)(1-t^3)^2}$\\
    $\#65$ & $\frac{1+2 t+2 t^{2}+3 t^{3}+3 t^{4}+2 t^{5}+2 t^{6}+t^{7}}{(1-t)^2(1-t^5)}$\\
    $\#67$ & $\frac{1+3 t+6 t^{2}+9 t^{3}+9 t^{4}+6 t^{5}+3 t^{6}+t^{7}}{(1-t)(1-t^3)^2}$\\
    $\#74$ & $\frac{1+3 t+6 t^{2}+10 t^{3}+14 t^{4}+18 t^{5}+18 t^{6}+14 t^{7}+10 t^{8}+6 t^{9}+3 t^{10}+t^{11}}{(1-t)(1-t^5)^2}$ \\
    $\#75$ & $\frac{1+3 t+6 t^{2}+9 t^{3}+9 t^{4}+6 t^{5}+3 t^{6}+t^{7}}{(1-t)(1-t^3)^2}$\\
    $\#76$ & $\frac{1+2 t+4 t^{2}+4 t^{3}+6 t^{4}+4 t^{5}+4 t^{6}+2 t^{7}+t^{8}}{(1-t)^2(1-t^6)}$\\
    $\#77$ & $\frac{1+2 t+2 t^{2}+3 t^{3}+3 t^{4}+2 t^{5}+2 t^{6}+t^{7}}{(1-t)^2(1-t^5)}$\\
    $\#80$ & $\frac{1+3 t+6 t^{2}+10 t^{3}+12 t^{4}+12 t^{5}+10 t^{6}+6 t^{7}+3 t^{8}+t^{9}}{(1-t)(1-t^3)(1-t^5)}$ \\
    $\#81$ & $\frac{1+3 t+7 t^{2}+12 t^{3}+14 t^{4}+15 t^{5}+15 t^{6}+14 t^{7}+12 t^{8}+7 t^{9}+3 t^{10}+t^{11}}{(1-t)(1-t^3)(1-t^7)}$\\
    $\#82$ & $\frac{1+2 t+3 t^{2}+5 t^{3}+5 t^{4}+3 t^{5}+2 t^{6}+t^{7}}{(1-t)^2(1-t^5)}$\\
    $\#86$ & $\frac{1+3 t+5 t^{2}+7 t^{3}+9 t^{4}+12 t^{5}+15 t^{6}+16 t^{7}+15 t^{8}+12 t^{9}+9 t^{10}+7 t^{11}+5 t^{12}+3 t^{13}+t^{14}}{(1-t)(1-t^6)(1-t^7)}$\\
    $\#87$ & $\frac{1+3 t+5 t^{2}+8 t^{3}+11 t^{4}+11 t^{5}+8 t^{6}+5 t^{7}+3 t^{8}+t^{9}}{(1-t)(1-t^3)(1-t^5)}$\\
    $\#88$ & $\frac{1+3 t+7 t^{2}+11 t^{3}+15 t^{4}+20 t^{5}+20 t^{6}+15 t^{7}+11 t^{8}+7 t^{9}+3 t^{10}+t^{11}}{(1-t)(1-t^5)^2}$\\
    $\#89$ & $\frac{1+2 t+4 t^{2}+3 t^{3}+4 t^{4}+4 t^{5}+6 t^{6}+4 t^{7}+4 t^{8}+3 t^{9}+4 t^{10}+2 t^{11}+t^{12}}{(1-t)^2(1-t^{10})}$\\
    $\#613$ & $\frac{1+3 t+6 t^{2}+9 t^{3}+9 t^{4}+6 t^{5}+3 t^{6}+t^{7}}{1-t-2 t^{3}+2 t^{4}+t^{6}-t^{7}}$\\
    \hline
  \end{tabular}

\vspace{3mm}

\caption{Growth series of the $22$ carbon allotropes.}
\label{table:22series}
\end{table}
\renewcommand{\arraystretch}{1}

\begin{figure}[h]
\begin{minipage}[t]{0.49\hsize}
\centering
\includegraphics[height=\zua]{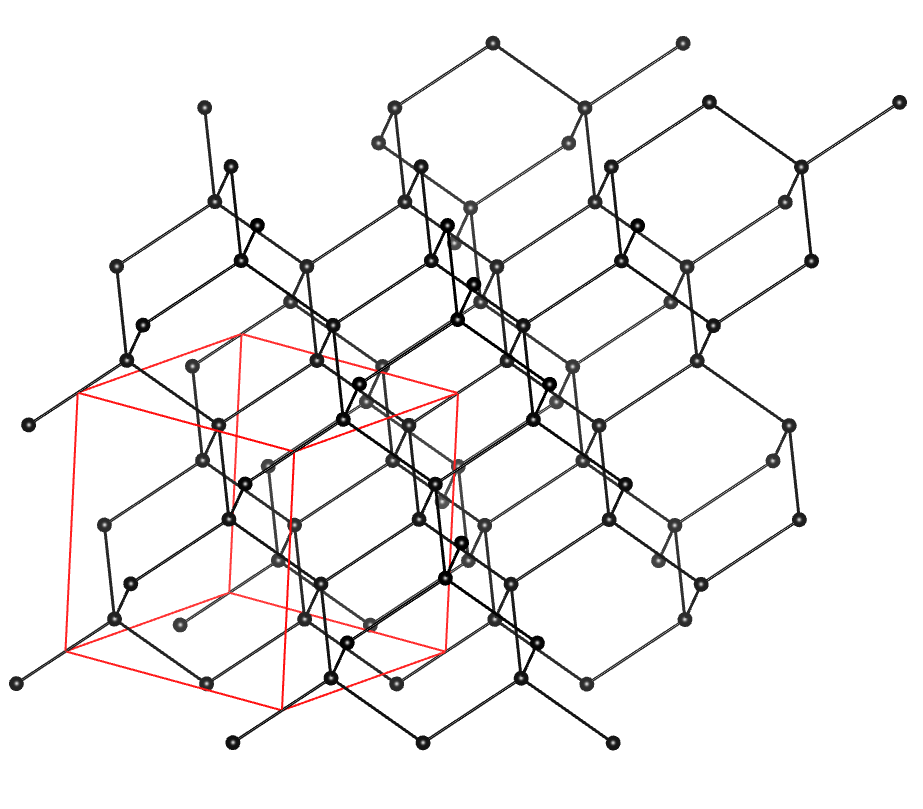}
\end{minipage}
\begin{minipage}[t]{0.49\hsize}
\centering
\includegraphics[height=\zua]{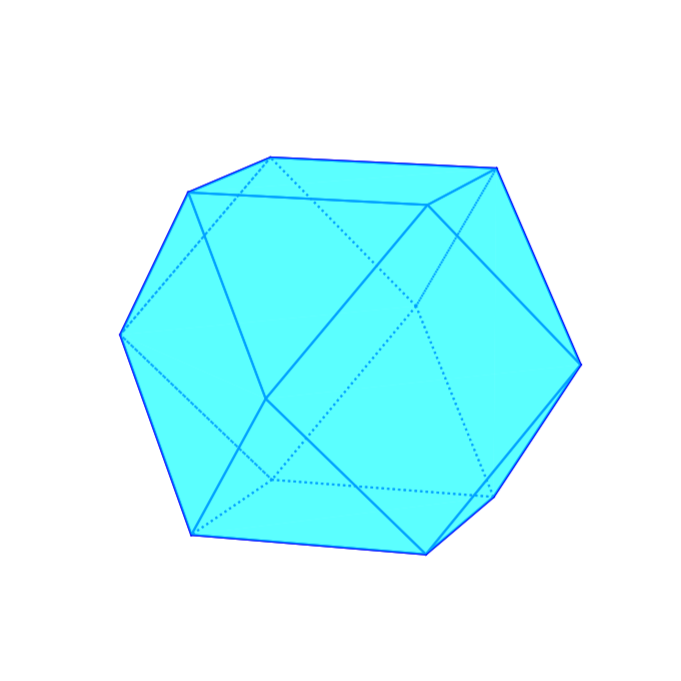}
\end{minipage}

\vspace{\zub}

\caption{The diamond crystal structure ({\sf SACADA} \#1) and its $P_{\Gamma}$.}
\label{fig:sacada1}

\vspace{6mm}

\begin{minipage}[t]{0.49\hsize}
\centering
\includegraphics[height=\zua]{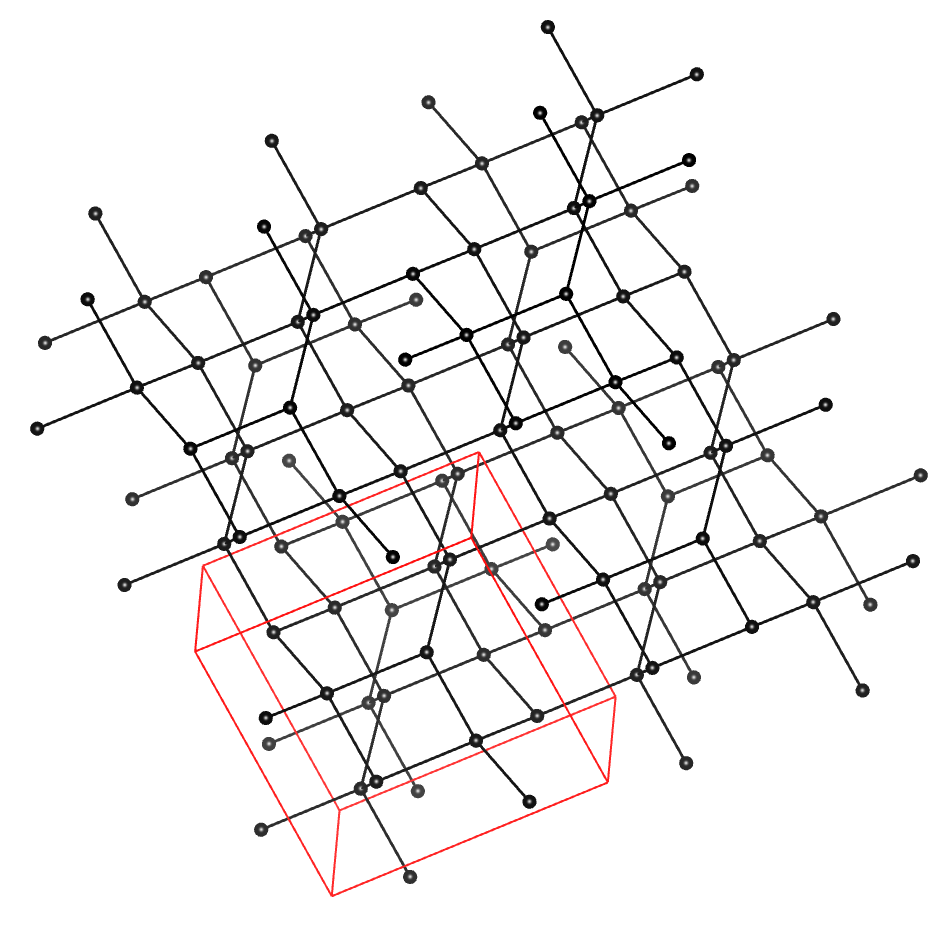}
\end{minipage}
\begin{minipage}[t]{0.49\hsize}
\centering
\includegraphics[height=\zua]{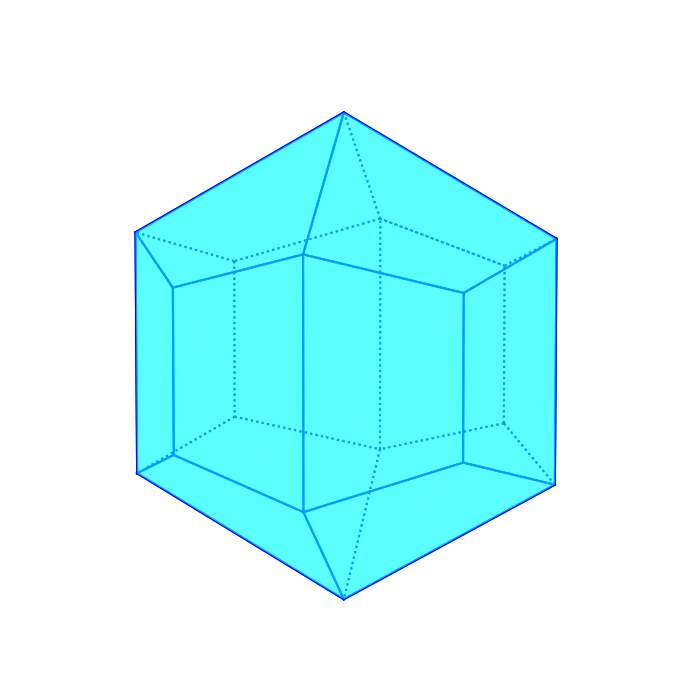}
\end{minipage}

\vspace{\zub}

\caption{A carbon allotrope {\sf SACADA} \#60 and its $P_{\Gamma}$.}
\label{fig:sacada60}
\end{figure}

\appendix

\section{Properties of the growth polytope}\label{section:GP}

In this section, we explain some properties of the growth polytope defined in Subsection \ref{subsection:GP}. 

\begin{lem}[{\cite{Fri13}*{Proposition 21}}]\label{lem:int}
If $(\Gamma, L)$ is a strongly connected $n$-dimensional periodic graph, then 
we have $0 \in \operatorname{int}(P _{\Gamma})$. 
In particular, we have $d_{P _{\Gamma}}(y) < \infty$ for any $y \in L_{\mathbb{R}}$. 
\end{lem}
\begin{proof}
It is sufficient to prove the following claim: 
\begin{itemize}
\item For any $u \in L$, there exists $m \in \mathbb{Z}_{>0}$ such that $u \in m P _{\Gamma}$. 
\end{itemize}
We fix $x_0 \in V_{\Gamma}$. 
Since the graph $\Gamma$ is strongly connected by assumption, 
there exists a walk $p$ in $\Gamma$ from $x_0$ to $u + x_0$. 
Then, by applying Lemma \ref{lem:bunkai}(1) to $\overline{p}$ (cf.\ Remark \ref{rmk:bunkai}(1)), 
there exists a sequence $q_1, \ldots, q_{\ell} \in \operatorname{Cyc}_{\Gamma /L}$ satisfying
$\langle \overline{p} \rangle = \sum _{i=1} ^{\ell} \langle q_i \rangle$. 
Since $\frac{\mu (\langle q \rangle)}{w(q)} \in P _{\Gamma}$ holds for any $q \in \operatorname{Cyc}_{\Gamma /L}$, 
we have 
\[
u = \mu ( \langle \overline{p} \rangle ) = \sum _{i=1}^{\ell} \mu(\langle q_i \rangle)
\in \biggl( \sum _{i=1} ^{\ell} w(q_i) \biggr) P _{\Gamma}, 
\]
which completes the proof. 
\end{proof}

The following theorem shows that the growth polytope $P _{\Gamma}$ can describe 
the asymptotic behavior of the growth sequence. 
In the theorem, we fix a periodic realization $\Phi$, and $C' _1$ and $C' _2$ depend on the choice of $\Phi$. 
We should emphasize that the growth polytope $P_{\Gamma}$ itself does not depend on the choice of the realization. 

In \cite{SM20b}*{Theorem 1}, Shutov and Maleev explain that the following theorem is proved 
in the papers \cite{Zhu02} and \cite{MS11} (written in Russian). 
We also emphasize that Kotani and Sunada in \cite{KS06}, Fritz in \cite{Fri13}, and 
Akiyama, Caalim, Imai and Kaneko in \cite{ACIK} have similar results. 
We will give a proof since the proof is referred to the proof of Proposition \ref{prop:C2}. 

\begin{thm}\label{thm:asymp}
Let $(\Gamma , L)$ be a strongly connected $n$-dimensional periodic graph. 
Let $\Phi: V_{\Gamma} \to L_{\mathbb{R}}$ be a periodic realization, 
and let $x_0 \in V_{\Gamma}$. Then there exist $C' _1, C' _2 \in \mathbb{R}_{\ge 0}$ such that 
for any $y \in V _{\Gamma}$, we have 
\[
d_{P_{\Gamma}, \Phi} (x_0, y) - C' _1 \le d_{\Gamma}(x_0, y) \le d_{P_{\Gamma}, \Phi} (x_0,y) + C' _2. 
\]
\end{thm}
\begin{proof}
Set $c := \# (V_{\Gamma} /L)$ and 
\[
B'_{c-1} := \{ y \in V_{\Gamma} \mid \text{there exists a walk $p$ from $x_0$ to $y$ with $\operatorname{length}(p) \le c-1$}\}. 
\]
We define $C' _1 \in \mathbb{R}_{\ge 0}$ by
\[
C' _1 := \max _{y \in B'_{c-1}} \bigl( d_{P_{\Gamma}, \Phi}(x_0, y) - d_{\Gamma}(x_0, y)  \bigr). 
\]
Note that $C' _1$ exists by Lemma \ref{lem:int} and the fact that $B'_{c-1}$ is a finite set. 

Let $y \in V_{\Gamma}$. 
We prove $d_{P_{\Gamma}, \Phi}(x_0, y) - C' _1 \le d_{\Gamma}(x_0, y)$. 
Let $p$ be a walk in $\Gamma$ from $x_0$ to $y$ such that $w(p) = d := d_{\Gamma}(x_0, y)$. 
By applying Lemma \ref{lem:bunkai}(1) to $\overline{p}$, 
there exists a walkable sequence $(q_0, q_1, \ldots , q_{\ell})$ 
such that $\langle \overline{p} \rangle = \sum _{i=0} ^{\ell} \langle q_i \rangle$. 
Then, the following three conditions hold. 
\begin{itemize}
\item 
$\operatorname{length}(q_0) \le c-1$. 

\item 
$\Phi(y) - \Phi(x_0) = \mu _{\Phi} (\langle \overline{p} \rangle) = 
\sum _{i=0}^{\ell} \mu _{\Phi} (\langle q_i \rangle)$. 

\item 
$d = w(p) = 
\sum _{i=0}^{\ell} w(q_i)$. 
\end{itemize}
Here, the first condition follows from the fact that $q_0$ is a path. 
Since we have $\mu (\langle q \rangle) \in w(q) \cdot P_{\Gamma}$ for each $q \in \operatorname{Cyc}_{\Gamma /L}$, 
we have 
\[
\sum _{i=1}^{\ell} \mu (\langle q_i \rangle) \in 
\left( \sum _{i=1}^{\ell} w(q_i) \right) P_{\Gamma}. 
\]
Let $p_0$ be the unique lift of $q_0$ with initial point $x_0$. 
Then, by the choice of $C' _1$, we have 
\[
d_{P_{\Gamma}, \Phi} (x_0, t(p_0)) - d_\Gamma \bigl(x_0,t(p_0) \bigr) \leq C' _1,
\]
and hence, 
\[
\mu _{\Phi} (\langle q_0 \rangle) = \operatorname{vec} _{\Phi} (p_0) = \Phi(t(p_0)) - \Phi(x_0) 
\in \bigl( C' _1 + w(p_0) \bigr) P_{\Gamma}. 
\]
Therefore, we have
\begin{align*}
\Phi(y) - \Phi(x_0) 
&= \sum _{i=0} ^{\ell} \mu _{\Phi} (\langle q_i \rangle) \\
& \in \left( C' _1 + w(q_0) +
\sum _{i=1} ^{\ell} w(q_i) \right) P_{\Gamma}\\
& = (C' _1 + d)P_{\Gamma}, 
\end{align*}
and hence, we have $d_{P_{\Gamma}, \Phi}(x_0, y) \le C' _1 + d$. 

Next, we define $C' _2 \in \mathbb{R}_{\ge 0}$ as follows. 
\begin{itemize}
\item 
First, we define $d_v := \min _{q \in \nu ^{-1}(v)} w(q)$ for each $v \in V(P_{\Gamma})$. 

\item
For $y \in V_{\Gamma}$, we define $d'(x_0, y)$ as the smallest weight $w(p)$ of 
a walk $p$ from $x_0$ to $y$ satisfying $\operatorname{supp}(\overline{p}) = V_{\Gamma} /L$.
We have $d'(x_0, y) < \infty$ since $\Gamma$ is assumed to be strongly connected. 

\item
For each $\sigma \in \operatorname{Facet}(P_{\Gamma})$, 
we fix a triangulation $T_{\sigma}$ of $\sigma$ such that $V(\Delta) \subset V(\sigma)$ holds for any $\Delta \in T_{\sigma}$. 

\item
We define a bounded set $Q \subset L_{\mathbb{R}}$ as follows: 
\[
Q := 
\bigcup _{\substack{\sigma \in \operatorname{Facet}(P_{\Gamma}),\\ 
\Delta \in T_{\sigma}}}
\left( \sum _{v \in V(\Delta)}  [0,1) d_v v \right)
\subset L_{\mathbb{R}}. 
\]

\item
Then, we set 
\[
C' _2 := \max \bigl \{ d'(x_0, y) - d_{P_{\Gamma}, \Phi} ( x_0, y ) \ \big | \ 
y \in V_{\Gamma}, \ \Phi(y) - \Phi(x_0) \in Q \bigr \}. 
\]
\end{itemize}
Such $C' _2$ exists since the set
\[
\bigl \{ y \in V_{\Gamma} \ \big | \ \Phi(y) - \Phi(x_0) \in Q \bigr \}
\]
is a finite set. 

Let $y \in V_{\Gamma}$. 
We prove $d_{\Gamma}(x_0, y) \le d_{P_{\Gamma}, \Phi} (x_0, y) + C' _2$. 
By Lemma \ref{lem:int}, there exist $\sigma \in \operatorname{Facet}(P_{\Gamma})$ and $\Delta \in T_{\sigma}$
such that $\Phi(y) - \Phi(x_0) \in \mathbb{R}_{\ge 0} \Delta$. 
Then we can uniquely write 
\[
\Phi(y) - \Phi(x_0)  = \sum _{v \in V(\Delta)} b_v d_v v
\]
with $b_v \in \mathbb{R}_{\ge 0}$. 
Then, we have 
\[
d_{P_{\Gamma}, \Phi} (x_0, y) = \sum _{v \in V(\Delta)} b_v d_v. 
\]
We define $y' := - \bigl( \sum _{v \in V(\Delta)} \lfloor b_v \rfloor d_v v \bigr) + y$. 
Here, we have $d_v v \in L$ by the choice of $d_v$. 
Then, $y'$ satisfies
\begin{align*}
\Phi(y') - \Phi(x_0)
&= \Phi(y) - \Phi(x_0) - \sum _{v \in V(\Delta)} \lfloor b_v \rfloor d_v v \\
&= \sum _{v \in V(\Delta)} (b_v - \lfloor b_v \rfloor) d_v v \\
& \in \sum_{v \in V(\Delta)} [0,1) d_v v
\subset Q. 
\end{align*}
By the choice of $C' _2$, 
there exists a walk $q$ in $\Gamma$ from $x_0$ to $y'$ satisfying 
\[
\operatorname{supp}(\overline{q}) = V_{\Gamma /L}, \qquad 
w(q) \le C' _2 + d_{P_{\Gamma}, \Phi} (x_0, y'). 
\]
For each $v \in V(\Delta)$, we can take $q_v \in \operatorname{Cyc}_{\Gamma /L}$ such that 
\[
w(q_v) = d_v, \qquad 
\mu(\langle q_v \rangle) = d_v v. 
\]
Since we have $\operatorname{supp}(\overline{q}) = V_{\Gamma /L}$, and $q_v$'s are closed walks, 
there exists a walk $r'$ in $\Gamma /L$ such that 
\[
\langle r' \rangle = \langle \overline{q} \rangle + \sum _{v \in V(\Delta)} \lfloor b_v \rfloor \langle q_v \rangle. 
\]
Let $r$ be the unique lift of $r'$ with initial point $x_0$. 
Then, we have
\[
t(r) 
= \left( \sum _{v \in V(\Delta)} \lfloor b_v \rfloor d_v v \right)+ t(q) 
= \left( \sum _{v \in V(\Delta)} \lfloor b_v \rfloor d_v v \right)+ y' 
= y. 
\]
Furthermore, its weight satisfies
\begin{align*}
w(r) 
&= w(q) + \sum _{v \in V(\Delta)} \lfloor b_v \rfloor w(q_v) \\
&\le C' _2 + d_{P_{\Gamma}, \Phi}(x_0,y') + \sum _{v \in V(\Delta)} \lfloor b_v \rfloor d_v \\
&= C' _2 + \sum _{v \in V(\Delta)} (b_v - \lfloor b_v \rfloor) d_v 
    + \sum _{v \in V(\Delta)} \lfloor b_v \rfloor d_v \\
&= C' _2 + d_{P_{\Gamma}, \Phi} (x_0,y).
\end{align*}
Therefore, we have $d_{\Gamma}(x_0, y) \le C' _2 + d_{P_{\Gamma}, \Phi} (x_0, y)$. 
\end{proof}

\begin{ex}\label{ex:C}
For the Wakatsuki graph with the injective periodic realization $\Phi$ in Example \ref{ex:WG} and the start point $x_0 = v'_2$, 
we shall compute the values $C' _1$ and $C' _2$. 
We identify $V_{\Gamma}$ with the subset of $L_{\mathbb{R}}$. 

First, $C' _1$ is given by 
\[
C' _1 := \max _{y \in B'_{c-1}} \bigl( d_{P_{\Gamma}, \Phi}(x_0, y) - d_{\Gamma}(x_0, y)  \bigr) = 1. 
\]
Note that we have $c = \# (V_{\Gamma} /L) = 3$ in this case, 
and $B'_2$ consists of seven vertices $x_0, x_1, \ldots, x_6$ shown in Figure \ref{fig:C1}. 
The values $d_{P_{\Gamma}, \Phi}(x_0, x_i)$ and $d_{\Gamma}(x_0, x_i)$ for $i = 0, 1 , \ldots , 6$ are as in Table \ref{table:C1}. 

\begin{figure}[htbp]
\centering
\includegraphics[height=6cm]{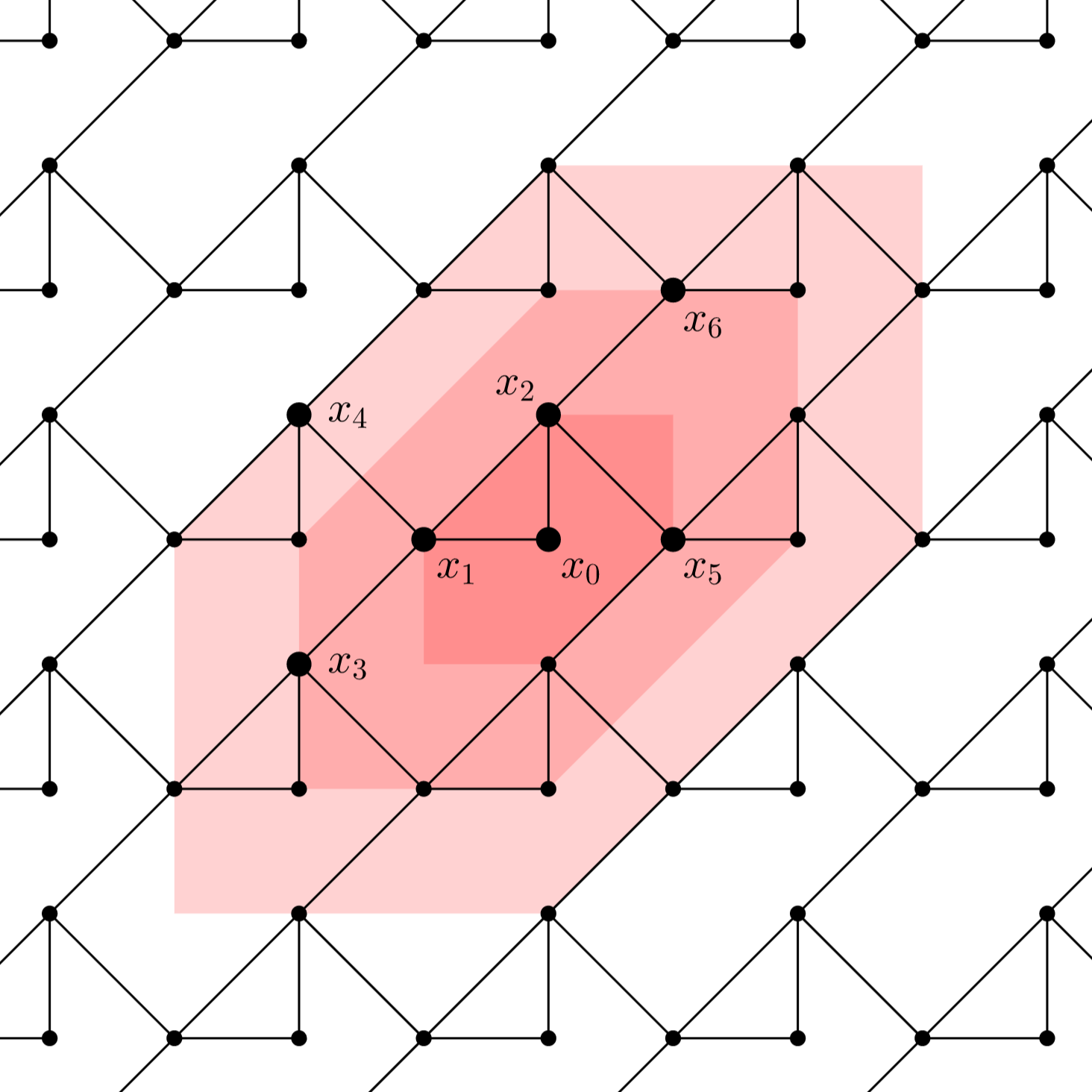}

\vspace{\zub}

\caption{$x_0, x_1, \ldots, x_6$, and $x_0 + i P_{\Gamma}$ for $i = 1, 2, 3$.}
\label{fig:C1}
\end{figure}

\begin{table}[h]
  \centering
  \begin{tabular}{cccccccc}
    \hline
    $i$ & $0$ & $1$ & $2$ & $3$ & $4$ & $5$ & $6$   \\
    \hline \hline
    $d_{P_{\Gamma}, \Phi}(x_0, x_i)$ & $0$ & $1$ & $1$ & $2$ & $3$ & $1$ & $2$ \\
    $d_{\Gamma}(x_0, x_i)$ & $0$ & $1$ & $1$ & $2$ & $2$ & $2$ & $2$ \\
    \hline
  \end{tabular}

\vspace{3mm}

\caption{$d_{P_{\Gamma}, \Phi}(x_0, x_i)$ and $d_{\Gamma}(x_0, x_i)$ for $i = 0, 1 , \ldots , 6$.}
\label{table:C1}
\end{table}

Next, $C' _2$ is given by
\[
C' _2 = \max _{y \in \operatorname{int}(Q) \cap V_{\Gamma}} \bigl( d'(x_0, y) - d_{P_{\Gamma}, \Phi}(x_0, y) \bigr) = 3, 
\]
where $Q$ is the hexagram-like figure illustrated as in Figure \ref{fig:C2}. 
Note that $\operatorname{int}(Q) \cap V_{\Gamma}$ consists of nine points $x_0, x_1, \ldots, x_8$. 
The values $d' (x_0, x_i)$ and $d_{P_{\Gamma}, \Phi}(x_0, x_i)$ for $i = 0, 1 , \ldots , 8$ are as in Table \ref{table:C2}. 

\begin{figure}[htbp]
\centering
\includegraphics[height=7cm]{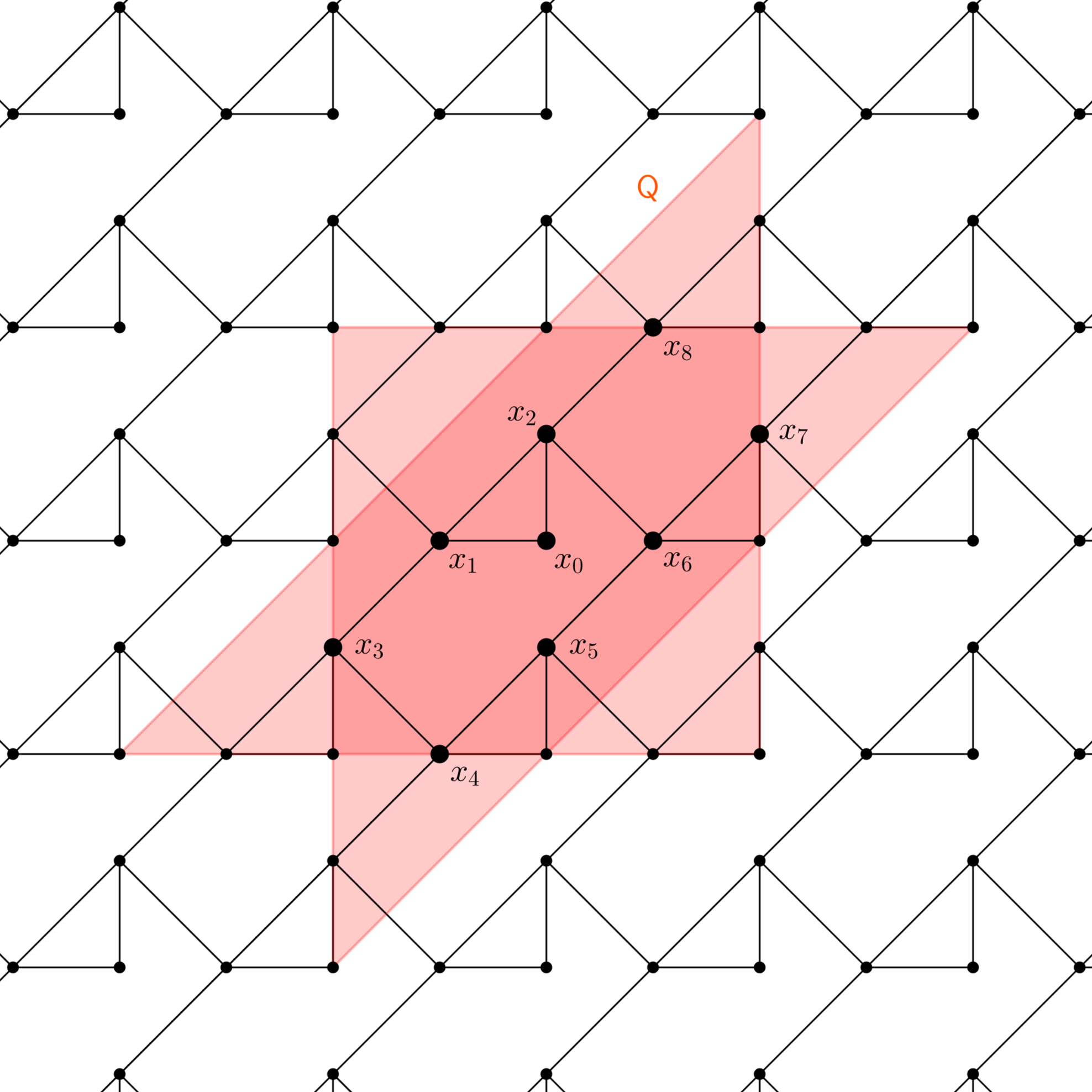}

\vspace{\zub}

\caption{$Q$ and $x_0, x_1, \ldots, x_8$.}
\label{fig:C2}
\end{figure}

\begin{table}[htbp]
  \centering
  \begin{tabular}{cccccccccc}
    \hline
    $i$ & $0$ & $1$ & $2$ & $3$ & $4$ & $5$ & $6$ & $7$ & $8$  \\
    \hline \hline
    $d' (x_0, x_i)$ & $3$ & $2$ & $2$ & $2$ & $3$ & $3$ & $2$ & $3$ & $2$ \\
    $d_{P_{\Gamma}, \Phi}(x_0, x_i)$ & $0$ & $1$ & $1$ & $2$ & $2$ & $1$ & $1$ & $2$ & $2$ \\
    \hline
  \end{tabular}

\vspace{3mm}

\caption{$d' (x_0, x_i)$ and $d_{P_{\Gamma}, \Phi}(x_0, x_i)$ for $i = 0, 1 , \ldots , 8$.}
\label{table:C2}
\end{table}
\end{ex}

The following corollary is proved by Shutov and Maleev \cite{SM20b}, and this is an immediate consequence of Theorem \ref{thm:asymp}. 
For a lattice $L \simeq \mathbb{Z}^n$ and a polytope $P \subset L_{\mathbb{R}}$, 
we can define the volume $\operatorname{vol}_L(P)$ of $P$ as follows. 
We fix a group isomorphism $i : L \xrightarrow{\simeq} \mathbb{Z}^n$. 
Then, $i$ is extended to an isomorphim $i_{\mathbb{R}}: L_{\mathbb{R}} \xrightarrow{\simeq} \mathbb{R}^n$, and 
we define $\operatorname{vol}_L(P)$ as the volume of $i_{\mathbb{R}} (P) \subset \mathbb{R}^n$. 
Note that the value $\operatorname{vol}_L(P)$ is independent of the choice of $i$. 

\begin{cor}[\cite{SM20b}*{Theorem 2}]\label{cor:leadcoef}
Let $(\Gamma,L)$ be a strongly connected $n$-dimensional periodic graph, and let $x_0 \in V_{\Gamma}$.
Let $f_b$ and $f_s$ be the quasi-polynomials corresponding to the functions 
$b: i \mapsto b_{\Gamma, x_0, i}$ and $s: i \mapsto s_{\Gamma, x_0, i}$ (Theorem \ref{thm:NSMN}). 
Then the following assertions hold. 
\begin{enumerate}
\item 
Each constituent of $f_b$ is a polynomial of degree $n$, and its leading coefficient is 
$\#(V_\Gamma/L) \cdot \operatorname{Vol}_L(P_\Gamma)$. 

\item 
Let $Q_0, \ldots , Q_{N-1}$ be the constituents of $f_s$. 
Then, for each $i = 0, \ldots, N-1$, we have $\deg Q_i \le n-1$. 
Furthermore, let $a_i$ denote the coefficient of $Q_i$ of degree $n-1$. 
Then, we have $\frac{1}{N}\sum _{i = 0}^{N-1} a_i = n \cdot \#(V_\Gamma/L) \cdot \operatorname{Vol}_L(P_\Gamma)$. 
\end{enumerate}
\end{cor}
\begin{proof}
We fix an injective periodic realization $\Phi$ satisfying $\Phi(x_0) = 0$. 
By $\Phi$, we identify $V_{\Gamma}$ with the subset of $L_{\mathbb{R}}$. 
We take $C' _1$ and $C' _2$ to satisfy Theorem \ref{thm:asymp}. 
Then, for each $d \in \mathbb{Z}_{\ge 0}$, we have 
\[
\#((d-C' _2)P_\Gamma \cap V_\Gamma) \le b_{\Gamma, x_0, d} \le \#((d+C' _1)P_\Gamma \cap V_\Gamma). 
\]
Therefore, (1) follows from the following fact (cf.\ \cite{BR}*{Lemma 3.19}): 
\[
\#(V_\Gamma/L) \cdot \operatorname{Vol}_L (P_\Gamma) 
= \lim _{d \to \infty} \frac{1}{d^n} \cdot \# \bigl( d P_{\Gamma} \cap V_{\Gamma} \bigr). 
\]
(2) follows from (1). 
\end{proof}

\begin{rmk}
In crystallography, 
the invariant $\frac{1}{N}\sum _{i = 0}^{N-1} a_i$ in Corollary \ref{cor:leadcoef}(2) 
is called the \textit{topological density} (cf.\ \cite{GKBS96}). 
\end{rmk}

\begin{ex}
For the Wakatsuki graph (see Example \ref{ex:WG}), 
the invariant $\frac{1}{N}\sum _{i = 0}^{N-1} a_i$ in Corollary \ref{cor:leadcoef}(2) for the start point $x_0 = v'_0$ is given by 
$\frac{1}{2}\left( \frac{9}{2} + \frac{9}{2} \right) = \frac{9}{2}$ (see Example \ref{ex:CS}). 
The same invariant for the start point $x_0 = v'_2$ is also given by $\frac{1}{2}(3 + 6) = \frac{9}{2}$. 
They are actually equal to $2 \cdot \#(V_\Gamma/L) \cdot \operatorname{Vol}_L (P_\Gamma) = 2 \cdot 3 \cdot \frac{3}{4} = \frac{9}{2}$ 
as proved in Corollary \ref{cor:leadcoef}(2). 
\end{ex}

\section{Ehrhart theory}\label{section:ET}
In this section, we discuss a variant of Ehrhart theory (Theorem \ref{thm:Eh}), which is necessary for the proof of Theorem \ref{thm:ET}. The difference from the usual Ehrhart theory is that the center $v$ of the dilation need not be the origin, and the dilation factor may be shifted by a constant $\alpha$. 

\begin{defi}
Let $P \subset \mathbb{R}^N$ be a rational polytope, and let $v \in \mathbb{R}^N$ and $\alpha \in \mathbb{R}$. 

\begin{enumerate}
\item
We define a function $h_{P, v, \alpha}: \mathbb{Z} \to \mathbb{Z}$ as follows
\[
h_{P, v, \alpha}(d) := \# \left( (v + (d + \alpha) P) \cap \mathbb{Z}^N \right). 
\]
Here, we define $t P = \emptyset$ for $t < 0$. 

\item
Let $\operatorname{relint}(P)$ denote the relative interior of $P$. 
Then, we also define a function $\overset{\circ}{h} _{P, v, \alpha}: \mathbb{Z} \to \mathbb{Z}$ as follows
\[
\overset{\circ}{h}_{P, v, \alpha}(d) := \# \left( (v + (d + \alpha) \cdot \operatorname{relint}({P})) \cap \mathbb{Z}^N \right). 
\]
Here, we define $t \cdot \operatorname{relint}({P}) = \emptyset$ for $t \le 0$. 

\item
Let  
\[
H_{P, v, \alpha} (t) := \sum _{i \in \mathbb{Z}} h_{P, v, \alpha}(i) t^i, \qquad
\overset{\circ}{H}_{P, v, \alpha} (t) := \sum _{i \in \mathbb{Z}} \overset{\circ}{h}_{P, v, \alpha}(i) t^i
\]
denote their generating functions. 
\end{enumerate}
\end{defi}

\begin{rmk}  \hfill
\begin{enumerate}
\item
In the definition above, we have defined $0 \cdot P = \{ 0 \}$ but $0 \cdot \operatorname{relint}({P}) = \emptyset$. This definition is necessary for the equations ($\spadesuit$) in the proof of Theorem \ref{thm:Eh}.

\item
The usual Ehrhart theory (cf.\ \cite{BR}) treats the case where $v = 0$ and $\alpha = 0$. 
McMullen (in \cite{McM78}) and de Vries and Yoshinaga (in \cite{dVY21}*{Section 3}) discuss $h_{P, v, \alpha}$ when $\alpha = 0$. 
\end{enumerate}
\end{rmk}

\begin{lem}\label{lem:sym}
We have $h_{P,v,\alpha} = h_{-P,-v,\alpha}$ and $\overset{\circ}{h}_{P,v,\alpha} = \overset{\circ}{h}_{-P,-v,\alpha}$. 
\end{lem}
\begin{proof}
The first assertion follows from
\begin{align*}
\# \left( (v + (d + \alpha) P) \cap \mathbb{Z}^N \right) 
&= \# \left( (- (v + (d + \alpha) P)) \cap \mathbb{Z}^N \right) \\
&= \# \left( (-v + (d + \alpha) (-P) \cap \mathbb{Z}^N \right). 
\end{align*}
The second assertion can be proved in the same way.
\end{proof}

\begin{thm}\label{thm:Eh}
Let $P \subset \mathbb{R}^N$ be a rational polytope of dimension $M$, and let $v \in \mathbb{R}^N$ and $\alpha \in \mathbb{R}$. 
Then, the following assertions hold: 
\begin{enumerate}
\item $h_{P, v, \alpha}$ is a quasi-polynomial on $d \ge - \alpha$, and 
$h_{P, v, \alpha}(d) = 0$ holds for $d < - \alpha$. 

\item $\overset{\circ}{h} _{P, v, \alpha}$ is a quasi-polynomial on $d > - \alpha$, and 
$\overset{\circ}{h} _{P, v, \alpha}(d) = 0$ holds for $d \le - \alpha$. 

\item $H_{P, v, \alpha}(1/t) = (-1)^{M+1} \overset{\circ}{H}_{P, -v, - \alpha}(t)$. 

\item 
Let $f_{P,v,\alpha}$ be the corresponding quasi-polynomial to the function $h_{P, v, \alpha}$ on $d \ge - \alpha$. 
Then, we have 
$f_{P,v,\alpha}(-i) = (-1)^M \overset{\circ}{h}_{P, -v, - \alpha}(i)$ for any $i \in \mathbb{Z}_{> \alpha}$. 
\end{enumerate}
\end{thm}
\begin{proof}
For a subset $S \subset \mathbb{R}^{N+1}$, let 
\[
\sigma _S({\bf z}) = \sigma _S (z_1, \ldots, z_{N+1}) := \sum _{{\bf m} \in S \cap \mathbb{Z}^{N+1}} {\bf z}^{\bf m}
\]
denote the integer-point transform of $S$ (see \cite{BR}*{Section 3.3}). 
This is a formal sum of Laurent monomials ${\bf z}^{\bf m} = z_1^{m_1} \cdots z_{N+1}^{m_{N+1}}$ 
for ${\bf m} = (m_1, \ldots , m_{N+1}) \in S \cap \mathbb{Z}^{N+1}$. 

Let $K := \mathbb{R}_{\ge 0} \bigl( \{ 1 \} \times P \bigr) \subset \mathbb{R}^{N+1}$ be the cone over $P$. 
Then, we have 
\begin{align*}
\begin{split}
H_{P,v,\alpha} (t) &= \sigma _{(- \alpha, v) + K} (t, 1 , \ldots, 1), \\
\overset{\circ}{H}_{P, -v, - \alpha}(t) &= \sigma _{(\alpha, - v) + \operatorname{relint}(K)} (t, 1 , \ldots, 1). 
\end{split} \tag{$\spadesuit$}
\end{align*}
Furthermore, we have 
\[
\sigma _{(- \alpha, v) + K} (z_1^{-1}, \ldots, z_{N+1}^{-1}) 
= (-1)^{M+1} \sigma _{(\alpha, -v) + \operatorname{relint}(K)} (z_1, \ldots, z_{N+1})
\]
by \cite{BR}*{Exercises 4.5 and 4.6}. Therefore, we get the desired equality in (3). 

The second assertions of (1) and (2) are obvious from the definition of $h_{P, v, \alpha}$ and $\overset{\circ}{h}_{P, v, \alpha}$. 

By taking a triangulation of $P$, 
the first assertions in (1) and (2) can be reduced to the same assertions for simplices $P$. 
Assume that $P$ is a simplex. Let $v_1, \ldots, v_{M+1}$ be the vertices of $P$. 
Since $P$ is rational polytope, we may write $v_i = u_i / a_i$ for some $u_i \in \mathbb{Z}^N$ and $a_i \in \mathbb{Z}_{>0}$. 
We set $D, D^{\circ} \subset \mathbb{R}^{N+1}$ by
\begin{align*}
D &:= \left\{ \sum _{i = 1}^{M+1}  \alpha _i (a_i, u_i) \ \middle | \ \alpha_1, \ldots, \alpha_{M+1} \in [0,1) \right\}, \\ 
D^{\circ} &:= \left\{ \sum _{i = 1}^{M+1} \alpha _i (a_i, u_i) \ \middle | \ \alpha_1, \ldots, \alpha_{M+1} \in (0,1] \right\}.
\end{align*}
Then, by \cite{BR}*{Theorem 3.5}, we have 
\begin{align*}
\sigma _{(- \alpha, v) + K} (z_1, {\bf z}) 
&= \frac{\sigma _{(- \alpha, v) + D} (z_1, {\bf z})}{\prod _{i=1} ^{M+1} (1-z_1^{a_i} {\bf z}^{u_i})}, \\
\sigma _{(- \alpha, v) + \operatorname{relint}(K)} (z_1, {\bf z}) 
&= \frac{\sigma _{(- \alpha, v) + D^{\circ}} (z_1, {\bf z})}{\prod _{i=1} ^{M+1} (1-z_1^{a_i} {\bf z}^{u_i})}, 
\end{align*}
where ${\bf z} = (z_2, \ldots, z_{N+1})$. 
Therefore, we have 
\begin{align*}
H_{P,v,\alpha} (t) 
&= \sigma _{(- \alpha, v) + K} (t, 1, \ldots , 1) 
= \frac{\sigma _{(- \alpha, v) + D} (t, 1, \ldots, 1)}{\prod _{i=1} ^{M+1} (1- t^{a_i})}, \\
\overset{\circ}{H}_{P,v,\alpha} (t) 
&= \sigma _{(- \alpha, v) + \operatorname{relint}(K)} (t, 1, \ldots , 1) 
= \frac{\sigma _{(- \alpha, v) + D^{\circ}} (t, 1, \ldots, 1)}{\prod _{i=1} ^{M+1} (1- t^{a_i})}. 
\end{align*}
Here, we have 
\begin{align*}
\deg \sigma _{(- \alpha, v) + D} (t, 1, \ldots, 1) 
&< - \alpha + \sum _{i=1} ^{M+1} a_i, \\
\deg \sigma _{(- \alpha, v) + D^{\circ}} (t, 1, \ldots, 1) 
&\le - \alpha + \sum _{i=1} ^{M+1} a_i. 
\end{align*}
Therefore, we can conclude that $h_{P, v, \alpha}$ is a quasi-polynomial on $d \ge - \alpha$, and 
that $\overset{\circ}{h}_{P, v, \alpha}$ is a quasi-polynomial on $d > - \alpha$. 
We complete the proof of (1) and (2). 

Finally, we prove (4). 
Since $f_{P,v,\alpha}$ is a quasi-polynomial, we have
\[
\sum _{i \in \mathbb{Z}_{< - \alpha}} {f_{P,v,\alpha}(i)} t^{i} 
= - \sum _{i \in \mathbb{Z}_{\ge - \alpha}} {f_{P,v,\alpha}(i)} t^i
\]
as rational functions (cf.\ \cite{BR}*{Exercise 4.7}). 
Therefore, we have 
\begin{align*}
\sum _{i \in \mathbb{Z}_{> \alpha}} {f_{P,v,\alpha}(- i)} t^{- i}
&= \sum _{i \in \mathbb{Z}_{< - \alpha}} {f_{P,v,\alpha}(i)} t^{i} \\ 
&=- \sum _{i \in \mathbb{Z}_{\ge - \alpha}} {f_{P,v,\alpha}(i)} t^i \\
&=- H_{P,v,\alpha}(t) \\
&= (-1)^{M} \overset{\circ}{H}_{P, -v, - \alpha}(t^{-1}) \\
&= (-1)^{M} \sum _{i \in \mathbb{Z}} \overset{\circ}{h}_{P, -v, - \alpha}(i) t^{-i}. 
\end{align*}
Here, the third equality follows from (1), and the fourth follows from (3). 
By comparing the coefficients, we conclude that $f_{P,v,\alpha}(- i) = (-1)^M \overset{\circ}{h}_{P, -v, - \alpha}(i)$ for $i \in \mathbb{Z}_{> \alpha}$. 
\end{proof}

\end{document}